\documentclass[11pt,preprint]{imsart}
\RequirePackage[OT1]{fontenc}
\RequirePackage{amsthm,amsmath}
\RequirePackage[numbers]{natbib}

\usepackage{graphicx,amsmath,amsthm,amsfonts,amssymb}
\usepackage{booktabs, rotating, float}
\usepackage[hmargin=3cm,vmargin=3cm]{geometry}
\usepackage{mathtools}
\usepackage{bbm}
\mathtoolsset{showonlyrefs}
\usepackage{enumerate}
\usepackage{color}
\usepackage{listings}
\usepackage{epstopdf}
\usepackage{makecell}
\usepackage{tikz}
\usetikzlibrary{decorations.pathreplacing}
\usepackage{array}
\usepackage{caption}
\usepackage{subcaption}
\usepackage{soul}
\usepackage{hyperref}
\usepackage{verbatim}
\usepackage{siunitx}
\usetikzlibrary{calc} \tikzset{>=latex}
\usetikzlibrary{calc,decorations.markings}
\usetikzlibrary{positioning,shapes,decorations.text}
\usetikzlibrary{decorations.pathmorphing}
\theoremstyle{plain}
\newtheorem{thm}{Theorem}[section]
\newtheorem{cor}[thm]{Corollary}

\newtheorem{prop}[thm]{Proposition}
\newtheorem{lemma}[thm]{Lemma}
\newtheorem{claim}[thm]{Claim}
\newtheorem{assumption}[thm]{Assumption}
\theoremstyle{definition}
\newtheorem{defn}[thm]{Definition}
\newtheorem{rem}[thm]{Remark}
\newtheorem{eg}[thm]{Example}
\newtheorem{fact}[thm]{Fact}
\newtheorem{observe}[thm]{Observation}

\numberwithin{equation}{section}

\newcommand{\rpm}{\sbox0{$1$}\sbox2{$\scriptstyle\pm$}
  \raise\dimexpr(\ht0-\ht2)/2\relax\box2 }

\newcommand{\R}{\mathbb{R}}

\newcommand{\argmax}{\operatornamewithlimits{\textrm{argmax}}}

\tikzstyle{nd} = [anchor=base, inner sep=0pt]
\tikzstyle{ndpic} = [remember picture, baseline, every node/.style={nd}]

\newcommand{\Real}{\mathbb{R}}

\pgfdeclaredecoration{jiggly}{step}
{
  \state{step}[width=+\pgfdecorationsegmentlength]
  {
    \pgfpathlineto{
      \pgfpointadd
      {\pgfpoint{\pgfdecorationsegmentlength}{0pt}}
      {\pgfpointpolar{90-\pgfdecoratedangle}
          {rand*\pgfdecorationsegmentamplitude}}
    }
  }
  \state{final}
  {
    \pgfpathlineto{\pgfpointdecoratedpathlast}
  }
}
\def\beq{\begin{equation}}
\def\eeq{\end{equation}}
\def\ba{\begin{enumerate}[(a)]}
\def\bei{\begin{enumerate}[(i)]}
\def\be{\begin{enumerate}[(1)]}
\def\ee{\end{enumerate}}
\def\bi{\begin{itemize}}
\def\ei{\end{itemize}}
\def\beg{\begin{eg}}
\def\eeg{\end{eg}}
\def\bd{\begin{defn}}
\def\ed{\end{defn}}
\def\bt{\begin{thm}}
\def\et{\end{thm}}
\def\bl{\begin{lemma}}
\def\el{\end{lemma}}
\def\bfac{\begin{fact}}
\def\efac{\end{fact}}

\def\bc{\begin{cor}}
\def\ec{\end{cor}}
\def\bp{\begin{prop}}
\def\ep{\end{prop}}
\def\bo{\begin{observe}}
\def\eo{\end{observe}}
\def\bas{\begin{assumption}}
\def\eas{\end{assumption}}
\def\RR{\mathbb{R}}

\def\ZZ{\mathbb{Z}}
\def\NN{\mathbb{N}}

\def\ii{\item}

\def\beg{\begin{eg}}
\def\eeg{\end{eg}}

\makeindex

\numberwithin{equation}{section}
\numberwithin{table}{section}

\startlocaldefs
\endlocaldefs

\begin{document}

\begin{frontmatter}

\title{KPZ equation tails for general initial data}
\runtitle{KPZ equation tails}
\runauthor{Corwin \& Ghosal}

\begin{aug}
  \author{\fnms{Ivan}  \snm{Corwin}\textsuperscript{1}\corref{}\ead[label=e1]{ivan.corwin@gmail.com}}
  \author{\fnms{Promit} \snm{Ghosal}\textsuperscript{2}\corref{}\ead[label=e2]{pg2475@columbia.edu}\vspace{0.1in}\\{Columbia University}}

\address{\textsuperscript{1}Department of Mathematics, 2990 Broadway, New York, NY 10027, USA\\ \printead{e1}}
\address{\textsuperscript{2}Department of Statistics, 1255 Amsterdam, New York, NY 10027, USA\\ \printead{e2}}




\end{aug}

\begin{abstract}
We consider the upper and lower tail probabilities for the centered (by time$/24$) and scaled (according to KPZ time$^{1/3}$ scaling) one-point distribution of the Cole-Hopf solution of the KPZ equation when started with initial data drawn from a very general class. For the lower tail, we prove an upper bound which demonstrates a crossover from super-exponential decay with exponent $3$ in the shallow tail to an exponent $5/2$ in the deep tail. For the upper tail, we prove  super-exponential decay bounds with exponent $3/2$ at all depths in the tail.
\end{abstract}




\end{frontmatter}









\section{Introduction}\label{introduction}

In this paper we consider the following question: How does the initial data for an SPDE affect the statistics of the solution at a later time? Namely, we consider the Kardar-Parisi-Zhang (KPZ) equation (or equivalently, the stochastic heat equation (SHE)) and probe the lower and upper tails of the centered (by time$/24$) and scaled (by time$^{1/3}$) one-point distribution for the solution at finite and long times.
Our main results (Theorems~\ref{Main1Theorem} and \ref{Main4Theorem}) show that within a very large class of initial data, the tail behavior for the KPZ equation does not change in terms of the super-exponential decay rates and at most changes in terms of the coefficient in the exponential. These results are the first tail bounds for general initial data which capture the correct decay exponents and which respect the long-time scaling behavior of the solution.

In order to state our results, let us recall the KPZ equation, which is formally written as
\begin{align}
\partial_T \mathcal{H}(T,X) &= \frac{1}{2}\partial^2_X \mathcal{H}(T, X) + \frac{1}{2}(\partial_X \mathcal{H}(T,X))^2+ \xi(T, X), \qquad
\mathcal{H}(0,X) = \mathcal{H}_0(X).
\end{align}
 Here, $\xi$ is the space-time white noise, whose presence (along with the non-linearity) renders this equation ill-posed.
A proper definition of the solution of the KPZ equation comes from the Cole-Hopf transform by which we {\em define}
\begin{align}\label{eq:ColeHpf}
\mathcal{H}(T,X) := \log \mathcal{Z}(T,X)
\end{align}
where $\mathcal{Z}(T,X)$ is the unique solution of the well-posed SHE
\begin{align}\label{eq:SHEDef}
\partial_T \mathcal{Z}(T,X) = \frac{1}{2} \partial^2_X \mathcal{Z}(T,X) + \mathcal{Z}(T,X) \xi(T,X), \quad \mathcal{Z}_0(X) = e^{\mathcal{H}_0(X)}.
\end{align}
Note that the logarithm in \eqref{eq:ColeHpf} is defined since $\mathcal{Z}(T,X)$ is almost-surely strictly positive for all $T>0$ and $X\in \R$ \cite{Mueller91}.  We refer to \cite{Quastel12, Corwin12, Hairer13} for more details about the KPZ equation and the SHE and their relation to random growth, interacting particle systems, directed polymers and other probabilistic systems (see also \cite{Mol96,Davar,Bertini1995,BarraquandCorwin15,Comets16}).

In this paper, we consider very general initial data as now describe.

\bd\label{Hypothesis}
Fix $\nu \in (0,1)$ and $C,\theta, \kappa, M>0$. A measurable function $f:\RR\to \RR\cup \{-\infty\}$ satisfies $\mathbf{Hyp}(C,\nu,\theta, \kappa, M)$ if:
   \be
\ii \mbox{}
\vskip-1.3cm
\begin{align}\label{eq:InMomBd}
 f(y) \leq C + \frac{\nu}{2^{2/3}} y^2, \quad \forall y\in \RR,
 \end{align}
 \ii  there exists a subinterval $\mathcal{I}\subset [-M,M]$ with $|\mathcal{I}|=\theta$ such that
 \begin{align}\label{eq:LowInitBd}
 f(y)\geq -\kappa, \quad  \forall y\in \mathcal{I}.
 \end{align}
\ee
\ed

For a measurable function $f:\RR\to \RR\cup \{-\infty\}$, and $T>0$ consider the solution to the KPZ equation with initial data $\mathcal{H}_0$ chosen such that
\begin{align}\label{eq:ScaledInitialData}
T^{-\frac{1}{3}}\mathcal{H}_0\big((2T)^{\frac{2}{3}}y\big) = f(y).
\end{align}
We consider the KPZ equation with this initial data and run until time\footnote{Notice that the initial data and time horizon are both dependent on $T$. This allows for a much wider class of initial data which are adapted to the KPZ fixed point scaling.} $T$. Namely, let
  \begin{align}\label{eq:ScalCentHeight}
h^f_T(y):=\frac{\mathcal{H}\big(2T, (2T)^{\frac{2}{3}}y\big)+\frac{T}{12} -\frac{2}{3}\log(2T)}{T^{\frac{1}{3}}}.
\end{align}

\begin{figure}[h]
\includegraphics[width=.45\linewidth]{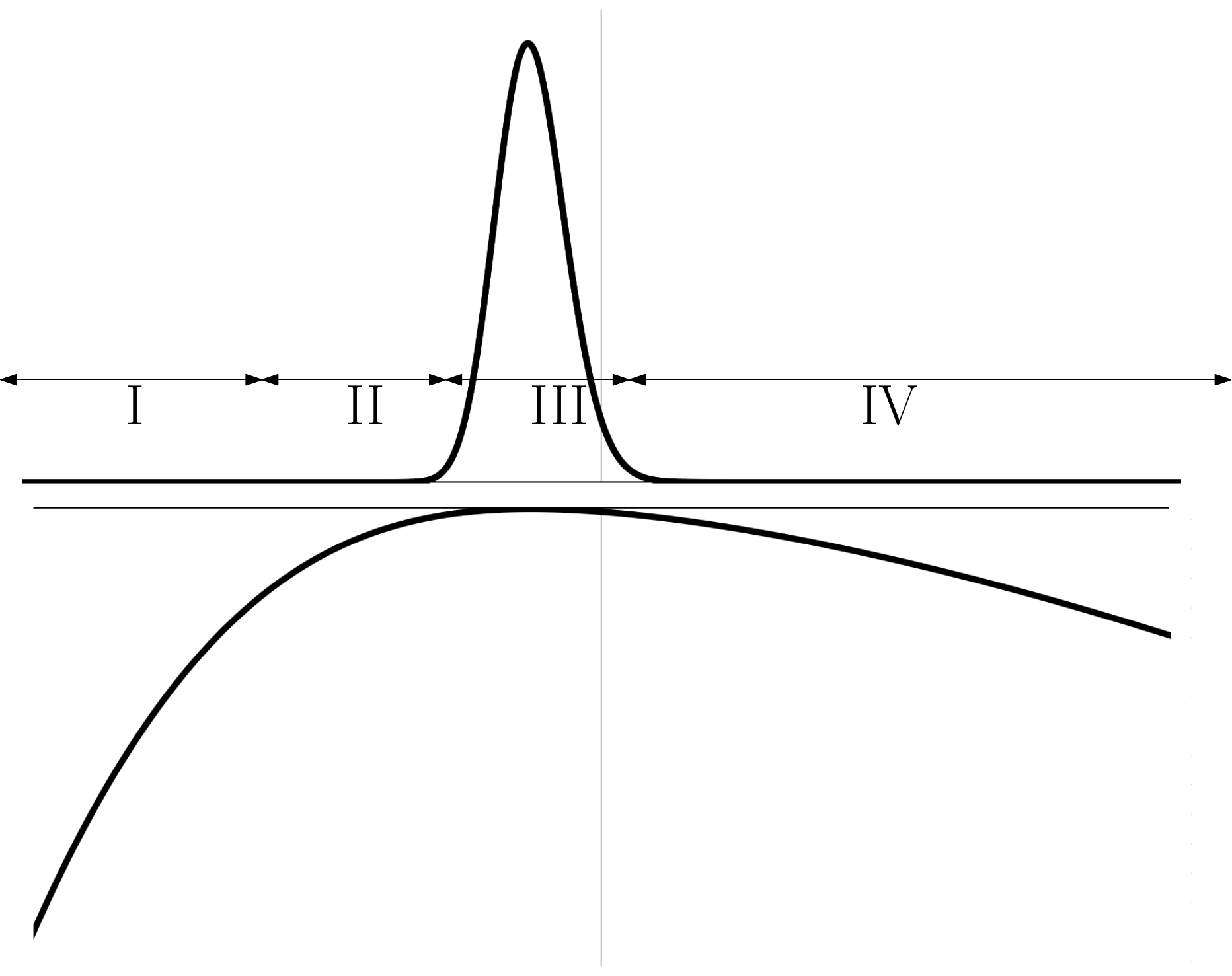}
\caption{Schematic plot of the density (top) and log density (bottom) of $h^f_T(0)$. Letting $s$ denote the horizontal axis variable, there are four regions which display different behaviors. Region $I$ (deep lower tail, when $s\ll -T^{2/3}$): the log density has power law decay with exponent $5/2$. Region  $II$ (shallow lower tail, when $-T^{2/3} \ll s\ll 0$): the log density has power law decay with exponent $3$. Region $III$ (center, when $s\approx 0$): the density depends on initial data as predicted by the KPZ fixed point. Region $IV$ (upper tail, when $s\gg 0$): the log density has power law decay with exponent $3/2$. The universality of the power law exponents (in regions $I$, $II$ and $IV$) for general initial data constitutes the main contribution of this paper.}
\label{fig:Figure1}
\end{figure}

Our first main result (Theorem~\ref{Main1Theorem}) provides an upper bound on the lower tail that holds uniformly over $f\in\mathbf{Hyp}(C,\nu, \theta, \kappa, M)$, and $T>1$. The proof of this and our other main results are deferred to the later sections of the paper.

\bt\label{Main1Theorem}
Fix any  $\epsilon, \delta\in (0,\frac{1}{3})$, $C, M,\theta>0$, $\nu\in (0,1)$, and $T_0>0$. There exist $s_0=s_0(\epsilon, \delta, C,M, \theta,\nu,T_0)$ and $K=K(\epsilon, \delta, T_0)>0$ such that for all $s\geq s_0$, $T\geq T_0$, and $f\in \mathbf{Hyp}(C,\nu, \theta, \kappa, M)$ (recall $h^f_T(y)$ is defined in \eqref{eq:ScaledInitialData} and \eqref{eq:ScalCentHeight}),
\begin{align}\label{eq:UpperBoundDetData}
\mathbb{P}\left(h^f_T(0)\leq -s\right)&\leq e^{- T^{1/3}\frac{4(1-\epsilon) s^{5/2}}{15\pi}}
+ e^{- Ks^{3-\delta} - \epsilon s T^{1/3}}+  e^{- \frac{(1-\epsilon)s^3}{12}}.
\end{align}
%
\et

\begin{rem}
There are three regions of the lower tail (see $I, II$, and $III$ in Figure \ref{fig:Figure1}). In each region (and for $T$ large) a different one of the three terms on the r.h.s. of \eqref{eq:UpperBoundDetData} becomes active. For instance, for region $I$ when $s\gg T^{2/3}$, the largest term in our bound is the first  term in the r.h.s. of \eqref{eq:UpperBoundDetData}.
Likewise, the middle term in the r.h.s. of \eqref{eq:UpperBoundDetData} is active in region $II$ and the last term in region $III$. We presently lack a matching lower bound for the lower tail probability. This is known for only the narrow wedge (see Proposition~\ref{NotMainTheorem}). See Section \ref{sec:previouswork} for some discussion regarding physics literature related to this tail.
Let us also note that one can get similar bound as in \eqref{eq:UpperBoundDetData} on $\mathbb{P}\big(h^{f}_T(y)\leq -s\big)$ when $y\neq 0$. This is explained in Section~\ref{sketch}. Finally, observe that two important choices of initial data --- narrow wedge and Brownian motion --- do not fit into this class\footnote{The flat initial data is in the class and arises from $f\equiv 0$.}. The narrow wedge result is in fact a building block for the proof of this result, while Brownian follows as a fairly easy corollary (see Section~\ref{NWandBM}).
\end{rem}

%
%
%


Our second main result pertains to the upper tail and shows upper and lower bounds which hold uniformly over  $f\in\mathbf{Hyp}(C,\nu,\theta,\kappa,M)$, and $T>\pi$.

 \bt\label{Main4Theorem}
Fix any $\nu \in (0,1)$ and $C,\theta, \kappa, M>0$.
For any $T_0>0$, there exist $s_0= s_0(C, \nu, \theta, \kappa, M, T_0)>0$, $c_1= c_1(T_0)>c_2=c_2(T_0)>0$ such that for all $s\geq s_0$, $T>T_0$ and  $f\in\mathbf{Hyp}(C,\nu, \theta, \kappa, M)$,
\begin{align}\label{eq:GenBdRough}
e^{-c_1s^{3/2}}\leq \mathbb{P}\big(h^{f}_{T}(0)\geq s\big)\leq e^{-c_2s^{3/2}}.
\end{align}

We may further specify values of $c_1$ and $c_2$ for which \eqref{eq:GenBdRough} holds, provided we assume $T_0>\pi$. In that case, for any $ \epsilon, \mu \in (0, \frac{1}{2})$, there exists $s_0=s_0(\epsilon, \mu, C, \nu, \theta, \kappa, M, T_0)>0$ such that for all $s\geq s^{\prime}_0$, $T\geq T_0$, and $f\in\mathbf{Hyp}(C,\nu, \theta, \kappa, M)$,
\eqref{eq:GenBdRough} holds with the following choices for $c_1>c_2$:
\bei
\ii If $s_0\leq s< \frac{1}{8}\epsilon^3(1-\frac{2\mu}{3})^{-1} T^{\frac{2}{3}}$ then we may take $c_1=\frac{8}{3}(1+\mu)(1+\epsilon)$ and $c_2= \frac{\sqrt{2}}{3}(1-\mu)(1-\epsilon)$.
\ii If $s\geq \max\{s_0, \frac{9}{16}\epsilon^{-2} (1-\frac{2\mu}{3})^{-1} T^{\frac{2}{3}} \}$ then we may take $c_1= 8\sqrt{3}(1+\mu)(1+\epsilon)$ and $c_2=\frac{\sqrt{2}}{3}(1-\mu)(1-\epsilon)$.
\ii If $\max\{s_0, \frac{1}{8}\epsilon^3(1-\frac{2\mu}{3})^{-1}  T^{\frac{2}{3}} \}\leq s\leq \max\{s_0, \frac{9}{16}\epsilon^{-2} (1-\frac{2\mu}{3})^{-1} T^{\frac{2}{3}} \}$ then we may take $c_1= 2^{9/2}\epsilon^{-3}(1+\mu)$ and $c_2= \frac{\sqrt{2}}{3}(1-\mu)\epsilon$.
\ee
 \et

%
%

\begin{rem}
 In Theorems~\ref{GrandUpTheorem} and \ref{Main6Theorem}, we prove similar results for narrow wedge and Brownian initial data. The upper and lower bounds on the constants $c_1$ and $c_2$ are not optimal. In fact, it is not clear to us how the initial data translates to the optimal value of $c_1$ or $c_2$. There, however, some predictions in the physics literature -- see Section \ref{sec:previouswork}. The condition $T_0>\pi$ assumed in the second part of Theorem \ref{Main4Theorem} could be replaced by an arbitrary lower bound, though the resulting conditions on $s$, $c_1$ and $c_2$  would need to change accordingly. This value $\pi$ turns out to work well in the computations leading to this result; in particular see \eqref{eq:IntegralBd1}.
\end{rem}

\subsection{Proof sketch}\label{sketch}
The fundamental solution to the SHE $\mathcal{Z}^{\mathbf{nw}}(T,X)$ corresponds to delta initial data $\mathcal{Z}_0(X)=\delta_{X=0}$. For any positive $T$, this results in a strictly positive solution, hence the corresponding KPZ equation solution is well-defined for $T>0$ and this initial data is termed \emph{narrow wedge} since in short time $\mathcal{Z}(T,X)$ is well-approximated by the Gaussian heat-kernel whose logarithm is a very thin parabola $\frac{X^2}{2T}$.

\bd[Cole-Hopf Transform]\label{bd:ColeHopf}
The Cole-Hopf transform of $\mathcal{Z}^{\mathbf{nw}}(T,X)$ is denoted here by $\mathcal{H}^{\mathbf{nw}}(T,X):= \log \mathcal{Z}^{\mathbf{nw}}(T,X)$. We further define a scaled and centered version of this as
\begin{align}\label{eq:DefUpsilon}
\Upsilon_T(y):= \frac{\mathcal{H}^{\mathbf{nw}}(2T, (2T)^{\frac{2}{3}} y)+ \frac{T}{12}}{T^{\frac{1}{3}}}.
\end{align}
\ed

The proof of our main results relies upon a combination of three ingredients: (1) lower tail bounds for the narrow wedge initial data recently proved in \cite{CG18}, (2) Gibbsian line ensemble techniques applied to the KPZ line ensemble \cite{CorHam16}, and (3) explicit integral formulas for moments of the SHE with delta initial data. Now, we give an overview of our proofs. A more involved discussion of the KPZ line ensemble is contained in Section \ref{Tools}.

To prove Theorem~\ref{Main1Theorem}, one of our main tools is the upper and lower bound for the lower tail of the one point distribution of the narrow wedge solution of the KPZ equation given in Proposition~\ref{NotMainTheorem}. However, to use this result, we need a connection between the solution of the KPZ equation under general initial conditions and the narrow wedge solution. This connection is made through the following identity (which follows from the Feynman-Kac formula) which represents the one point distribution of the KPZ equation started from $\mathcal{H}_0$ as a convolution between the spatial process $\Upsilon_T(\cdot)$ and the initial data $\mathcal{H}_0(\cdot)$.
\bp[Lemma~1.18 of \cite{CorHam16}]\label{Distribution}
For general initial data $\mathcal{H}_{0}(\cdot):=\mathcal{H}(0,\cdot)$ and for a fixed pair $T>0$ and $X\in \Real$, the Cole-Hopf solution $\mathcal{H}(T,X)$ of the KPZ equation satisfies
\begin{align}\label{eq:DistUnderGI}
\mathcal{H}(2T,X) \stackrel{d}{=} \log\left(\int^{\infty}_{-\infty} e^{\mathcal{H}^{\mathbf{nw}}(2T,Y)+ \mathcal{H}_0(X-Y)} dY\right)\stackrel{d}{=} -\frac{T}{12}+ \log \Big( \int_{-\infty}^{\infty} e^{T^{\frac{1}{3}} \Upsilon_T((2T)^{-\frac{1}{3}}Y) +\mathcal{H}_0(X-Y) } dY\Big).
\end{align}
\vskip-.2cm
\noindent Furthermore, for $\mathcal{H}_0$ as in \eqref{eq:ScaledInitialData}, we have
\begin{align}\label{eq:DistUnderGI2}
\frac{\mathcal{H}(2T, (2T)^{\frac{1}{3}} X) + \frac{T}{12}- \frac{2}{3}\log (2T)}{T^{\frac{1}{3}}} \stackrel{d}{=} \frac{1}{T^{\frac{1}{3}}}\log \Big( \int_{-\infty}^{\infty} e^{T^{\frac{1}{3}} \big(\Upsilon_T(Y) +f(X-Y)\big)} dY\Big).
\end{align}
\ep

To employ this identity, we need tail bounds for the entire spatial process  $\Upsilon_T(\cdot)$. Presently, exact formulas amenable to rigorous asymptotics are only available for one-point tail probabilities, and not multi-point. However, by using the Gibbs property for the KPZ line ensemble (introduced in \cite{CorHam16} and recalled here in Section~\ref{Tools}) we will be able to extend this one-point tail control to the entire spatial process. Working with the Gibbs property is a central technical aspect of our present work and forms the backbone of the proof of Theorem~\ref{Main1Theorem}.

Besides the KPZ line ensemble, another helpful property of the narrow wedge KPZ solution is the stationarity of the spatial process $\Upsilon_T(\cdot)$ after a parabolic shift.

 \bp[Proposition~1.4 of \cite{Amir11}]\label{StationarityProp}
The one point distribution of $\Upsilon_T(y)+\frac{y^2}{2^{2/3}}$ does not depend on the value of $y$.
\ep
%

The proof of Theorem~\ref{Main4Theorem} shares a similar philosophy with that of Theorem~\ref{Main1Theorem}. We first prove an upper (as Theorem~\ref{GrandUpTheorem}) and a lower bound for the upper tail probability of $\Upsilon_T(0)$. The proof of Theorem~\ref{GrandUpTheorem} employs a combination of the one-point Laplace transform formula (see Proposition~\ref{ppn:PropConnection}) and moment formulas (see the proof of Lemma \ref{MomBoundLem}) for $\mathcal{Z}^{\mathbf{nw}}$.

The rest of the proof of Theorem~\ref{Main4Theorem} is based on the Gibbs property of the KPZ line ensemble and the FKG inequality of the KPZ equation. The FKG inequality of the KPZ equation is, for example (as shown in \cite[Proposition~1]{CQ11})  a consequence of the positive associativity of its discrete analogue, the asymmetric simple exclusion process (ASEP).
\bp[Proposition~1 of \cite{CQ11}]\label{FKGProp} Let $\mathcal{H}$ be the Cole-Hopf solution to KPZ started from initial data $\mathcal{H}_0$. Fix $k\in \ZZ_{>0}$. For any $T_1,\ldots ,T_k\geq 0$, $X_1,\ldots , X_k\in \RR$ and $s_1, \ldots ,s_k\in \RR$,
 \begin{align}\label{eq:FKGStatement}
 \mathbb{P}\Big(\bigcap_{\ell=1}^{k}\big\{\mathcal{H}(T_\ell,X_\ell)\leq s_\ell\big\}\Big)\geq \prod_{\ell =1}^{k}\mathbb{P}\Big(\mathcal{H}(T_\ell,X_\ell)\leq s_{\ell}\Big).
 \end{align}
 \ep

A simply corollary of this result is that for $T_1, T_2\in \RR_{>0}$, $X_1, X_2 \in \RR$ and $s_1, s_2\in \RR$,
\begin{align}\label{eq:RevFKG}
\mathbb{P}\Big(\mathcal{H}(T_1, X_1)>s_1, \mathcal{H}(T_2, X_2)>s_2\Big)\geq \mathbb{P}\Big(\mathcal{H}(T_1, X_1)>s_1\Big)\mathbb{P}\Big(\mathcal{H}(T_2, X_2)>s_2\Big).
\end{align}

\subsection{Narrow wedge and Brownian initial data results}\label{NWandBM}
Neither narrow wedge nor two-sided Brownian initial data belongs to the class of functions in Definition~\ref{Hypothesis}. We record  here the analogues of Theorems~\ref{Main1Theorem} and \ref{Main4Theorem} for these two cases.
As mentioned in the last section, the one point tail results for the narrow wedge solution are important inputs to the proof of Theorems~\ref{Main1Theorem} and \ref{Main4Theorem}. We recall these below.
\bp[Theorem~1.1 of \cite{CG18}]\label{NotMainTheorem}
Fix $\epsilon, \delta \in (0,\frac{1}{3})$ and $T_0>0$. Then, there exist $s_0= s_0(\epsilon, \delta, T_0)$, $K_1 = K_1(\epsilon, \delta, T_0)>0$, $K_2= K_2(T_0)>0$ such that for all $s\geq s_0$ and $T\geq T_0$,
\begin{align}\label{eq:PrevRes1}
\mathbb{P}(\Upsilon_T(0)\leq -s)&\leq e^{-T^{1/3}\frac{4s^{5/2}(1-\epsilon)}{15\pi} } + e^{-K_1s^{3-\delta}- \epsilon sT^{1/3}} + e^{-\frac{(1-\epsilon)s^3}{12}}\\
\textrm{and,}\qquad \mathbb{P}(\Upsilon_T(0)\leq -s)&\geq e^{-T^{1/3}\frac{4s^{5/2}(1+\epsilon)}{15\pi} } + e^{-K_2s^3}.
\end{align}
\ep

Our general initial data results also rely upon upper and lower bounds on the upper tail probability of $\Upsilon_T(\cdot)$ which are, in fact, new (see Section \ref{sec:previouswork} for a discussion of previous work).


 \bt\label{GrandUpTheorem}
For any $T_0>0$, there exist $s_0= s_0(T_0)>0$ and $c_1=c_1(T_0)> c_2=c_2( T_0)>0$ such that for all $s\geq s_0$ and $T>T_0$
\begin{align}\label{eq:RoughBd1}
e^{-c_1s^{3/2}}\leq \mathbb{P}(\Upsilon_T(0)\geq s)\leq e^{-c_2s^{3/2}}.
\end{align}

We may further specify values of $c_1$ and $c_2$ for which \eqref{eq:RoughBd1} holds, provided we assume $T_0>\pi$. In that case, for any $\epsilon\in (0,\frac{1}{2})$, there exists $s_0= s_0(\epsilon, T_0)>0$ such that for all $s\geq s_0$ and $T\geq T_0$, \eqref{eq:RoughBd1} holds with the following choices for $c_1>c_2$:
\bei
\ii If $s_0\leq s< \frac{1}{8}\epsilon^{2} T^{\frac{2}{3}}$ then we may take $c_1= \frac{4}{3}(1+\epsilon)$ and $c_2= \frac{4}{3}(1-\epsilon)$.
\ii If $s\geq \max\{s_0, \frac{9}{16}\epsilon^{-2} T^{\frac{2}{3}} \}$ then we may take $c_1= 4\sqrt{3}(1+\epsilon)$ and $c_2= \frac{4}{3}(1-\epsilon)$. Furthermore, for $c_1=\frac{4}{3}(1+\epsilon)$  there exists a sequence $\{s_n\}_{n\geq 1}$ with $s_n\to \infty$ as $n\to \infty$ such that $\mathbb{P}(\Upsilon_T>s_n)> e^{-c_1 s_n^{3/2}}$ for all $n$.
\ii If $\max\{s_0, \frac{1}{8}\epsilon^2 T^{\frac{2}{3}} \} \leq s \leq \max\{s_0, \frac{9}{16}\epsilon^{-2} T^{\frac{2}{3}} \}$ then we may take  $c_1= 2^{7/2}\epsilon^{-3}$ and $c_2=\frac{4}{3}\epsilon$.
\ee
\et

%

\begin{rem}
Part $(i)$ of Theorem~\ref{GrandUpTheorem} shows that $\mathbb{P}(\Upsilon_T(0)>s)$ is close to $\exp(-4s^{\frac{3}{2}}/3)$  when $s\ll T^{\frac{2}{3}}$. This is in agreement with the fact that the tail probabilities of $\Upsilon_T(0)$ should be close to the tails of the Tracy-Widom GUE distribution as $T$ increases to $\infty$. Part $(ii)$ of Theorem~\ref{GrandUpTheorem} shows that the upper bound to $\mathbb{P}(\Upsilon_T(0)>s)$ is close $\exp(-4s^{\frac{3}{2}}/3)$ when $s\gg T^{\frac{2}{3}}$. We also have some lower bound which is not tight. However, part $(ii)$ further tells that the lower bound for $\mathbb{P}(\Upsilon_T(0)>s)$ cannot differ much from $\exp(-4s^{\frac{3}{2}}/3)$ for all large $s$. In the regime $s=O(T^{\frac{2}{3}})$, we do not have tight upper and lower bounds in \eqref{eq:RoughBd1}, although, the decay exponent of $\mathbb{P}(\Upsilon_T(0)>s)$ will still be equal to $3/2$.
\end{rem}

Our next two results are about the tail probabilities for the KPZ equation with two sided Brownian motion initial data; as this initial data falls outside our class, some additional arguments are necessary. Define $\mathcal{H}^{\mathrm{Br}}_0:\RR\to \RR$ as $\mathcal{H}^{\mathrm{Br}}_0(x):= B(x)$ where $B$ is a two sided standard Brownian motion with $B(0)=0$. Denote the Cole-Hopf solution of the KPZ equation started from this initial data $\mathcal{H}^{\mathrm{Br}}_0$ by $\mathcal{H}^{\mathrm{Br}}(\cdot,\cdot)$ and define
\begin{align}\label{eq:h_BrDefine}
h^{\mathrm{Br}}_{T}(y):= \frac{\mathcal{H}^{\mathrm{Br}}(2T,(2T)^{\frac{2}{3}} y)+ \frac{T}{12} - \frac{2}{3}\log (2T)}{T^{\frac{1}{3}}} \quad \forall T >0.
\end{align}
We first state our result on the lower tail of $h^{\mathrm{Br}}_{T}(0)$.
\bt\label{Main3Theorem}
 Fix $\epsilon, \delta \in (0, \frac{1}{3})$ and $T_0>0$.
There exist $s_0=s_0(\epsilon, \delta, T_0)$ and $K=K(\epsilon, \delta, T_0)>0$ such that for all $s\geq s_0$ and $T\geq T_0$,
\begin{align}\label{eq:UpBoundBrData}
\mathbb{P}\big(h^{\mathrm{Br}}_{T}(0)\leq -s\big)\leq e^{- T^{1/3}\frac{4(1-\epsilon) s^{5/2}}{15\pi}} + e^{- Ks^{3-\delta}- \epsilon s T^{1/3}} + e^{-\frac{(1-\epsilon)s^3}{12}}.
\end{align}
\et

Our last result of this section is about the upper tail probability of $h^{\mathrm{Br}}_T(0)$.
  \bt\label{Main6Theorem}
Fix $\epsilon,\mu\in (0,\frac{1}{2})$ and $T_0>0$. Then, there exists $s_0 = s_0(\epsilon, \mu, T_0)$ such that for all $s\geq s_0$ and $T\geq T_0$,
\begin{align}\label{eq:UpTailBrData1}
e^{-c_1s^{3/2}} \leq \mathbb{P}\big(h^{\mathrm{Br}}_T(0)>s\big)\leq e^{-c_2s^{3/2}}+ e^{-\frac{1}{9\sqrt{3}}(\mu s)^{3/2}}
\end{align}
where $c_1>c_2$ depend on the values of $\epsilon$, $\mu$ and $T_0$ as described in Theorem~\ref{Main4Theorem}.

%
%
\et

In Theorem \ref{Main6Theorem}, the second term of the upper bound (on the right-hand side of the equation) comes from the fact that Brownian motion is random, and the first term arises in an analogous way as it does for deterministic initial data in Theorem~\ref{Main4Theorem}.

As proved in \cite[Theorem~2.17]{BCFV}, $h^{\mathrm{Br}}_T(0)$ converges in law to the Baik-Rain distribution (see \cite{BR00, FS06, SS04, PS04, BFP10}). The following corollary strengthens the notion of that convergence and implies that the moments of $h^{\mathrm{Br}}_T(0)$ converge to the moments of the limiting Baik-Rains distribution. This answers a question posed to us by Jean-Dominique Deutschel (namely, that the variance converges).
\bc\label{ShortCor}
Let $X$ be a Baik-Rains distributed random variable (see \cite[Definition~2.16]{BCFV}). Then, $\mathbb{E}[e^{t |X|}]<\infty$ and for all $t\in \RR$,
\begin{align}\label{eq:MGFConv}
\mathbb{E}\big[e^{t |h^{\mathrm{Br}}_T(0)|}\big]\rightarrow \mathbb{E}\big[e^{t |X|}\big], \quad \text{as } T\to \infty.
\end{align}
\ec

\begin{proof}
Theorems~\ref{Main3Theorem} and \ref{Main6Theorem} show that $e^{t|h^{\mathrm{Br}}_T(0)|}$ is uniformly integrable. The dominated convergence theorem, along with \cite[Theorem~2.17]{BCFV} yields \eqref{eq:MGFConv} and $\mathbb{E}[e^{t |X|}]~<~\infty$.
\end{proof}

\subsection{Previous work and further directions}\label{sec:previouswork}

The study of tail probabilities for the KPZ equation and the SHE has a number of motivations including intermittency and large deviations. We recall some of the relevant previous literature here and compare what is done therein to the results of this present work.

The first result regarding the lower tail probability of $\mathcal{Z}(T,X)$ the proof of its almost sure positivity by \cite{Mueller91}. Later, \cite{CN08} investigated the lower tail of the SHE restricted on the unit interval with general initial data and Dirichlet boundary condition; they bounded $\mathbb{P}(\log \mathcal{Z}(T,X)\leq -s)$ from above by $c_1\exp(-c_2s^{\frac{3}{2}-\delta})$  (where $c_1,c_2$ are two positive constants depending inexplicitly on $T$). In \cite{GMF14}, this upper bound was further improved to $c_1\exp(-c_2s^{2})$ for the delta initial data SHE (the constants are different but still depend inexplicitly on $T$). Using these bounds, \cite{CorHam16} demonstrated similar upper bounds on the lower tail probability of the KPZ equation under general initial data. There are also tail bounds for the fractional Laplacian ($\Delta^{\alpha/2}$ with $\alpha\in (1,2]$) SHE. \cite[Theorem 1.5]{CHN16} generalizes the bound of \cite{CN08} and shows an upper bound\footnote{In light of our results, it might natural to expect the true decay exponent is $3-1/\alpha$. Perhaps the methods of \cite{GMF14} can be applied to give decay at least with exponent $2$. Heuristically, one may be able to see the true exponent by using the physics weak noise theory as in, for example, \cite{MKV16b}.} with exponent $2- 1/\alpha$ ($=3/2$ when $\alpha=2$).

None of the previous SHE lower tail bounds were suitable to taking time $T$ large. Specifically, the constants depend inexplicitly on $T$ and the centering by $T/24$ and scaling by $T^{1/3}$ were not present. Thus, as $T$ grows, the bounds weaken significantly to the point of triviality. For instance, one cannot conclude tightness of the centered and scaled version of $\log \mathcal{Z}(T,X)$ ($\Upsilon_T(X)$ herein) as $T$ goes to infinity using the bounds.

The first lower tail bounds suitable to taking $T$ large came in our previous work \cite{CG18} which dealt with the delta initial data SHE (see Proposition \ref{NotMainTheorem} herein). That result relied upon an identity of \cite{BorGor16} (see Proposition~\ref{ppn:PropConnection}). No analog of that identity seems to exist for general initial data. This is why we use the KPZ line ensemble approach in our present work.

%


The upper tail probability of the SHE had been studied before in a number of places. For instance, see \cite{ChD15, CJK, KKX17} in regards to its connection to the moments and the intermittency property \cite{GaMo90, GKM07} of the SHE. Again, there is a question of whether results are suitable to taking $T$ large. The only such result is \cite[Corollary 14]{CQ11} which shows that for some constants $c_1, c_2, c^{\prime}_1, c^{\prime}_2$, and $s,T\geq 1$,
$ \mathbb{P}(\Upsilon_T>s)\leq c_1\exp(-c^{\prime}_1s T^{1/3})+ c_2\exp(-c^{\prime}_2s^{3/2}).$
When $s \ll  T^{\frac{2}{3}}$ the second bound is active and one sees the expected $3/2$ power-law in the exponent. However, as $s\gg T^{\frac{2}{3}}$, the leading term above become $c_1\exp(-c^{\prime}_1s T^{\frac{1}{3}})$ and only demonstrates exponential decay.  Our result (Theorem~\ref{GrandUpTheorem}) shows that $c_1\exp(-c^{\prime}_1s T^{\frac{1}{3}})$ is not a tight upper bound for $\mathbb{P}(\Upsilon_T>s)$ in this regime of $s$. In fact, the $3/2$ power-law is shown to be valid for all $s$ even as $T$ grows (with upper and lower bounds of this sort).

Some works have focused on the large $s$ but fixed $T$ upper tail, e.g. \cite{CJK} showed that
$
\log \mathbb{P}\big(\log \mathcal{Z}(T,X)> s\big) \asymp - s^{\frac{3}{2}} \quad \text{as }s\to \infty
$
where $Z(0,X) \equiv 1$. These results are not suitable for taking $T$ and $s$ large together.
Our results (Theorems~\ref{Main4Theorem}, \ref{GrandUpTheorem} and \ref{Main6Theorem}) provide the first upper and lower bound for the upper tail probability which are well-adapted to taking $T$ large. In particular, we showed that for a wide range of initial data the exponent of the upper tail decay is always $\frac{3}{2}$ (a result which was not proved before for any specific initial data). However, the constants in the exponent for our bounds on the upper tail probability are not optimal.

It is natural to speculate on the values of these optimal coefficients. There is some discussion of this in the physics literature (see, for example, \cite{MKV16b, HLMRS18}) based on numerics and the weak noise theory (WNT)\footnote{The approach is to look at the KPZ equation in short time with very weak noise. This is a different problem than looking at the deep tail, but so far the results one gets from the WNT seem to be true even in long time.}. In the deep lower tail (the $5/2$ exponent region) the coefficient depends on the initial data and can be predicted using the WNT as in \cite{MKV16b}. For the shallow lower tail (the $3$ exponent region) one expects (by reason of continuity) to have a coefficient corresponding to the tail decay of the KPZ fixed point with the corresponding initial data. Remarkably, for the upper tail (the $3/2$ exponent region) it seem that for all deterministic initial data, the upper tail coefficient remains the same\footnote{For instance, for flat and narrow wedge initial data, the upper tail seems to have the same $4/3$ coefficient.}. However, for Brownian initial data, the coefficient changes by a factor of 2.

There have been previous considerations of tail bounds in the direction of studying large deviations for the KPZ equation (i.e., the probability that as $T\to \infty$,  $\log \mathcal{Z}(T,X)$ looks like $cT$ for some constant not equal to $-1/24$). The speed for the upper tail and lower tail are different (the former being $T$ and the later being $T^2$). The lower tail large deviation principle has been the subject of significant study in the physics literature (see  \cite{SMP17,CGKLT18,KD18b,KD18a} and references therein). Recently, \cite{LCT18} provided a rigorous proof of the lower tail rate function. We are not aware of a rigorous proof of the (likely) simpler upper tail rate function for the KPZ equation (there is some non-rigorous predictions about this, see e.g. \cite{PSG16}). However, for a discrete analog (the log-gamma polymer) and a semi-discrete analog (the O'Connell-Yor polymer) such an upper tail bound is proved in \cite{GS13} and \cite{Janjigian15} respectively.

We finally mention a few directions worth pursuing. Theorem \ref{Main1Theorem} only provides an upper bound on the lower tail. Our KPZ line ensemble methods are able to produce a lower bound, but with a worse (larger) power law. It is only for the narrow wedge initial data that we have a tight matching lower bound. We conjecture that there should be a similarly tight upper and lower bound for the lower tail which holds true for general initial data.
The large deviation result for the lower tail (see \cite{SMP17,CGKLT18,LCT18}) is only shown for narrow wedge initial data (though there is also some work needed for flat and Brownian initial data). It would be interesting to determine how the large deviation rate function depends on the initial data. In fact, even for the KPZ fixed point (e.g. TASEP) this does not seem to be resolved.
\smallskip

\noindent\textbf{Outline.}
Section~\ref{Tools} reviews the KPZ line ensemble and its Gibbs property. Sections~\ref{Proof1Theorem} and \ref{Proof2Theorem}  establish the lower tail bounds of Theorems~\ref{Main1Theorem} and~\ref{Main3Theorem} by first analyzing the narrow wedge initial condition tails and then feeding those bounds into an argument leveraging the Gibbs property and the convolution formula of Proposition~\ref{Distribution}. We prove the upper tail bounds of Theorem~\ref{GrandUpTheorem} in Section~\ref{UpperTailNSEC} by analyzing the moment formula (see Lemma~\ref{MomBoundLem}) and the Laplace transform formula (see Proposition~\ref{ppn:PropConnection}) of the narrow wedge solution. Sections ~\ref{Proof4Theorem} and~\ref{Proof5Theorem} contain the proofs of (respectively) Theorems~\ref{Main4Theorem} and~\ref{Main6Theorem} on the upper tail bounds under general initial data.

\smallskip

\noindent\textbf{Acknowledgements.}
 We thank J. Baik, G. Barraquand, S. Das, P. Le Doussal, A. Krajenbrink, J. Quastel, L.-C. Tsai, and B. Virag for helpful conversations and comments, as well as an anonymous referee for many helpful comments. I.C. was supported in part by a Packard Fellowship for Science and Engineering, and by NSF DMS-1811143, DMS-1664650.

\section{KPZ line ensemble}\label{Tools}


This section reviews (following the work of \cite{CorHam16}) the KPZ line ensemble and its Gibbs property. We use this construction in order to transfer one-point information (namely, tail bounds) into spatially uniform information for $\Upsilon_T(y)$ (see Definition~\ref{bd:ColeHopf}). It is through this mechanism that we can escape the bonds of exact formulas and generalize the conclusions of \cite{CG18} to general  initial data.

\bd\label{LineEnsemble}
Fix intervals $\Sigma\subset \NN$ and $\Lambda\subset \RR$. Let $\mathcal{X}$ be the set of all continuous functions $f:\Sigma\times \Lambda \mapsto \RR$ endowed with the topology of uniform convergence on the compact subsets of $\Sigma\times \Lambda$. Denote the sigma field generated by the Borel subsets of $\mathcal{X}$ by $\mathcal{C}$.

 A $\Sigma\times \Lambda$-indexed \emph{line ensemble} $\mathcal{L}$ is a random variable in a probability space $(\Omega, \mathfrak{B}, \mathbb{P})$ such that it takes values in $\mathcal{X}$ and is measurable with respect to $(\mathfrak{B}, \mathcal{C})$. In simple words, $\mathcal{L}$ is a collection of $\Sigma$-indexed random continuous curves, each mapping $\Lambda$ to $\RR$.

 Fix two integers $k_1\leq k_2$, $a<b$ and two vectors $\vec{x}, \vec{y}\in \RR^{k_2-k_1+1}$. A $\{k_1, \ldots ,k_2\}\times (a,b)$ - indexed line ensemble is called a \emph{free Brownian bridge} line ensemble with the entrance data $\vec{x}$ and the exit data $\vec{y}$ if its law, denoted here as $\mathbb{P}^{k_1,k_2, (a,b), \vec{x}, \vec{y}}_{\mathrm{free}}$, is that of $k_2-k_1+1$ independent Brownian bridges starting at time $a$ at points $\vec{x}$ and ending at time $b$ at points $\vec{y}$. We use the notation $\mathbb{E}^{k_1,k_2, (a,b), \vec{x}, \vec{y}}_{\mathrm{free}}$ for the associated expectation operator.

 Consider a continuous function $\mathbf{H}:[0,\infty)\to \RR$, which we call a {\it Hamiltonian}. Given $\mathbf{H}$ and two measurable functions $f:[0,\infty)\to \RR\cup \{\infty\}$ and $g:[0,\infty)\to \RR\cup \{-\infty\}$, we define a $\{k_1, \ldots ,k_2\}\times (a,b)$ - indexed line ensemble with the entrance data $\vec{x}$, the exit data $\vec{y}$, boundary data $(f,g)$ and $\mathbf{H}$ to be the law of $\mathbb{P}^{k_1,k_2, (a,b), \vec{x}, \vec{y}, f, g}_{\mathbf{H}}$ on curves $\mathcal{L}_{k_1}, \ldots , \mathcal{L}_{k_2}:[0,\infty)\to \RR$ which is given in terms of the following Radon-Nikodym derivative
 \begin{align}
 \frac{d\mathbb{P}^{k_1,k_2, (a,b), \vec{x}, \vec{y}, f,g}_{\mathbf{H}}}{d\mathbb{P}^{k_1,k_2,(a,b)}_{\mathrm{free}}}(\mathcal{L}_{k_1}, \ldots , \mathcal{L}_{k_2}) &= \frac{W^{k_1,k_2,(a,b),\vec{x}, \vec{y} f,g}_{\mathbf{H}}(\mathcal{L}_{k_1}, \ldots , \mathcal{L}_{k_2})}{Z^{k_1, k_2, (a,b), \vec{x}, \vec{y}, f, g}_{\mathbf{H}}}\\
 W^{k_1, k_2, (a,b), \vec{x}, \vec{y}, f,g}_{\mathbf{H}} (\mathcal{L}_{k_1}, \ldots , \mathcal{L}_{k_2})&= \exp\left\{- \sum_{k=k_1-1}^{k_2}\int^{b}_{a} \mathbf{H}\big(\mathcal{L}_{k_1+1}(u)- \mathcal{L}_{k}(u)\big) du\right\}
 \end{align}
 with the convention $\mathcal{L}_{k_1-1}= f$  and $\mathcal{L}_{k_2+1}=g$. Here, the normalizing constant is given by
 \begin{align}
 Z^{k_1,k_2, (a,b), \vec{x}, \vec{y}, f,g}_{\mathbf{H}} = \mathbb{E}^{k_1,k_2, (a,b)}_{\mathrm{free}}\big[  W^{k_1, k_2, (a,b), \vec{x}, \vec{y}, f, g}_{\mathbf{H}} (\mathcal{L}_{k_1}, \ldots , \mathcal{L}_{k_2})\big]
 \end{align}
 where the curves $(\mathcal{L}_{k_1}, \ldots , \mathcal{L}_{k_2})$ are distributed via $\mathbb{P}^{k_1,k_2, (a,b), \vec{x},\vec{y}}_{\mathrm{free}}$. Throughout this paper we will restrict our attention to one parameter family of Hamiltonians indexed by $T\geq 0$:
 \begin{align}\label{eq:Hamiltonian}
 \mathbf{H}_T(x): = e^{T^{1/3} x}.
\end{align}


A $\Sigma\times \Lambda$-indexed line ensemble  $\mathcal{L}$ satisfies the $\mathbf{H}$-\emph{Brownian Gibbs property} if for any subset $K=\{k_1, k_1+1, \ldots , k_2\} \subset \Sigma$ and $(a,b)\subset \Lambda$, one has the following distributional invariance
  \begin{align}\label{eq:GibbsProperty}
  \mathrm{Law}\left(\mathcal{L}\big|_{K\times (a,b)} \text{ conditional on } \mathcal{L}\big|_{\Sigma\times \Lambda \backslash K\times (a,b)}\right) = \mathbb{P}^{k_1,k_2, (a,b), \vec{x}, \vec{y}, f,g}_{\mathbf{H}}
\end{align}
where $\vec{x} = (a_{k_1}, \ldots , a_{k_2})$, $\vec{y} = (b_{k_1}, \ldots , b_{k_2})$ and $f=\mathcal{L}_{k_1-1}|_{(a,b)}$, $g= \mathcal{L}_{k_2+1}$ with $f=-\infty$ if $k_1-1\notin \Sigma$ and $g= +\infty$ if $k_2+1\notin \Sigma$. This is a spatial Markov property --- the ensemble in a given region has marginal distribution only dependent on the boundary-values of said region.

 Denote the sigma field generated by the curves with indices outside $K\times (a,b)$ by $\mathcal{F}_{\mathrm{ext}}(K\times (a,b))$. The random variable $(\mathfrak{a}, \mathfrak{b})$ is a $K$-\emph{stopping domain} if
$
\{\mathfrak{a}\leq a, \mathfrak{b}\geq b\}\in \mathcal{F}_{\mathrm{ext}}(K\times (a,b)).
$
Let $C^{K}(a,b)$ be the set of continuous functions $(f_{k_1}, \ldots , f_{k_2})$ where $f_i:(a,b)\to \RR$ and define
\begin{align}
C^{K}:= \Big\{(a,b, f_{k_1}, \ldots , f_{k_2}): a<b \text{ and }(f_{k_1}, \ldots , f_{k_2})\in C^{K}(a,b) \Big\}.
\end{align}
Denote the set of all Borel measurable functions from $C^{K}$ to $\RR$ by $\mathcal{B}(C^{K})$. Then, a $K$-stopping domain $(\mathfrak{a}, \mathfrak{b})$ is said to satisfy the \emph{strong} $\mathbf{H}$-\emph{Brownian Gibbs property} if for all $F\in \mathcal{B}(C^{K})$, following holds $\mathbb{P}$-almost surely,
\begin{align}\label{eq:StrongGibbs}
\mathbb{E}\Big[F\big(\mathfrak{a}, \mathfrak{b},\mathcal{L}\big|_{K\times (\mathfrak{a}, \mathfrak{b})}\big)\Big| \mathcal{F}_{\mathrm{ext}}(K\times (\mathfrak{a}, \mathfrak{b}))\Big] = \mathbb{E}^{k_1,k_2,(\ell,r),\vec{x}, \vec{y}, f,g}_{\mathbf{H}}\Big[F(\ell, r, \mathcal{L}_{k_1}, \ldots , \mathcal{L}_{k_2})\Big]&&&&
\end{align}
 where $\ell = \mathfrak{a}$, $r=\mathfrak{b}$, $\vec{x} = \{\mathcal{L}_{i}(\mathfrak{a})\}^{k_2}_{i=k_1}$, $\vec{y} = \{\mathcal{L}_{i}(\mathfrak{b})\}^{k_2}_{i=k_1}$, $f(\cdot) = \mathcal{L}_{k_1-1}(\cdot)$ (or $+\infty$ if $k_1-1\notin \Sigma$) and $g(\cdot) = \mathcal{L}_{k_2+1}(\cdot)$ (or $-\infty$ if $k_2+1\notin \Sigma$). On the l.h.s. of \eqref{eq:StrongGibbs}, $\mathcal{L}\Big|_{K\times (\mathfrak{a}, \mathfrak{b})}$ is the restriction of the $\mathbb{P}$-distributed curves and on the r.h.s. $\mathcal{L}_{k_1}, \ldots , \mathcal{L}_{k_2}$ is $\mathbb{P}^{k_1,k_2,(\ell,r),\vec{x}, \vec{y}, f,g}_{\mathbf{H}}$-distributed.
\ed

 \begin{rem}\label{LineEnsembleToBM}
When $k_1=k_2=1$ and $(f,g) = (+\infty, -\infty)$ the measure $\mathbb{P}^{k_1,k_2,(a,b), \vec{x}, \vec{y}, f,g}_{\mathbf{H}}$ is same as the measure of a free Brownian bridge started from $\vec{x}$ and ended at $\vec{y}$.
 \end{rem}
The following lemma demonstrates a sufficient condition under which the strong $\mathbf{H}$-Brownian Gibbs property holds.
\bl[Lemma 2.5 of \cite{CorHam16}]
Any line ensemble which enjoys the $\mathbf{H}$-Brownian Gibbs property also enjoys the strong $\mathbf{H}$-Brownian Gibbs property.
\el

The next proposition relates the narrow wedge KPZ equation to the KPZ line ensemble\footnote{Note, we do not require the full strength of the result proved in Theorem~2.15 of \cite{CorHam16}. That result also proves uniform over $T$ of the local Brownian nature of the top curve $\Upsilon^{(1)}_T(x)$ as $x$ varies.}.

\bp[Theorem~2.15 of \cite{CorHam16}]\label{NWtoLineEnsemble} Fix any $T>0$. Then there exists an $\NN\times \RR$-indexed line ensemble $\mathcal{H}_T =\{\mathcal{H}^{n}_{T}(x)\}_{n\in \NN, x\in \RR}$ satisfying the following properties:
\be
\ii The lowest indexed curve $\mathcal{H}^{1}_{T}(X)$ is equal in distribution (as a process in $X$) the Cole-Hopf solution $\mathcal{H}^{\mathbf{nw}}(T,X)$ of KPZ started from the narrow wedge initial data.
 \ii $\mathcal{H}_T$ satisfies the $\mathbf{H}_{1}$-Brownian Gibbs property (see Definition~\ref{LineEnsemble}).
 \ii Define the scaled KPZ line ensemble $\{\Upsilon^{(n)}_T(x)\}_{n\in \NN,x\in \RR}$ as follows
 \begin{align}\label{eq:UsilonNDef}
 \Upsilon^{(n)}_T(x) := \frac{\mathcal{H}^{n}_{2T}\big((2T)^{\frac{1}{3}} x\big)+ \frac{T}{12}}{T^{\frac{1}{3}}}.
\end{align}
Then, $\{2^{-\frac{1}{3}}\Upsilon^{(n)}_T(x)\}_{n\in \NN,x\in \RR}$ satisfies the $\mathbf{H}_{2T}$-Brownian Gibbs property\footnote{This pesky $2^{-\frac{1}{3}}$ compensates for the fact that it is missing in the denominator of $\Upsilon^{(n)}_T(x)$.}
\ee
\ep

  The following proposition is a monotonicity result which shows that two line ensembles with the same index set can be coupled in such a way that if the boundary conditions of one ensemble dominates the other, then likewise do the curves.

 \bp[Lemmas ~2.6 and 2.7 of \cite{CorHam16}]\label{Coupling1}
Fix an interval $K=\{k_1, \ldots , k_2\}\subset \Sigma $ for some fixed positive integers $k_1<k_2$, $(a,b)\subset \Lambda$ for $a<b$ and two pairs of vectors $\vec{x}_{1},\vec{x}_{2}$ and $\vec{y}_{1}, \vec{y}_{2}$ in $ \RR^{k_2-k_1+1}$. Consider any two pairs of measurable functions $f, \tilde{f}: (a,b)\to\RR \cup \{+\infty\} $ and $g, \tilde{g}: (a,b) \to \RR \cup \{-\infty\}$ such that $\tilde{f}(s)\leq f(s)$, $\tilde{g}(s)\leq  g(s)$ for all $s\in (a,b)$ and $x^{(k)}_{2}\leq x^{(k)}_{1}$, $y^{(k)}_{2}\leq y^{(k)}_{1}$ for all $k\in K$. Let $\mathcal{Q}= \{\mathcal{Q}^{(n)}(x)\}_{n\in K, x\in(a,b)}$ and $\widetilde{\mathcal{Q}}= \{\widetilde{\mathcal{Q}}^{(n)}(x)\}_{n\in K, x\in(a,b)}$ be two $K\times (a, b)$-indexed line ensembles in the probability space $(\Omega, \mathcal{B}, \mathbb{P})$ and $(\widetilde{\Omega}, \widetilde{\mathcal{B}}, \widetilde{\mathbb{P}})$ respectively such that $\mathbb{P}$ equals to $ \mathbb{P}^{k_1, k_2, (a,b), \vec{x}_1, \vec{y}_1, f,g}_{\mathbf{H}}$ and $\widetilde{\mathbb{P}}$ equals to $ \mathbb{P}^{k_1, k_2, (a,b), \vec{x}_2, \vec{y}_2, \tilde{f},\tilde{g}}_{\mathbf{H}}$.
If $H:[0,\infty)\to \RR$ is convex, then, there exists a coupling (i.e., a common probability space upon which both measures are supported) between $\mathbb{P}$ and $\widetilde{\mathbb{P}}$ such that $\widetilde{\mathcal{Q}}^{(j)}(s)\leq \mathcal{Q}^{(j)}(s)$ for all $n\in K$.
 \ep

%

Let us provide the basic idea behind how we use Lemma~\ref{Coupling1}. Note that by $\mathbf{H}$-Brownian Gibbs property the lowest indexed curve  $2^{-\frac{1}{3}}\Upsilon^{(1)}_T(\cdot)$ of the $\NN$-indexed KPZ line ensemble $\{2^{-\frac{1}{3}}\Upsilon^{(n)}_{T}(x)\}_{n\in \NN, x\in \RR}$, when restricted to the interval $(a,b)$, has the conditional measure $\mathbb{P}^{1,1,(a,b), 2^{-\frac{1}{3}}\Upsilon^{(1)}_{T}(a), 2^{-\frac{1}{3}}\Upsilon^{(1)}_T(b), +\infty, 2^{-\frac{1}{3}}\Upsilon^{(2)}_T}_{\mathbf{H}_{2T}}$. On the other hand, replacing $2^{-\frac{1}{3}}\Upsilon^{(2)}_T$ by $-\infty$, \newline  $\mathbb{P}^{1,1,(a,b), 2^{-\frac{1}{3}}\Upsilon^{(1)}_{T}(a), 2^{-\frac{1}{3}}\Upsilon^{(1)}_T(b), +\infty, -\infty}_{\mathbf{H}_{2T}}$ is the probability measure of a Brownian bridge on the interval $(a,b)$ with the entrance and exit data $2^{-\frac{1}{3}}\Upsilon^{(1)}_{T}(a)$ and $2^{-\frac{1}{3}}\Upsilon^{(1)}_{T}(b)$ respectively. Lemma~\ref{Coupling1} constructs a coupling between these two measures on the curve $2^{-\frac{1}{3}}\Upsilon^{(1)}_T\big\vert_{(a,b)}$ such that
\begin{align}\label{eq:MeasureDom}
\mathbb{P}^{1,1,(a,b), 2^{-\frac{1}{3}}\Upsilon^{(1)}_{T}(a), 2^{-\frac{1}{3}}\Upsilon^{(1)}_T(b), +\infty, 2^{-\frac{1}{3}}\Upsilon^{(2)}_T}_{\mathbf{H}_{2T}}(\mathcal{A})\leq \mathbb{P}^{1,1,(a,b), 2^{-\frac{1}{3}}\Upsilon^{(1)}_{T}(a), 2^{-\frac{1}{3}}\Upsilon^{(1)}_T(b), +\infty, -\infty}_{\mathbf{H}_{2T}}(\mathcal{A})
\end{align}
for any event $\mathcal{A}$ whose chance increases\footnote{If increase is replaced by decrease, then, the inequality \eqref{eq:MeasureDom} is reversed.} under the pointwise decrease of $\Upsilon^{(1)}_T$.

In most of our applications of this idea, it is easy to find upper bounds on the r.h.s. of \eqref{eq:MeasureDom} using Brownian bridge calculations. Via \eqref{eq:MeasureDom}, those bounds transfers to the spatial process $\Upsilon^{(1)}_T(\cdot)$. Since, by Proposition \ref{NWtoLineEnsemble}, this curve is equal in law to $\Upsilon_T(\cdot)$ (the scaled and centered narrow wedge KPZ equation solution), these bounds in conjunction with the convolution formula of Proposition~\ref{Distribution} embodies the core of our techniques to generalize the tail bounds from narrow wedge to general initial data. The following lemma is used in controlling the probabilities which arise on r.h.s. of \eqref{eq:MeasureDom}.

\bl\label{BBFlucLem} Let $B(\cdot)$ be a Brownian bridge on $[0,L]$ with $B(0)= x$ and $B(L) = y$. Then,
\begin{align}\label{eq:BBineq}
\mathbb{P}\Big(\inf_{t\in [0,L]} B(t)\leq \min\{x,y\}-s\Big)\leq e^{- \frac{2s^2}{L}}.
\end{align}
\el
\begin{proof}
Due to symmetry, we may assume $\min\{x,y\} = y$. Note that $\tau=\min\{t\in [0,1]: B(t)\leq y\}$ is a stopping time for the natural filtration  of $B(\cdot)$. Thanks to the resampling invariance property of the Brownian bridge measure, $\{B(t)\}_{t\in [\tau,L]}$ conditioned on the sample paths outside the interval $(\tau,L)$ is again distributed as a Brownian bridge with $B(\tau)=B(L) =y$. Now, applying \cite[(3.40)]{KS91} (see also Lemma~2.11 of \cite{CorHam16}), we get
\begin{align}\label{eq:BBfluc}
\mathbb{P}\Big(\inf_{t\in (\tau, L] }B(t)\leq \min\{x,y\} -s \Big| \mathcal{F}([0,\tau])\Big)= e^{-\frac{2s^2}{(L-\tau)}}.
\end{align}
Here, $\mathcal{F}([0,\tau])$ denotes the natural filtration of $\{B\}_{t\in [0,L]}$ stopped at time $\tau$. Taking expectation of \eqref{eq:BBfluc} with respect to the $\sigma$-algebra $\mathcal{F}_\tau$ and noting  $e^{-\frac{2s^2}{(L-\tau)}}\leq e^{-\frac{2s^2}{L}}$ yields \eqref{eq:BBineq}.
\end{proof}
 It is worth noting that Proposition 4.3.5.3 of \cite{MMFM} contains an exact formulas for the left hand side of \eqref{eq:BBineq}.
The next result (which follows from \cite[(3.14)]{GT11}) is used in Theorem~\ref{Main6Theorem}.
\bl\label{BMminusParabola}
Let $B(\cdot)$ be a two-sided standard Brownian motion with $B(0) =0 $. Then, for any given $\xi\in (0,1)$, there exists $s_0=s_0(\xi)$ such that for all $c>0$ and $s\geq s_0$,
\begin{align}\label{eq:BMParaFluc}
\mathbb{P}\Big(B(t)\geq s+ct^2 \text{ for some }t\in \RR\Big)\leq \frac{1}{\sqrt{3}}e^{-\frac{8(1-\xi)\sqrt{c}s^{\frac{3}{2}}}{3\sqrt{3}}}.
\end{align}
\el

 \section{Lower tail under general initial data}\label{LowerTailSEC}

 In this section, we prove Theorems~\ref{Main1Theorem}  and \ref{Main3Theorem}.  Starting with the tail bounds of Proposition~\ref{NotMainTheorem}, we first bound the lower tail probabilities of the narrow wedge solution at a countable set of points of $\RR$ (see Lemma~\ref{MeshBound}). Combining this with the Brownian Gibbs property of the narrow wedge solution and the growth conditions of initial data (given in Definition~\ref{Hypothesis}), we prove the lower tail bound of Theorem~\ref{Main1Theorem} in Section~\ref{Proof1Theorem} via the convolution formula of Proposition~\ref{Distribution}. By controlling the fluctuations of a two sided Brownian motion in small intervals, we prove the lower tail bound of Theorem~\ref{Main3Theorem} (see Section~\ref{Proof2Theorem}) in a similar way.



\subsection{Proof of Theorem~\ref{Main1Theorem}}\label{Proof1Theorem}

Recall that the initial data  $\mathcal{H}_0$ is defined from $f$ via \eqref{eq:LowInitBd}. Also recall the definition of $\Upsilon_T(\cdot)$ from \eqref{eq:DefUpsilon}. Fix the sequence $\{\zeta_n\}_{n\in \ZZ}$ where $\zeta_n:=\frac{n}{s^{1+\delta}}$. Let us define the following events
\begin{align}
\mathcal{A}^{f} &:= \left\{ \int^{\infty}_{-\infty} e^{T^{\frac{1}{3}}\big(\Upsilon_T(y)+ f(-y)\big)} dy\leq e^{-T^{\frac{1}{3}}s}\right\},\label{eq:PrincipleEvent1}\\
E_n &:= \left\{\Upsilon_T(\zeta_n)\leq -\frac{(1+2^{-1}\nu)\zeta^2_n}{2^{2/3}}-(1-\epsilon)s\right\}, \label{eq:PrincipleEvent2}\\
F_n &:= \left\{\Upsilon_T(y)\leq -\frac{(1+\nu) y^2}{2^{2/3}}-\left(1-\frac{\epsilon}{2}\right)s \quad \text{ for some }y \in (\zeta_n, \zeta_{n+1})\right\}. \label{eq:PrincipleEvent3}
\end{align}
Here, we suppress the dependence on the various variables.
By \eqref{eq:DistUnderGI2} of Proposition~\ref{Distribution}, $\mathbb{P}(h^{f}_T(0)\leq -s) = \mathbb{P}(\mathcal{A}^{f})$
which we need to bound. To begin to bound this, note that
\begin{align}\label{eq:BasicStep}
\mathbb{P}(\mathcal{A}^{f}) &\leq \mathbb{P}\Big(\bigcup_{n\in \ZZ} E_n\Big) + \mathbb{P}\Big(\mathcal{A}^{f}\cap \Big(\bigcup_{n\in \ZZ} E_n\Big)^{c}\Big)\leq \sum_{n\in \ZZ}\mathbb{P}\big(E_n\big)+\mathbb{P}\Big(\mathcal{A}^{f}\cap \Big(\bigcup_{n\in \ZZ} E_n\Big)^{c}\Big).
\end{align}
We focus on bounding separately the two terms on the right side of \eqref{eq:BasicStep}.

\bl\label{MeshBound}
 There exist $s_0=s_0(\epsilon, \delta, C, \nu, T_0)$ and $K_{*}=K_{*}(\epsilon, \delta, T_0)>0$ such that for all $T\geq T_0$ and $s\geq s_0$,
\begin{equation}\label{eq:SumOfProbBd}
\sum^{\infty}_{n=-\infty}\mathbb{P}\left(E_n\right)\leq e^{-T^{1/3} \frac{4(1-\epsilon)s^{5/2}}{15\pi}} + e^{-K_{*}s^{3-\delta} - \epsilon s T^{1/3}}+e^{-\frac{(1-\epsilon)s^3}{12}}.
\end{equation}
\el

\begin{proof}
Recall that the one point distribution of $\Upsilon_T(y)+\frac{y^2}{2^{2/3}}$ is independent of $y$ (see Proposition~\ref{StationarityProp}). Setting $s_n:= (1-\epsilon)s+ \frac{\nu\zeta^2_n}{2^{5/3}}$ and invoking Propositions~\ref{StationarityProp} and \ref{NotMainTheorem}, we write
\begin{align}\label{eq:EnProb}
\mathbb{P}\left(E_n\right)= \mathbb{P}(\Upsilon_T(0)\leq -s_n)\leq e^{-T^{1/3}(1-\epsilon)\frac{4s^{5/2}_n}{15\pi}} + e^{-Ks^{3-\delta}_n - \epsilon s_n T^{1/3}}+e^{-(1-\epsilon)\frac{s^3_n}{12}}.
\end{align}
 Applying the reverse Minkowski inequality, we get $s^{\alpha}_n\geq ((1-\epsilon)s)^{\alpha} + (\nu n^2\kappa^2/2^{5/3}s^2)^{\alpha}$   for all $\alpha\geq 1$. Plugging this into \eqref{eq:EnProb} and
summing over all $n\in \ZZ$, we get
\begin{align}
\sum_{n\in \ZZ}\mathbb{P}(E_n)\leq &  e^{-T^{1/3}(1-\epsilon)\frac{4s^{5/2}}{15\pi}}\sum_{n\in \ZZ} e^{-T^{1/3}K_{1}\frac{|n|^5}{s^5}} +  e^{-(1-\epsilon)\frac{s^3}{12}} \sum_{n\in \ZZ} e^{-K_{2} \frac{n^6}{s^6}}\\& + e^{- K s^{3-\delta} - \epsilon s T^{1/3}} \sum_{n\in \ZZ} e^{ -K_3 \frac{|n|^{2(3-\delta)}}{s^{2(3-\delta)}}- \epsilon\frac{\nu}{2^{5/3}}\frac{n^2}{s^2} T^{1/3}}\label{eq:SumExpand}
\end{align}
for three positive constants $K_1$, $K_2$ and $K_3$. By a direct computation, we observe
\begin{align}
\sum_{n\in \ZZ} e^{-T^{1/3} K_1 s^{-5}|n|^{5}} \leq K^{\prime}_1 T^{-\frac{1}{3}}s^{5}, &\qquad \sum_{n\in \ZZ} e^{- K_2 \frac{n^6}{s^6}} \leq K^{\prime}_2 s^{6},\label{eq:SumBound1}\\  \sum_{n\in \ZZ} e^{ -K_3 \frac{|n|^{2(3-\delta)}}{s^{2(3-\delta)}}- \epsilon\frac{\nu}{2^{5/3}}\frac{n^2}{s^2} T^{\frac{1}{3}}} & \leq K^{\prime}_3 \Big(s^{3(2-\delta)}+s^2 T^{-\frac{1}{3}}\Big) .
\label{eq:SumBound2}
\end{align}
Combining \eqref{eq:SumBound1} and \eqref{eq:SumBound2} with \eqref{eq:SumExpand} yields \eqref{eq:SumOfProbBd}.
\end{proof}

  Now it suffices to control the second term on the right side of \eqref{eq:BasicStep}. We start by showing:
  \smallskip

\bl\label{EffOfInitCond}
Under the assumption that $f$ belongs to the class $\mathbf{Hyp}(C, \nu, \theta, \kappa, M)$, there exists $s_1=s_1(C,\nu, \theta, \kappa, M)$ such that for all $s\geq s_1$,
\begin{align}\label{eq:Containment}
\bigcap_{n\in \ZZ}\{E^c_n\cap F^c_n\} \subset (\mathcal{A}^{f})^c.
\end{align}
\el

\begin{proof}
Assume the events on the l.h.s. of \eqref{eq:Containment} occur. Appealing to \eqref{eq:LowInitBd}, we observe
\begin{align*}
\int^{\infty}_{-\infty} & e^{T^{1/3}\big(\Upsilon_T(y)+f(-y)\big)} dy
\geq \int_{\mathcal{I}}e^{-T^{1/3}\Big(\frac{(1+\nu /2) }{2^{2/3}}y^2+(1-\frac{\epsilon}{2})s- \kappa\Big)}dy  \geq \theta e^{-T^{1/3}\Big(\frac{1+\nu/2}{2^{2/3}}M^2+\kappa-\frac{\epsilon s}{2}\Big)}e^{-T^{\frac{1}{3}} s}.&
\end{align*}
Clearly, there exists $s_1 = s_1(C,\nu, \theta, \kappa, M)$ such that
 the right side above is bounded below by $e^{-T^{\frac{1}{3}} s}$ for all $s\geq s_1$. This shows the claimed containment of the events in \eqref{eq:Containment}.
\end{proof}

Owing to \eqref{eq:Containment} and then, Bonferroni's union bound,
\begin{equation}\label{eq:BonfrBd}
\mathbb{P}\Big(\mathcal{A}^{f}\cap \Big(\bigcup_{n\in \ZZ} E_n \Big)^{c}\Big) = \mathbb{P}\Big(\mathcal{A}^{f}\cap \Big\{\bigcap_{n\in \ZZ} E^{c}_n\Big\} \cap \Big\{\bigcup_{n\in \ZZ} F_n\Big\} \Big)\leq \sum_{n\in \ZZ}\mathbb{P}\left(E^{c}_n\cap E^c_{n+1} \cap F_n\right).
\end{equation}
We obtain an upper bound of the r.h.s. of \eqref{eq:BonfrBd} in the following lemma.
 \bl\label{LineEnsmbUse}
 There exists $s_2 =s_2(\epsilon)>0$ such that for all $s\geq s_2$
\begin{equation}\label{eq:ResidualEvent}
\sum_{n\in \ZZ}\mathbb{P}\left( E^{c}_n\cap E^c_{n+1} \cap F_n\right)\leq e^{-s^{3+\delta}}.
\end{equation}
\el

Combining \eqref{eq:BonfrBd} with \eqref{eq:ResidualEvent}  of Lemma~\ref{LineEnsmbUse} yields
 \begin{align}\label{eq:UnionProbBd}
 \mathbb{P}\Big(\mathcal{A}^{f}\cap \Big(\bigcup_{n\in \ZZ} E_n \Big)^{c}\Big)\leq  e^{-s^{3+\delta}}
 \end{align}
 for some $\delta>0$. Plugging the bounds \eqref{eq:SumOfProbBd} and \eqref{eq:UnionProbBd} into the r.h.s. of \eqref{eq:BasicStep} yields \eqref{eq:UpperBoundDetData}. To complete the proof of Theorem~\ref{Main1Theorem}, it only remains to prove Lemma~\ref{LineEnsmbUse} which we show below.

\begin{proof}[Proof of Lemma~\ref{LineEnsmbUse}]  We aim to bound $\mathbb{P}(E^{c}_n \cap E^{c}_{n+1}\cap F_n)$. By Proposition~\ref{NWtoLineEnsemble}, $\Upsilon_T$ equals in law the curve $\Upsilon^{(1)}_{T}$ of the scaled KPZ line ensemble $\{2^{-\frac{1}{3}}\Upsilon^{(n)}_T(x)\}_{n\in \NN,x\in \RR}$. Hence, without loss of generality, we replace $\Upsilon_T$ by $\Upsilon^{(1)}_T$ in the definitions of $E_n$ and $F_n$ for the rest of this proof. By the $\mathbf{H}_{2T}$-Brownian Gibbs property of $\{2^{-\frac{1}{3}}\Upsilon^{(n)}_T(x)\}_{n\in \NN,x\in \RR}$,
\begin{align}
\mathbb{P}(E^c_n &\cap E^c_{n+1}\cap F_n)= \mathbb{E}\Big[ \mathbbm{1}(E^c_n\cap E^c_{n+1})\cdot \mathbb{E}\big[\mathbbm{1}(F_n)|\mathcal{F}_{\mathrm{ext}}(\{1\}, (\zeta_n,\zeta_{n+1}))\big]\Big]\\&= \mathbb{E}\Big[\mathbbm{1}(E^c_n\cap E^c_{n+1})\cdot \mathbb{P}^{1,1,(\zeta_n, \zeta_{n+1}), 2^{-\frac{1}{3}}\Upsilon^{(1)}_T(\zeta_n), 2^{-\frac{1}{3}}\Upsilon^{(1)}_T(\zeta_{n+1}), +\infty, 2^{-\frac{1}{3}}\Upsilon^{(2)}_T}_{\mathbf{H}_{2T}}(F_n)\Big].\label{eq:LineEnsemble}
\end{align}
Recall $\mathcal{F}_{\mathrm{ext}}(\{1\}, (\zeta_n,\zeta_{n+1}))$ is the $\sigma$-algebra generated by $\{\Upsilon^{(n)}_T(x)\}_{n\in \NN,x\in \RR}$ outside the set $\{\Upsilon^{(1)}_T(x):x\in (\zeta_n, \zeta_{n+1})\}$.
Via Proposition~\ref{Coupling1}, there exists a monotone coupling between the probability measures $\mathbb{P}_{\mathbf{H}_{2T}}:=\mathbb{P}^{1,1,(\zeta_n, \zeta_{n+1}), 2^{-\frac{1}{3}}\Upsilon^{(1)}_T(\zeta_n), 2^{-\frac{1}{3}}\Upsilon^{(1)}_T(\zeta_{n+1}), +\infty, 2^{-\frac{1}{3}}\Upsilon^{(2)}_T}_{\mathbf{H}_{2T}}$ and $\widetilde{\mathbb{P}}_{\mathbf{H}_{2T}}:=\mathbb{P}^{1,1,(\zeta_n, \zeta_{n+1}), 2^{-\frac{1}{3}}\Upsilon^{(1)}_T(\zeta_n), 2^{-\frac{1}{3}}\Upsilon^{(1)}_T(\zeta_{n+1}), +\infty, -\infty}_{\mathbf{H}_{2T}}= \mathbb{P}^{1,1, (\zeta_n, \zeta_{n+1}), 2^{-\frac{1}{3}}\Upsilon^{(1)}_{T}(\zeta_n), 2^{-\frac{1}{3}}\Upsilon^{(1)}_T(\zeta_{n+1})}_{\mathrm{free}}$ such that
\begin{align}\label{eq:MeasureDominition}
\mathbb{P}_{\mathbf{H}_{2T}}(F_n)\leq \widetilde{\mathbb{P}}_{\mathbf{H}_{2T}}(F_n).
\end{align}
The r.h.s. of \eqref{eq:MeasureDominition} is a probability with respect a Brownian bridge measure.
For the rest of the proof, we use shorthand notation $\theta_n:= (1-\epsilon)s+2^{-\frac{2}{3}}(1+2^{-1}\nu)\zeta^2_n$ for $ n\in \ZZ$.
The probability of the event $F_n$ increases under the pointwise decrease of the end points of $\Upsilon^{(1)}_T$.
Using $\{E^c_n\cap E^c_{n+1}\}= \{\Upsilon^{(1)}_T(\zeta_n)\geq -\theta_n\}\cap \{\Upsilon^{(1)}_T(\zeta_{n+1})\geq -\theta_{n+1}\}$ and Proposition~\ref{NWtoLineEnsemble},
\begin{align}
\mathbbm{1}(E^{c}_n\cap E^c_{n+1})\times \widetilde{\mathbb{P}}_{\mathbf{H}_{2T}}(F_n)\leq \mathbb{P}^{1,1,(\zeta_n, \zeta_{n+1}), -2^{-\frac{1}{3}}\theta_n, -2^{-\frac{1}{3}}\theta_{n+1}}_{\mathrm{free}}(F_n). \label{eq:OrderedProb}
\end{align}
 Combining \eqref{eq:MeasureDominition} and \eqref{eq:OrderedProb} yields
\begin{align}
\mathbbm{1}(E^{c}_n\cap E^c_{n+1})\times &\mathbb{P}^{1,1,(\zeta_n, \zeta_{n+1}), 2^{-\frac{1}{3}}\Upsilon^{(1)}_T(\zeta_n), 2^{-\frac{1}{3}}\Upsilon^{(1)}_T(\zeta_{n+1}), +\infty, 2^{-\frac{1}{3}}\Upsilon^{(2)}_T}_{\mathbf{H}_{2T}}(F_n)\\&\leq \mathbb{P}\Big(\min_{x\in [\zeta_n,\zeta_{n+1}]}B(t)\leq 2^{-\frac{1}{3}}\{\theta_n\wedge \theta_{n+1}\}-\frac{\epsilon s}{2^{4/3}}-\frac{\nu \zeta^2_n}{4}\Big)\label{eq:BrownianDominition}
\end{align}
where $B(\cdot)$ is a Brownian bridge such that $B(\zeta_n)= -2^{-\frac{1}{3}}\theta_{n}$ and $B(\zeta_{n+1}) = -2^{-\frac{1}{3}}\theta_{n+1}$. Applying Lemma~\ref{BBFlucLem} yields $\text{r.h.s. of \eqref{eq:BrownianDominition}}\leq e^{-2^{1/3}s^{1+\delta}\big( \frac{\epsilon s}{2^{4/3}}+\frac{\nu \zeta^2_n}{4}\big)^2}$.
Combining this upper bound with \eqref{eq:BrownianDominition} and taking the expectations, we arrive at
\begin{align}\label{eq:FinalUpperBound}
\mathbb{P}(E^c_n\cap E^c_{n+1}\cap F_n)\leq e^{-2^{1/3}s^{1+\delta}\big( \frac{\epsilon s}{2^{4/3}}+ \frac{\nu n^2}{4 s^{2(1+\delta)}}\big)^2}.
\end{align}
Summing both side of \eqref{eq:FinalUpperBound} over $n\in \ZZ$, we obtain \eqref{eq:ResidualEvent}.
\end{proof}

\subsection{Proof of Theorem~\ref{Main3Theorem}}\label{Proof2Theorem}

This proof is similar to that of Theorem~\ref{Main1Theorem}. We use the same notations $\zeta_n$, $E_n$ and $F_n$ introduced in the beginning of the proof of Theorem~\ref{Main1Theorem} and additionally define
\begin{align}\label{eq:ABr}
\mathcal{A}^{\mathrm{Br}}= \left\{\int^{\infty}_{-\infty} e^{T^{1/3}\big(\Upsilon_T(y)+ B(-y)\big)} dy \leq e^{-T^{\frac{1}{3}}s}\right\}
\end{align}
where $B$ is a two sided Brownian motion with diffusion coefficient $2^{\frac{1}{3}}$ and $B(0)=0$. In particular, $B(y)\stackrel{d}{=} \widetilde{B}(2^{\frac{2}{3}}y)$ where $\widetilde{B}(\cdot)$ is standard two sided Brownian motion. Owing to \eqref{eq:DistUnderGI2}, $\mathbb{P}(h^{\mathrm{Br}}(0)\leq -s)= \mathbb{P}(\mathcal{B}^{\mathrm{Br}})$ which we need to bound. As in \eqref{eq:BasicStep}, we write
\begin{align}\label{eq:BasicStep2}
\mathbb{P}\big(\mathcal{A}^{\mathrm{Br}}\big)\leq \sum_{n\in \ZZ} \mathbb{P}(E_n)+\mathbb{P}\Big(\mathcal{A}^{\mathrm{Br}} \cap \Big(\bigcap_{n\in \ZZ} E_n\Big)^{c}\Big).
\end{align}
 We can use \eqref{eq:SumOfProbBd} of Lemma~\ref{MeshBound} to bound $\sum_{n}\mathbb{P}(E_n)$. While the conclusion of Lemma~\ref{EffOfInitCond} does not hold in the present case, we will show that it does hold with high probability.

 \bl\label{BrowLowLemma}
There exist $s_1=s_1(\epsilon, \delta)$, $c_1=c_1(\epsilon),c_2=c_2(\epsilon)>0$ such that for all $s\geq s_1$,
\begin{align}\label{eq:SmallProb}
\mathbb{P}\Big(\bigcap_{n\in \ZZ} \big\{E^c_n \cap F^c_n\big\}\cap \mathcal{A}^{\mathrm{Br}}\Big)\leq c_1e^{-c_2s^{3+\delta}}.
\end{align}
\el

Combining \eqref{eq:ResidualEvent} of Lemma~\ref{LineEnsmbUse} and \eqref{eq:SmallProb} of Lemma~\ref{BrowLowLemma} yields
\begin{align}\label{eq:ResEvent}
\mathbb{P}\left(\mathcal{A}^{\mathrm{Br}}\cap \Big(\bigcup_{n\in \ZZ} E_n\Big)^c\right)\leq c_2e^{-c_1s^{3+\delta}}.
\end{align}
Applying \eqref{eq:ResEvent} and \eqref{eq:SumOfProbBd} to \eqref{eq:BasicStep2}, we obtain \eqref{eq:UpBoundBrData}. To complete the proof of Theorem~\ref{Main3Theorem}, we now need to prove Lemma~\ref{BrowLowLemma} which is given as follows.

\begin{proof}[Proof of Lemma~\ref{BrowLowLemma}] Observe first that
\begin{align}
\bigcap_{n\in \ZZ} \big\{E^c_n \cap F^c_n\big\}\cap \mathcal{A}^{\mathrm{Br}} &\subseteq \Big\{ \int^{\infty}_{-\infty} e^{-T^{1/3}\big(\frac{(1+\nu)y^2}{2^{2/3}}-\frac{\epsilon s}{2}-B_2(y) \big)}dy\leq 1\Big\}.\label{eq:ContainmentProb}
\end{align}
Note that if $B(y)\geq -\frac{\epsilon}{4}s$ for all $y\in [-1/s^{1+\delta}, 1/s^{1+\delta}]$, then,
$\frac{(1+\nu)y^2}{2^{2/3}}-\frac{\epsilon s}{2}-B(y)\leq -\frac{\epsilon}{8}s$ for all $y \in [-1/s^{1+\delta}, 1/s^{1+\delta}]$
which implies
\begin{align}
\int_{-\infty}^{\infty}  e^{-T^{1/3}\big(\frac{(1+\nu)y^2}{2^{2/3}}-\frac{\epsilon s}{2}-B(y) \big)}dy \geq \frac{2}{s^{1+\delta}}e^{\frac{\epsilon s}{8}T^{1/3}}> 1
\end{align}
when $s$ is large. Hence, there exists $s_1=s_1(\epsilon, \delta)$ such that for all $s\geq s_1$, one has
\begin{align}
\left\{\int^{\infty}_{-\infty} e^{-T^{1/3}\big(\frac{(1+\nu)y^2}{2^{2/3}}-\frac{\epsilon s}{2}-B(y) \big)}dy \leq 1\right\}\subseteq \left\{ \min_{y\in [-1/s^{1+\delta}, 1/s^{1+\delta}]}B(y)< -\frac{\epsilon}{4}s\right\}.
\end{align}
Thanks to this containment, we get
\begin{equation}\label{eq:TailOfBPlusPar}
\mathbb{P}\Big(\bigcap_{n\in \ZZ} \big\{E^c_n \cap F^c_n\big\}\cap \mathcal{A}^{\mathrm{Br}}\Big)\leq \mathbb{P}\Big(\min_{y\in [-\frac{1}{s^{1+\delta}}, \frac{1}{s^{1+\delta}}]}B(y)<-\frac{\epsilon}{4}s\Big).
\end{equation}
We bound the r.h.s. of \eqref{eq:TailOfBPlusPar}, via the reflection principle as
\begin{align}\label{eq:RefPrinc}
 \mathbb{P}\Big(\min_{y\in [-\frac{1}{s^{1+\delta}}, \frac{1}{s^{1+\delta}}]}B(y)\leq -\frac{\epsilon}{4}s\Big)\leq \mathbb{P}\Big(2|X_1| +2 |X_2|\geq \frac{\epsilon}{4}s\Big)
\end{align}
  where $X_1, X_2$ are independent Gaussians with variance $2^{\frac{1}{3}}s^{-(1+\delta)}$. By tail estimates, it follows that the r.h.s. of \eqref{eq:RefPrinc} is bounded above by $c_1e^{-c_2 s^{3+\delta}}$ for some constants $c_1, c_2>0$ which only depend on $\epsilon$.
Plugging this into \eqref{eq:TailOfBPlusPar} and combining with \eqref{eq:ContainmentProb}, we find \eqref{eq:SmallProb}.
\end{proof}

\section{Upper Tail under narrow wedge initial data}\label{UpperTailNSEC}

 The aim of this section is to prove Theorem~\ref{GrandUpTheorem}. To achieve this, we first state a few auxiliary results which combine together to prove Theorem~\ref{GrandUpTheorem}. These auxiliary results are proved in the end of Section~\ref{UpperTailNSEC}. Recall the definition of $\Upsilon_T$ from \eqref{eq:DefUpsilon}. Our first result of this section (Proposition~\ref{thm:UpTail}) gives an upper and lower bound for the probability $\mathbb{P}(\Upsilon_T(0)\geq s)$. These bounds are close to optimal when $s\gg T^{\frac{2}{3}}$. When $s=O(T^{\frac{2}{3}})$ or $s\ll T^{\frac{2}{3}}$, those bounds are not optimal (see Remark~\ref{UpperTailNewRemark}). In those cases, we obtain better bounds using Proposition~\ref{UpperDeepTails}.

\bp\label{thm:UpTail}
Fix some $\zeta\leq \epsilon\in (0, 1)$ and $T_0>0$. There exists $s_0=s_0(\epsilon, \zeta, T_0)$ such that for all $s\geq s_0$ and $T\geq T_0$,
\begin{align}
\mathbb{P}\left(\Upsilon_T(0)> s\right)&\leq e^{-T^{1/3} \zeta s}+ e^{- \frac{4}{3} (1-\epsilon)s^{3/2}},\label{eq:UpBoundBd}\\
1-\exp\big(-e^{-\zeta s T^{1/3}}\big)\mathbb{P}\left(\Upsilon_T(0)\leq s\right)&\geq e^{-T^{1/3}(1+\zeta) s}+ e^{- \frac{4}{3}(1+\epsilon) s^{3/2}}.\label{eq:LowBound}
\end{align}
\ep

\begin{rem}\label{UpperTailNewRemark}
Proposition~\ref{thm:UpTail} implies that for $s\ll T^{\frac{2}{3}}$
\begin{align}\label{eq:Me}
 \exp\Big(- \frac{4}{3}(1+\epsilon)s^{\frac{3}{2}}\Big) \leq \mathbb{P}\big(\Upsilon_T(0)>s\big) \leq \exp\Big(- \frac{4}{3}(1-\epsilon)s^{\frac{3}{2}}\Big).
\end{align}
To see this, we first note that
\begin{align}
\text{r.h.s. of \eqref{eq:UpBoundBd}} \leq
\exp\Big(- \frac{4}{3}(1-\epsilon)s^{\frac{3}{2}}\Big), & \qquad \text{when }s\ll T^{\frac{2}{3}}.\label{eq:UpTailPhases}
\end{align}
Using the approximation $1-\exp\big(-e^{-\zeta s T^{1/3}}\big)\approx \exp(-\zeta s T^{1/3})$, we see that \eqref{eq:LowBound} implies
\begin{align}\label{eq:LowBdCons}
 \mathbb{P}\big(\Upsilon_T(0)> s\big)\geq \exp\big(e^{-\zeta s T^{\frac{1}{3}}}\big)\Big(e^{-(1+\zeta)s T^{\frac{1}{3}}} - e^{-\zeta s T^{\frac{1}{3}}} + e^{-\frac{4}{3}(1+\epsilon)s^{\frac{3}{2}}}\Big).
\end{align}
The r.h.s. of \eqref{eq:LowBdCons} is bounded below by $\exp(- \frac{4}{3}(1+\epsilon)s^{\frac{3}{2}})$ when $s\ll T^{\frac{2}{3}}$.
 Note, when $s\gg T^{\frac{2}{3}}$, the dominating term of the r.h.s. of \eqref{eq:UpBoundBd} is $\exp(-\zeta s T^{1/3})$ which we show in our next theorem is the not correct order of decay of $\mathbb{P}(\Upsilon_T(0)>s)$.

\end{rem}



 \bp\label{UpperDeepTails}
  Fix $\epsilon\in (0,1)$. Then, for all pairs $(s,T)$ satisfying $s\geq \frac{9}{16}\epsilon^{-2}T^{\frac{2}{3}}$ and $T>\pi$,
 \begin{align}
 \mathbb{P}(\Upsilon_T(0)> s) &\leq e^{- \frac{4(1-\epsilon)}{3}s^{3/2}} \label{eq:TightUpBd}\\
 \mathbb{P}(\Upsilon_T(0) > s) & \geq e^{-4\sqrt{3}(1+3\epsilon)s^{3/2}}\label{eq:TightLowBd}
 \end{align}
 Furthermore, for all $s\in \big[\frac{1}{8}\epsilon^2 T^{\frac{2}{3}}, \frac{9}{16} \epsilon^{-2} T^{\frac{2}{3}}\big]$,
 \begin{align}\label{eq:LowBdAllWay}
 \mathbb{P}\big(\Upsilon_T(0)>s\big)\geq \frac{1}{2}e^{-2^{7/2}\epsilon^{-3}s^{3/2}}.
\end{align}
Moreover, for any $0<T_0\leq \pi$ and $\epsilon \in (0,3/5)$, there exist $c_1=c_1(T_0)>c_2=c_2(T_0)>0$ such that for all $T\in [T_0,\pi]$ and $s\geq \frac{9}{16}\epsilon^{-2}T^{\frac{2}{3}} + 24T^{-\frac{1}{3}}_0(1-\epsilon)^{-1}|\log (T_0/\pi)|$,
\begin{align}\label{eq:UpLowRough}
e^{-c_1s^{3/2}}\leq \mathbb{P}(\Upsilon_T(0)>s) \leq e^{-c_2s^{3/2}}.
\end{align}

 \ep

%

\bp\label{UpperDeepTailLemma}
 Fix $\epsilon\in (0,1)$, $T>\pi$ and $c>\frac{4}{3}\big(1+\frac{1}{3}\epsilon\big)$. Then, there exists $\{s_n\}_{n}$ such that $s_n \to \infty$ as $n\to \infty$ and $\mathbb{P}(\Upsilon_T(0)>s_n)\geq e^{-cs^{3/2}_n}$ for all $n\in \NN$.
\ep

\subsection{Proof of Theorem~\ref{GrandUpTheorem}}

We first show \eqref{eq:RoughBd1} when $T_0 \in (0, \pi)$. Fix $\epsilon \in (0,\frac{3}{4})$ and define $s_0= \frac{9}{16}\epsilon^{-2} \pi^{\frac{2}{3}}+ 3(1-\epsilon)^{-1} T^{\frac{1}{3}}_0|\log T_0|$.  Then, for all $T\in [T_0, \pi]$ and $s\geq s_0$, \eqref{eq:RoughBd1} follows from \eqref{eq:UpLowRough}.
\vspace{0.2cm}

Now, we show \eqref{eq:RoughBd1} for $T_0>\pi$. Fix $\zeta = \epsilon\in \big(0,\frac{1}{2}\big)$. Proposition~\ref{thm:UpTail} says that there exists $s_0=s_0(\epsilon, T_0)$ such that \eqref{eq:UpBoundBd} and \eqref{eq:LowBound} holds for all $s\geq s_0$ and $T>T_0$.

\noindent (i) For all $s\in (0, \frac{1}{8}\epsilon^{2} T^{\frac{2}{3}})$, we note
\begin{equation}\label{eq:obs1}
 \frac{4}{3}(1+\epsilon)s^{\frac{3}{2}}\leq 2 s^{\frac{3}{2}} \leq \frac{1}{\sqrt{2}}\epsilon sT^{\frac{1}{3}}
\end{equation}
where the first and second inequalities follow from $\epsilon\leq \frac{1}{2}$ and $s\leq \frac{1}{8}\epsilon^2 T^{\frac{2}{3}}$ respectively. Furthermore, there exists $s^{\prime}_0 = s^{\prime}_{0}(\epsilon, T_0)$ such that for all $s\geq s^{\prime}_0$, one has
\begin{align}\label{eq:obs2}
\exp\big(-\frac{1}{\sqrt{2}}\epsilon s T^{\frac{1}{3}}\big)\geq 2\exp\big( -\epsilon s T^{\frac{1}{3}}\big).
\end{align}
Combining \eqref{eq:obs1} and \eqref{eq:obs2} yields
 \begin{align}\label{eq:obs3}
 \exp(-\frac{4}{3}(1+\epsilon)s^{\frac{3}{2}}) \geq 2 \exp(- \epsilon s T^{\frac{1}{3}}), \qquad \forall s\in  (s^{\prime}_0, \frac{1}{8}\epsilon^2 T^{\frac{2}{3}}).
 \end{align}
Plugging this into the r.h.s. of \eqref{eq:UpBoundBd} yields
\begin{align}\label{eq:NewUpTBd1}
\mathbb{P}(\Upsilon_T(0)> s)\leq 2\exp\big(-\frac{4}{3}(1-\epsilon)s^{\frac{3}{2}}\big)
\end{align}
for all $s\in (\max\{s_0, s^{\prime}_0\}, \frac{1}{8}\epsilon^2 T^{\frac{2}{3}})$ where $s_0=s_0(\epsilon, T_0)$ comes with Proposition~\ref{thm:UpTail}. Moreover, applying \eqref{eq:obs3} in \eqref{eq:LowBdCons}, we observe
\begin{align}\label{eq:NewUpTBd2}
\mathbb{P}(\Upsilon_T(0)>s)\leq  \frac{1}{2}\exp\big(-\frac{4}{3}(1+\epsilon)s^{3/2}\big).
\end{align}
Combining \eqref{eq:NewUpTBd1} and \eqref{eq:NewUpTBd2}, we obtain \eqref{eq:RoughBd1} with $c_1\leq \frac{4}{3}(1+\epsilon)$ and $c_2 \geq \frac{4}{3}(1-\epsilon)$ for all $s\in (s^{\prime\prime}_0,\frac{1}{8}\epsilon^2 T^{\frac{2}{3}})$ for some $s^{\prime\prime}_0=s^{\prime\prime}_0(\epsilon, T_0)$.
  \smallskip

\noindent (ii)  When $s\geq \frac{9}{16}\epsilon^{-2} T^{\frac{2}{3}}$, we first apply Proposition~\ref{UpperDeepTails}. Using \eqref{eq:TightUpBd} and \eqref{eq:TightLowBd}, yields \eqref{eq:RoughBd1} with $c_1\leq  4\sqrt{3}(1+\epsilon)$ and $c_2\geq  \frac{4}{3}(1-\epsilon)$. The second part of the claim follows from  Proposition~\ref{UpperDeepTailLemma}.
  \smallskip

\noindent (iii)  For all $s\in (\frac{1}{8}\epsilon^2 T^{\frac{2}{3}}, \frac{9}{16}\epsilon^{-2}T^{\frac{2}{3}})$, appealing to \eqref{eq:LowBdAllWay} of Lemma~\ref{UpperDeepTails}, we get $c_1\leq 2^{\frac{7}{2}}\epsilon^{-3}$. Furthermore, one has the following bound on the r.h.s. of \eqref{eq:UpBoundBd}
\begin{align}\label{eq:Interm}
\exp\Big(-\epsilon s T^{\frac{1}{3}}\Big) + \exp\Big(-\frac{4}{3}(1-\epsilon)s^{\frac{3}{2}}\Big)\leq 2\exp\Big(-\min\big\{\epsilon s T^{\frac{1}{3}}, \frac{4}{3}(1-\epsilon)s^{\frac{3}{2}}\big\}\Big). &&
\end{align}
For all $\epsilon\leq \frac{1}{2}$ and $s\in (\frac{1}{8}\epsilon^2 T^{\frac{2}{3}}, \frac{9}{16}\epsilon^{-2}T^{\frac{2}{3}})$, the r.h.s. of \eqref{eq:Interm} is bounded above by $\exp(-\frac{4}{3}\epsilon s^{\frac{3}{2}})$.
Plugging this bound into \eqref{eq:TightLowBd}, we get
\[\mathbb{P}(\Upsilon_T(0)>s)\leq 2 e^{-\frac{4}{3}\epsilon s^{3/2}}, \quad \forall s\in \Big(\max\{s_0, \frac{1}{8}\epsilon^2 T^{\frac{2}{3}}\},\max\{s_0, \frac{9}{16}\epsilon^{-2} T^{\frac{2}{3}}\} \Big).\]
  Therefore, \eqref{eq:RoughBd1} holds when $s$ lies in the interval $(\max\{s_0, \frac{1}{8}\epsilon^2 T^{\frac{2}{3}}\},\max\{s_0, \frac{9}{16}\epsilon^{-2} T^{\frac{2}{3}}\} )$ with $c_1\leq  2^{\frac{7}{2}}\epsilon^{-3}$ and $c_2\geq \frac{4}{3}\epsilon$. This completes the proof of Theorem~\ref{GrandUpTheorem}.

 \subsection{Proof of Proposition~\ref{UpperDeepTails}}

 To prove Proposition~\ref{UpperDeepTails}, we need the following lemma. Let $$\psi_T(k)=
 \begin{cases}
 \frac{k!e^{\frac{Tk^3}{12}}}{2\sqrt{\pi T} k^{\frac{3}{2}}}, & \text{when }T\geq \pi\\
 \frac{\pi^{(k-1)/2}k!e^{\frac{Tk^3}{12}}}{2T^{k/2} k^{\frac{3}{2}}} & \text{when }T<\pi.
 \end{cases} $$

 \bl\label{MomBoundLem} Fix $k\in \NN$ and $T_0\in \RR_{+}$. Then, we have
\begin{align}\label{eq:MomBound}
C\psi_T(k)\leq \mathbb{E}\big[\exp\big(kT^{\frac{1}{3}}\Upsilon_T(0)\big)\big] \leq 69 \psi_T(k)
\end{align}
where $C=C(k,T_0)>0$ is bounded below by $1$ for all $T>T_0>\pi$ and by $ T^{(k-1)/2}_0 \pi^{-k/2}$ for all $T\in [T_0, \pi]$.
\el

\begin{proof}
Recall that $\mathcal{Z}(2T,0)= \exp(T^{\frac{1}{3}}\Upsilon_T(0)-\frac{T}{12})$. The moments of $\mathcal{Z}(2T, 0)$ are given by\footnote{These formulas were formally derived in \cite{BorodinCorwinMac} with a proof given as \cite[Theorem 2.1]{Ghosal18}.}:
\begin{align}
\mathbb{E}\Big[\frac{\exp(kT^{\frac{1}{3}}\Upsilon_T(0))}{k!}
\Big]& = \sum_{\substack{\lambda\vdash k\\ \lambda = 1^{m_1} 2^{m_2}\ldots }} \frac{1}{m_1!m_2!\ldots }\int^{\mathbf{i}\infty}_{-\mathbf{i}\infty} \frac{dw_1}{2\pi \mathbf{i}} \ldots \int_{-\mathbf{i}\infty}^{\mathbf{i}\infty} \frac{dw_{\ell(\lambda)}}{2\pi \mathbf{i}} \mathrm{det}\left[\frac{1}{w_j+\lambda_j- w_i}\right]^{\ell(\lambda)}_{i,j=1}\\
&\times  \exp\left[T \sum_{j=1}^{\ell(\lambda)}\Big(\frac{\lambda^3_j}{12} + \lambda_j\Big(w_j+\frac{\lambda_j}{2}-\frac{1}{2}\Big)^2\Big)\right].\label{eq:MomContFor}
\end{align}
Here, $\lambda\vdash k$ denotes that $\lambda=(\lambda_1\geq \lambda_2\geq \ldots )$ partitions $k$, $\ell(\lambda)=\#\{i:\lambda_i>0\}$ and $m_j = \#\{i:\lambda_i=j\}$.
By Cauchy's determinant formula,
\begin{align}\label{eq:Cauchy}
\mathrm{det}\Big[\frac{1}{w_i+\lambda_i-w_j}\Big] = \prod_{i=1}^{\ell(\lambda)} \frac{1}{\lambda_i}\prod^{\ell(\lambda)}_{i< j} \frac{\big(w_i-w_j+\lambda_i-\lambda_j\big)\big(w_j-w_i\big)}{\big(w_i+\lambda_i-w_j\big)\big(w_j+\lambda_j-w_i\big)}.
\end{align}
Applying \eqref{eq:Cauchy} to \eqref{eq:MomContFor} followed by substituting $\mathbf{i}z_j = T^{\frac{1}{3}}(w_j + \frac{\lambda_j}{2}-\frac{1}{2})$ in \eqref{eq:MomContFor} and deforming the contours to the real axis (note that no pole will be crossed) implies that
     \begin{align*}
    \text{r.h.s. of \eqref{eq:MomContFor}} &=\!\!\!\!\!\! \sum_{\substack{\lambda\vdash k\\ \lambda = 1^{m_1} 2^{m_2}\ldots }} \!\!\!\!\!\!\frac{\prod_{i=1}^{\ell(\lambda)}e^{ \frac{T\lambda^3_i}{12}}/2\pi}{m_1!m_2!\ldots } \int^{\infty}_{-\infty}\!\!\!\!\ldots \int^{\infty}_{-\infty}  \prod_{i=1}^{\ell(\lambda)} \frac{dz_ie^{-T^{\frac{1}{3}}\lambda_i z^2_i }}{T^{\frac{1}{3}}\lambda_i}\prod^{\ell(\lambda)}_{i< j} \frac{\frac{T^{\frac{2}{3}}(\lambda_i-\lambda_j)^2}{4}+ (z_i-z_j)^2}{\frac{T^{\frac{2}{3}}(\lambda_i+\lambda_j)^2}{4} + (z_i-z_j)^2}
     \end{align*}
     Taking $\lambda=(k)$ (i.e., $\lambda_1 =k$ and $\lambda_i=0$ for all $i\geq 2$), evaluating the single integral and noting that all the terms on the r.h.s. above are positive yields the lower bound in \eqref{eq:MomBound} when $T_0>\pi$. 
   In the case when $T_0<\pi$, the term corresponding to $\lambda =(k)$ is bounded below by $T^{(k_0-1)/2}_0\pi^{k/2}\psi_T(k)$ for all $T\in [T_0,\pi]$. This yields the lower bound in \eqref{eq:MomBound} when $T_0<\pi$.    
     

     For the upper bound, we first show that if $\lambda$ is a partition of $k$ not equal to $(k)$ then
     \begin{align}\label{eq:PartitionBd}
     \frac{k^3}{12} -\sum_{j=1}^{\ell(\lambda)} \frac{\lambda^3_j}{12}\geq \frac{k^2-k}{4}
     \end{align}
   with equality only when $\lambda=(k-1,1)$. We prove this by induction. It is straightforward to check that \eqref{eq:PartitionBd} holds when $k=1,2$. Assume \eqref{eq:PartitionBd} holds when $k=k_0-1$. Now we show it for $k=k_0$. Let us assume that $\lambda$ is a partition of $k_0$ and write
\[\frac{k^3_0}{12} -\sum_{j=1}^{\ell(\lambda)} \frac{\lambda^3_j}{12}=\frac{k^3_0}{12} -\frac{(k_0-1)^3+1}{12}+ \frac{(k_0-1)^3+1}{12}-\sum_{j=1}^{\ell(\lambda)} \frac{\lambda^3_j}{12}. \]
The right hand side of the above display is equal to $\frac{k^3_0}{12} -\frac{(k_0-1)^3+1}{12}=\frac{k^2_0-k_0}{4}$ when $\lambda =(k_0-1,1)$. It suffices to show
\begin{align}\label{eq:PartitionBDinduc}
\frac{(k_0-1)^3+1}{12}-\sum_{j=1}^{\ell(\lambda)} \frac{\lambda^3_j}{12}\geq 0
\end{align}
when $\lambda\neq (k_0),(k_0-1,1)$. In the case when $\lambda_{\ell(\lambda)}=1$, the above inequality follows by our assumption since $(\lambda_1, \ldots , \lambda_{\ell(\lambda)-1})$ is a partition of $k_0-1$. For $\lambda_{\ell(\lambda)}>1$, we write
\begin{align}
 \frac{(k_0-1)^3+1}{12}-\sum_{j=1}^{\ell(\lambda)} \frac{\lambda^3_j}{12} &= \frac{(k_0-1)^3}{12}-\sum_{j=1}^{\ell(\lambda)-1} \frac{\lambda^3_j}{12}-\frac{(\lambda_{\ell(\lambda)}-1)^3}{12} -\frac{(\lambda_{\ell(\lambda)}-1)(\lambda_{\ell(\lambda)}-2)}{4}.
 \end{align}
Note that $(\lambda_1, \ldots , \lambda_{\ell(\lambda)}-1)$ is a partition of $k_0-1$. Since $\lambda_{\ell(\lambda)}<k_0$ and \eqref{eq:PartitionBd} holds for $k=k_0-1$, the right hand side of the above display is greater than $0$. This shows \eqref{eq:PartitionBDinduc} and hence, proves \eqref{eq:PartitionBd}.

      We return to the proof of the upper bound in \eqref{eq:MomBound}. Observe that by bounding the cross-product over $i<j$ by 1 and using Gaussian integrals, we may bound
 \begin{align}\label{eq:IntegralBd1}
   \int^{\infty}_{-\infty}\ldots \int^{\infty}_{-\infty} \prod_{i=1}^{\ell(\lambda)} \frac{dz_i  e^{- T^{\frac{1}{3}}\lambda_iz^2_i}}{T^{\frac{1}{3}}\lambda_i}\prod^{\ell(\lambda)}_{i<j} \frac{\frac{T^\frac{2}{3}(\lambda_i-\lambda_j)^2}{4} + (z_i-z_j)^2}{\frac{T^{\frac{2}{3}}(\lambda_i +\lambda_j)^2}{4} + (z_i-z_j)^2} \leq \prod_{i=1}^{\ell(\lambda)} \frac{\sqrt{2\pi}}{\sqrt{2T}\lambda_i^{\frac{3}{2}}}
   \end{align}
When $T>\pi$, the r.h.s. of \eqref{eq:IntegralBd1} $\leq 1$. Otherwise, the r.h.s. of \eqref{eq:IntegralBd1} is bounded above by $(\pi/T)^{k/2}$.
Owing to this, \eqref{eq:PartitionBd}, and $m_1!m_2!\ldots \leq k!$, we get
   \begin{align}
   \mathbb{E}\big[\exp(kT^{\frac{1}{3}}\Upsilon_T(0))\big] &\leq \Big(1+ k^{\frac{3}{2}}e^{-\frac{k^2-k}{4}}\#\{\lambda: \lambda\vdash k\} \Big)\times\begin{cases}
   \frac{k!e^{\frac{k^3T}{12}}}{2\sqrt{\pi T} k^{\frac{3}{2}}} & T\geq \pi\\
 \frac{\pi^{(k-1)/2}k!e^{\frac{k^3T}{12}}}{2T^{k/2} k^{\frac{3}{2}}} & T<\pi.
\end{cases}
    \label{eq:MomUppBd}
   \end{align}
     Applying Siegel's bound (see \cite[pp. 316-318]{Apostol76}, \cite[pp. 88-90]{Knopp70}) on the number partition of any integer $k\geq 1$, we find that
     \begin{align}\label{eq:Siegel}
     k^{\frac{3}{2}}e^{-\frac{k^2-k}{4}}\#\{\lambda:\lambda\vdash k\}\leq k^{\frac{3}{2}}e^{-\frac{k^2-k}{4}+\pi\sqrt{2k/3}}\leq 68 \quad \forall k\in \NN.
     \end{align}
   Combining \eqref{eq:Siegel} with \eqref{eq:MomUppBd} completes the proof of the upper bound in \eqref{eq:MomBound}.



\end{proof}


\begin{proof}[Proof of \eqref{eq:TightUpBd}]
Combining Markov's inequality and the second inequality of \eqref{eq:MomBound}, we get
 \begin{align}\label{eq:UpsilonTB}
 \mathbb{P}(\Upsilon_T(0)\geq  s) \leq 69 \exp\Big(-\max_{k\in \NN} \big[ksT^{\frac{1}{3}}- \log \psi_T(k)\big]\Big).
 \end{align}
 By Stirling's formula $\psi_T(k)=\exp\big(\frac{Tk^3(1+O(k^{-3/2}))}{12}\big)$.
 Set $k_0=\lfloor 2s^{\frac{1}{2}}T^{-\frac{1}{3}}\rfloor$. When $s\geq \frac{9}{16}\epsilon^{-2} T^{\frac{2}{3}}$,
 \begin{align}\label{eq:SpecialKBound}
k_0sT^{\frac{1}{3}}-\log\psi_{T}(k_0)\geq  k_0sT^{\frac{1}{3}}- \frac{Tk^3_0(1+O(\epsilon^{\frac{3}{2}}))}{12} \geq \frac{4\big(1-\epsilon\big)s^{\frac{3}{2}}}{3}. &&&
 \end{align}
 The first inequality of \eqref{eq:SpecialKBound} follows by noting that $k_0\geq c\epsilon^{-1}$ for some positive constant $c$. We get the second inequality of \eqref{eq:SpecialKBound} by noticing that $\lfloor2s^{\frac{1}{2}}T^{-\frac{1}{3}}\rfloor\geq 2s^{\frac{1}{2}}T^{-\frac{1}{3}}-1\geq 2s^{\frac{1}{2}} T^{-\frac{1}{3}}\big(1-\frac{2\epsilon}{3}\big)$.
 Finally, \eqref{eq:TightUpBd} follows by plugging \eqref{eq:SpecialKBound} into the r.h.s. of \eqref{eq:UpsilonTB}.
\end{proof}

\begin{proof}[Proof of \eqref{eq:TightLowBd}] Fixing now $k_0 = \lceil 2\cdot(3(1+5\epsilon/6)s)^{\frac{1}{2}} T^{-\frac{1}{3}}\rceil$, we observe that
\begin{align}\label{eq:MLowBd}
\exp\big(k_0 s T^{\frac{1}{3}}\big)\leq \frac{1}{2}\frac{k_0!}{2\sqrt{\pi T}k^{\frac{3}{2}}_0}\exp\left(\frac{k^3_0 T}{12}\right).
\end{align}
To prove this inequality first note that
\begin{align}\label{eq:IneqSeri1}
k_0 s T^{\frac{1}{3}}\leq \Big(2\cdot \big(3(1+\frac{5\epsilon}{6})s\big)^{\frac{1}{2}} T^{-\frac{1}{3}} +1\Big) s T^{\frac{1}{3}} \leq 2\sqrt{3}\Big(1+\frac{5\epsilon}{12}+\frac{2\epsilon}{3\sqrt{3}}\Big)s^{\frac{3}{2}}.
\end{align}
where the first inequality follows from $\lceil k\rceil\leq k+1$ and the second inequality is obtained using $s\geq \frac{9}{16}\epsilon^{-2}T^{\frac{2}{3}}$. Moreover, using $k!\geq k^{\frac{3}{2}}$ which holds for all $k\in \ZZ_{\geq 3}$, we see
\begin{align}
\text{r.h.s. of \eqref{eq:MLowBd}} \geq \frac{1}{4\sqrt{\pi T}} \exp\Big(2\sqrt{3}\big(1+\frac{5\epsilon}{4}\big)s^{\frac{3}{2}}\Big). \label{eq:IneqSeri2}
\end{align}
Now, \eqref{eq:MLowBd} follows from \eqref{eq:IneqSeri1} and \eqref{eq:IneqSeri2} by noting that $\frac{5}{4}\geq \frac{5}{12}+\frac{2}{3\sqrt{3}}$ and $T\leq\frac{64}{27}(\epsilon^2 s)^{\frac{3}{2}}$.

Combining the first inequality of \eqref{eq:MomBound} with \eqref{eq:MLowBd}  yields
\begin{align}\label{eq:ReArr}
 \mathbb{P}\big(\Upsilon_T(0)> s\big)\geq \mathbb{P}(E),\quad \textrm{with} \quad E = \Big\{\exp\big(k_0T^{\frac{1}{3}}\Upsilon_T(0)\big)> \frac{1}{2}\mathbb{E}[\exp(k_0T^{\frac{1}{3}}\Upsilon_T(0) )]\Big\}.\qquad
\end{align}

\begin{claim} Fix $p,q>1$ such that $p^{-1}+q^{-1}=1$. Then,
\begin{align}\label{eq:PZTypeIneq}
\mathbb{P}(E)\geq 2^{-q}\, \mathbb{E}\big[\exp(k_0T^{\frac{1}{3}}\Upsilon_T(0))\big]^{q}\, \mathbb{E}\big[\exp(pk_0T^{\frac{1}{3}}\Upsilon_T(0))\big]^{-q/p}
\end{align}
\end{claim}
\begin{proof}
Let us write
\begin{align}
\mathbb{E}\big[\exp(k_0T^{\frac{1}{3}} \Upsilon_T(0))\big]= \mathbb{E}\Big[\exp(k_0T^{\frac{1}{3}}\Upsilon_T(0))\mathbbm{1}\big(E^c\big)\Big]+ \mathbb{E}\Big[\exp(k_0T^{\frac{1}{3}}\Upsilon_T(0))\mathbbm{1}\big(E\big)\Big].&&\label{eq:PZarg}
\end{align}
 The first term on the r.h.s. of \eqref{eq:PZarg} is bounded above by $\frac{1}{2}\mathbb{E}\big[\exp(k_0T^{\frac{1}{3}}\Upsilon_T(0))\big]$. To bound the second term, we use H\"older's inequality
 \begin{align}\label{eq:Holder}
 \mathbb{E}\Big[\exp(k_0T^{\frac{1}{3}}\Upsilon_T(0))\mathbbm{1}\big(E\big)\Big]\leq \Big[\mathbb{E}\big[\exp(pk_0T^{\frac{1}{3}}\Upsilon_T(0))\big]\Big]^{\frac{1}{p}} \mathbb{P}(E)^{\frac{1}{q}}
 \end{align}
where $p^{-1}+q^{-1}=1$. Plugging the upper  bound of \eqref{eq:Holder} into the r.h.s. of \eqref{eq:PZarg} and simplifying yields \eqref{eq:PZTypeIneq} and proves the claim.
  \end{proof}

Returning to the proof of \eqref{eq:TightLowBd}, thanks to \eqref{eq:MomBound}, we find that
\[\text{r.h.s. of \eqref{eq:PZTypeIneq} }\geq \exp\Big(-\frac{q(p^2-1)k^3_0 T(1+O(\epsilon^{3/2}))}{12}\Big).\] From $p^{-1}+q^{-1}=1$, it follows that $q(p^2-1) =p(p+1)$. Taking $p = 1+\epsilon/6$ and recalling that $k_0= \lceil 2\cdot(3(1+5\epsilon/6)s)^{\frac{1}{2}} T^{-\frac{1}{3}}\rceil$, we get
$
\text{l.h.s. of \eqref{eq:PZTypeIneq}}\geq 2^{-q}\exp\big(-4\sqrt{3}(1+3\epsilon/2) s^{\frac{3}{2}}\big).
$
Since $q= 6\epsilon^{-1}+1$, we find that the r.h.s. of the above inequality is bounded below by $\exp(-4\sqrt{3}(1+3\epsilon)s^{\frac{3}{2}})$ for all $s\geq \frac{9}{16}\epsilon^{-2}T$ and $T\geq T_0\geq\pi$. This completes the proof.
\end{proof}

\begin{proof}[Proof of \eqref{eq:LowBdAllWay}]
Fix $k_0=\lceil 2\cdot (3(1+5\epsilon/6)s)^{\frac{1}{2}}T^{-\frac{1}{3}}\rceil $. Our aim is to obtain a lower bound for the r.h.s. of \eqref{eq:ReArr}. Applying \eqref{eq:PZTypeIneq} with $p=q=2$ yields
\begin{align}\label{eq:Obs5}
\mathbb{P}(\Upsilon_T(0)>s)\geq \frac{1}{2}\exp\Big(-\frac{7k^3_0 T}{12}\Big).
\end{align}
For $k_0\geq 2$, we have  $k_0\leq 2(k_0-1)$ which implies $k_0\leq 4\cdot (3(1+\epsilon))^{\frac{1}{2}} T^{-\frac{1}{3}}$ and hence
$\mathbb{P}(\Upsilon_T(0)>s)\geq \frac{1}{2}\exp(- 2^{6}s^{\frac{3}{2}}).
$
When $k_0=1$, r.h.s. \eqref{eq:Obs5}$\geq \frac{1}{2}\exp(-2^{\frac{7}{2}}\epsilon^{-3}s^{\frac{3}{2}})$ for all $s\geq \frac{1}{8}\epsilon T^{\frac{2}{3}}$.
\end{proof}

\begin{proof}[Proof of \eqref{eq:UpLowRough}]
We first prove the second inequality of \eqref{eq:UpLowRough}. Fix $T \in [T_0,\pi]$. Applying Markov's inequality yields
\begin{align}\label{eq:Markov2}
\mathbb{P}(\Upsilon_{T}(0)\geq s)\leq 69 \exp\big(-\max_{k\in \NN}\big[ksT^{\frac{1}{3}}- \log \psi_{T}(k)\big]\big).
\end{align}
Owing to Stirling's formula, we get $\psi_{T}(k)= \exp(Tk^{3}(1+O(k^{-3/2}))-\frac{k}{2}\log T_0)$. Set $k_0= \lfloor 2s^{\frac{1}{3}} T^{-\frac{1}{3}}\rfloor$ and when $s\geq \frac{9}{16}\epsilon^{-2}T^{\frac{2}{3}}+ 24T^{-\frac{1}{3}}_0(1-\epsilon)^{-1}|\log (T_0/\pi)|$, we have
\begin{align}
k_0sT^{\frac{1}{3}} - \log \psi_{T}(k_0)\geq k_0sT^{\frac{1}{3}} - \frac{Tk^3_0(1+O(\epsilon^{\frac{3}{2}}))}{12} + \frac{k}{2}\log T_0\geq \frac{4(1-\epsilon)}{3}s^{\frac{3}{2}} +\frac{k_0}{2}\log T_0. &&\label{eq:SeriesIneq}
\end{align}
for some constant $c=c(\epsilon, T_0)>0$.
The first inequality of \eqref{eq:SeriesIneq} follows since $k_0\geq c\epsilon^{-1}$ for some positive constant $c>0$ and the second inequality follows since $\lfloor 2s^{\frac{1}{2}}T^{-\frac{1}{3}}\rfloor \geq 2s^{\frac{1}{2}}T^{-\frac{1}{3}}(1-\frac{2\epsilon}{3})$. Now, we claim that the r.h.s. of \eqref{eq:SeriesIneq} is bounded below by $(1-\epsilon)s^{\frac{3}{2}}$. To see this, we write
\begin{align}
\frac{k_0}{2}\log T_0\geq \min\{s^{\frac{1}{2}} T^{-\frac{1}{3}}\log T_0, 0\}\geq - \frac{1}{24}s^{\frac{3}{2}}(1-\epsilon)(T_0/T)^{1/3}\geq -\frac{1}{24}(1-\epsilon)s^{\frac{3}{2}}
\end{align}
where the first inequality follows since $k_0\leq 2s^{\frac{1}{2}}T^{-\frac{1}{3}}$, the second inequality holds since $s\geq 24 T^{-\frac{1}{3}}_0(1-\epsilon)^{-1} |\log (T_0/\pi)|$ and the last inequality is obtained by noting that $T_0\leq T$. Substituting the inequalities in the above display in the r.h.s. of \eqref{eq:SeriesIneq} proves the claim. As a consequence, for all $T\in [T_0,\pi]$,
\begin{align}
\max_{k\in \NN}\big[ksT^{\frac{1}{3}}- \log \psi_{T}(k)\big] \geq k_0sT^{\frac{1}{3}} - \log \psi_{T}(k_0)\geq (1-\epsilon)s^{\frac{3}{2}}.
\end{align}
Applying the inequality in the above display in the r.h.s. of \eqref{eq:Markov2} yields the second inequality of \eqref{eq:UpLowRough}.

Now, we turn to show the first inequality of \eqref{eq:UpLowRough}.  Fix $k_0= \lceil 4s^{\frac{1}{2}} T^{-\frac{1}{3}}\rceil$. We claim that for all $T\in [T_0,\pi]$
\begin{align}\label{eq:BdLow}
\exp\big(k_0 sT^{\frac{1}{3}}\big) \leq \frac{1}{2}\big(T_0/T\big)^{\frac{k_0-1}{2}}\frac{k_0!}{2\sqrt{\pi T}k^{\frac{3}{2}}_0}\exp\big(\frac{k^3_0T}{12}\big).
\end{align}
To prove \eqref{eq:BdLow} we note
\begin{align}\label{eq:BdLowHere1}
k_0sT^{\frac{1}{3}}\leq \big(4s^{\frac{1}{2}} T^{-\frac{1}{3}} +1\big)sT^{\frac{1}{3}}\leq 4\big(1+\frac{\epsilon }{3}\big)s^{\frac{3}{2}}
\end{align}
where the first inequality follows since $\lceil k\rceil \leq k+1$ and the second inequality is obtained using $s\geq \frac{9}{16}\epsilon^{-2}T^{2/3}$. Since we know $T_0\leq T\leq \pi$ and $k^3_0T=(\lceil 4s^{\frac{1}{2}} T^{-\frac{1}{3}}\rceil)^3 T\geq 64s^{\frac{3}{2}}$,
\begin{align}\label{eq:BdLowHere2}
\text{r.h.s. of \eqref{eq:BdLow}}\geq \Big(\frac{T_0}{\pi}\Big)^{\frac{k_0-1}{2}}\frac{k_0!}{4\pi k^{\frac{3}{2}}_0} \exp\Big(\frac{64}{12}s^{\frac{3}{2}}\Big)= \Big(\frac{T_0}{\pi}\Big)^{\frac{k_0-1}{2}} \frac{k_0!}{4\pi k^{\frac{3}{2}}_0} \exp\big((5+3^{-1})s^{\frac{3}{2}}\big) &&
\end{align}
By using the fact that $s\geq \frac{9}{16}\epsilon^{-2}T^{2/3}+24T^{-\frac{1}{3}}_0(1-\epsilon)|\log (T_0/\pi)|$ and $\epsilon<3/5$, we get
\begin{align}
 k_0 = \lceil 4s^{\frac{1}{2}} T^{-\frac{1}{3}}\rceil \geq 4s^{\frac{1}{2}} T^{-\frac{1}{3}} > 3\epsilon^{-1}>5, \quad
 \frac{1}{3}s^{\frac{3}{2}} > 2s^{\frac{1}{2}}T^{-\frac{1}{3}}_0 |\log (T_0/\pi)|\geq \frac{k_0-1}{2}|\log (T_0/\pi)|.
\end{align}
Now, \eqref{eq:BdLow} follows from \eqref{eq:BdLowHere1}, \eqref{eq:BdLowHere2} and the inequalities of the above display by noting that $4(1+\epsilon/3)\leq 5$, $k_0\geq 6$ and $(T_0/\pi)^{(k_0-1)/2}\exp(3^{-1}s^{3/2})\geq 1$.

For any $T \in [T_0, \pi]$, combining the first inequality of \eqref{eq:MomBound} with \eqref{eq:BdLow} yields
\begin{align}
\mathbb{P}\big(\Upsilon_{T}(0)>s\big)\geq \mathbb{P}(\tilde{E}), \quad \text{where }\tilde{E} =\Big\{\exp\big(k_0T^{\frac{1}{3}}\Upsilon_{T}(0)\big)> \frac{1}{2}\mathbb{E}\big[\exp\big(k_0 T^{\frac{1}{3}} \Upsilon_{T}(0)\big)\big]\Big\}.
\end{align}
Applying \eqref{eq:PZTypeIneq} with $p=q=2$ shows
\begin{align}\label{eq:LowBd2}
\mathbb{P}\big(\Upsilon_{T}(0)>s\big) \geq \frac{1}{2}\exp\Big(-\frac{7k^3_0 T}{12}\Big)\geq \exp\big(-cs^{\frac{3}{2}}\big)
\end{align}
for some absolute constant $c>0$. The last inequality of the above display follows since $k_0 = \lceil 4s^{\frac{1}{2}} T^{-\frac{1}{3}}\rceil$. Note that \eqref{eq:LowBd2} implies the first inequality of \eqref{eq:UpLowRough}. This completes the proof.
\end{proof}

%


\subsection{Proof of Proposition~\ref{UpperDeepTailLemma}}
We prove this by contradiction. Assume there exists $M>0$ such that $\mathbb{P}(\Upsilon_T(0)>s)\leq e^{-cs^{\frac{3}{2}}}$ for all $s\geq M$. Dividing the expectation integral into $(-\infty,0]$, $[0,M]$ and $(M,\infty)$, we have
\begin{align}
 \mathbb{E}\big[\exp(k\Upsilon_T(0) T^{\frac{1}{3}})\big] \leq 1+MkT^{\frac{1}{3}}e^{kMT^{\frac{1}{3}}}+ \int^{\infty}_{M} kT^{\frac{1}{3}}e^{ks T^{\frac{1}{3}} -cs^{\frac{3}{2}}} ds. \label{eq:1stApprox}
\end{align}
Observing that
\begin{align}\label{eq:Maximizer}
\argmax_{s\geq 0} \big\{ks T^{\frac{1}{3}} - cs^{\frac{3}{2}}\big\} = \frac{4k^2 T^{\frac{2}{3}}}{9 c^2},
\end{align}
we may choose $k$ to be a sufficiently large integer such that the r.h.s. of \eqref{eq:Maximizer} exceeds $M$. Then, approximating the integral of \eqref{eq:1stApprox} by $C^{\prime}kT^{\frac{1}{3}}\exp(\max_{s\geq 0} \big\{ksT^{\frac{1}{3}} -cs^{\frac{3}{2}}\big\})$ for some absolute constant $C^{\prime}=C^{\prime}(k)$ and plugging in the value of the maximizer from \eqref{eq:Maximizer}, we find
\begin{align}\label{eq:MomUpBd}
\mathbb{E}\big[\exp(k\Upsilon_T(0) T^{\frac{1}{3}})\big]\leq (M+1)kT^{\frac{1}{3}} + C^{\prime}kT^{\frac{1}{3}} e^{\frac{4k^3T}{27c^2}}.
\end{align}
Applying $c>\frac{4}{3}\big(1+\frac{1}{3}\epsilon\big)$ into \eqref{eq:MomUpBd} shows that the r.h.s. of \eqref{eq:MomUpBd} is less than $e^{(1-\epsilon)\frac{k^3T}{12}}$ which contradicts \eqref{eq:MomBound}. Hence, the claim follows.



\subsection{Proof of Proposition~\ref{thm:UpTail}}
Our proof of Proposition~\ref{thm:UpTail} relies on a Laplace transform formula for $\mathcal{Z}^{\mathbf{nw}}(T,0)$ which was proved in \cite{BorGor16} and follows from the exact formula for the probability distribution of $\Upsilon_T(0)$ of \cite{Amir11}. It connects $\mathcal{Z}^{\mathbf{nw}}(T,0)$ with the Airy point process $\mathbf{a}_1> \mathbf{a}_2>\ldots $. The latter is a well studied determinantal point process in random matrix theory (see, e.g., \cite[Section~4.2]{AGZ10}).

For convenience, we introduce following shorthand notations:
\begin{align}
\mathcal{I}_s(x) := \frac{1}{1+\exp(T^{\frac{1}{3}}(x-s))} , \qquad \mathcal{J}_s(x):= \log \big(1+\exp(T^{\frac{1}{3}}(x-s))\big).
\end{align}
 It is worth noting that $\mathcal{I}_s(x)= \exp(-\mathcal{J}_s(x))$.
\bp[Theorem~1 of \cite{BorGor16}]\label{ppn:PropConnection}
For all $s\in\RR$,
\begin{align}\label{eq:Connection}
\mathbb{E}_{\mathrm{KPZ}}\Big[\exp\Big(-\exp\big(T^{\frac{1}{3}}(\Upsilon_T(0)-s)\big)\Big)\Big]=\mathbb{E}_{\mathrm{Airy}}\left[\prod_{k=1}^{\infty} \mathcal{I}_s(\mathbf{a}_k)\right].
\end{align}
\ep

We start our proof of Proposition~\ref{thm:UpTail} with upper and lower bounds on the r.h.s. of \eqref{eq:Connection}.

\bp\label{thm:MainTheorem}
Fix some $\zeta\leq \epsilon\in (0, 1)$ and $T_0>0$. Continuing with the notation of Proposition~\ref{ppn:PropConnection}, there exists $s_0=s_0(\epsilon, \zeta, T_0)$ such that for all $s\geq s_0$,
\begin{align}
1- \mathbb{E}\Big[\prod_{k=1}^{\infty} \mathcal{I}_s(\mathbf{a}_k)\Big]&\leq e^{-\zeta s T^{1/3} }+ e^{- \frac{4}{3} (1-\epsilon)s^{3/2}},\label{eq:UpBound}\\
1- \mathbb{E}\Big[\prod_{k=1}^{\infty} \mathcal{I}_s(\mathbf{a}_k)\Big]&\geq e^{-(1+\zeta)s T^{1/3}} + e^{- \frac{4}{3}(1+\epsilon) s^{3/2}}.\label{eq:LowrBound}
\end{align}
\ep

We defer the proof of Proposition~\ref{thm:MainTheorem} to Section~\ref{NWLaplaceTail}.
\begin{proof}[Proof of Proposition~\ref{thm:UpTail}]
 Define $\bar{s}:= (1+\zeta)s$ and $\theta(s) := \exp\big(- \exp\big(T^{\frac{1}{3}}(\Upsilon_T(0)-s)\big)\big)$. Thanks to \eqref{eq:Connection}, we have $\mathbb{E}_{\mathrm{KPZ}}[\theta(s)] = \mathbb{E}_{\mathrm{Airy}}[\prod_{k=1}^{\infty} \mathcal{I}_s(\mathbf{a}_k)]$. Note that
\begin{align}\label{eq:1stStepIneq}
\theta(s)\leq \mathbbm{1}(\Upsilon_T(0)\leq \bar{s})+\mathbbm{1}(\Upsilon_T(0)> \bar{s}) \exp(-\exp(\zeta sT^{1/3})).
\end{align}
Rearranging, taking expectations and applying \eqref{eq:Connection}, we arrive at
\begin{align}\label{eq:StepOfUpperBound}
\mathbb{P}(\Upsilon_T(0)> \bar{s})\leq  \Big(1-\exp(-\exp(\zeta sT^{\frac{1}{3}}))\Big)^{-1} \Big(1-\mathbb{E}_{\mathrm{Airy}}\big[\prod_{k=1}^{\infty} \mathcal{I}_s(\mathbf{a}_k)]\Big).
\end{align}
By taking $s$ sufficiently large and $T\geq T_0$, we may assume that $1-\exp(-\exp(\zeta s T^{\frac{1}{3}}))\geq \frac{1}{2}$. Plugging this bound and \eqref{eq:UpBound} into the r.h.s. of \eqref{eq:StepOfUpperBound} yields
\begin{align}
\mathbb{P}(\Upsilon_T(0)\geq \bar{s}) &\leq e^{-\zeta s T^{1/3}} + e^{-\frac{4}{3}(1-\epsilon)s^{3/2}}
\end{align}
   for all $s\geq s_{0}$ where $s_0$ depends on $\epsilon, \zeta$ and $ T_0$. This proves \eqref{eq:UpBoundBd}.

We turn now to prove \eqref{eq:LowBound}. Using Markov's inequality,
\begin{align}
\mathbb{P}(\Upsilon_T(0)\leq s)= \mathbb{P}\Big(\theta(\bar{s})\geq \exp\big(- e^{-\zeta sT^{1/3}}\big)\Big)\leq \exp\big(e^{-\zeta sT^{1/3}}\big)\cdot \mathbb{E}[\theta(\bar{s})].
\end{align}
Rearranging yields
$1- \exp\Big(- e^{-\zeta s T^{1/3}}\Big)\mathbb{P}(\Upsilon_T(0)\leq s) \geq 1- \mathbb{E}\left[\theta(\bar{s})\right]$.
Finally, applying \eqref{eq:Connection} and \eqref{eq:LowrBound} to the r.h.s. of this result, we get \eqref{eq:LowBound}.
\end{proof}


\subsubsection{Proof of Proposition~\ref{thm:MainTheorem}}\label{NWLaplaceTail}
\begin{proof}[Proof of \eqref{eq:UpBound}]
  We start by noticing the following trivial lower bound
   \begin{align}\label{eq:TrivLowerBound}
   \mathbb{E}_{\mathrm{Airy}}\big[\prod_{k=1}^{\infty} \mathcal{I}_{s}(\mathbf{a}_k)\big] &\geq \mathbb{E}_{\mathrm{Airy}}\big[\prod_{k=1}^{\infty} \mathcal{I}_{s}(\mathbf{a}_k)\mathbbm{1}(\mathbf{A})\big]
   \end{align}
   where $\mathbf{A}=\big\{\mathbf{a}_1\leq  (1-\zeta)s\big\}$.    Setting $k_0:= \lfloor \frac{2}{3\pi}s^{\frac{9}{4}+2\epsilon}\rfloor$ we observe that
   \begin{align}
   \prod_{k=1}^{k_0} \mathcal{I}_s(\mathbf{a}_k)\mathbbm{1}(\mathbf{A})&= \exp\Big(- \sum_{k=1}^{k_0}\mathcal{J}_s(\mathbf{a}_k)\Big)\mathbbm{1}(\mathbf{A})\geq \exp\Big(-\frac{2}{3\pi}s^{\frac{9}{4}+2\epsilon}e^{-T^{\frac{1}{3}s\zeta}}\Big). \label{eq:1stBound}
   \end{align}
  where inequality is obtained via $\mathcal{J}_s(\mathbf{a}_k)\leq e^{-T^{\frac{1}{3}}s\zeta}$ which follows on the event $\mathbf{A}$. Our next task is to bound $\prod_{k>k_0} \mathcal{I}_s(\mathbf{a}_k)$ from below. To achieve this, we recall the result of \cite[Proposition~4.5]{CG18} which shows that for any $\epsilon,\delta\in (0,1)$ we can augment the probability space on which the Airy point process is defined so that there exists a random variable $C^{\mathrm{Ai}}_{\epsilon}$  satisfying
   \begin{equation}\label{eq:AiryConcentration}
(1+\epsilon)\lambda_{k} - C^{\mathrm{Ai}}_{\epsilon}\leq \mathbf{a}_k\leq  (1-\epsilon)\lambda_k+ C^{\mathrm{Ai}}_{\epsilon}\quad \text{for all }k\geq 1\quad \text{ and } \quad \mathbb{P}(C^{\mathrm{Ai}}_{\epsilon}\geq s)\leq e^{- s^{1-\delta}}
\end{equation}
 for all $s\geq s_0$ where $s_0=s_0(\epsilon,\delta)$ is a constant. Here, $\lambda_k$ is the $k$-th zero of the Airy function (see \cite[Proposition~4.6]{CG18}) and we fix some $\delta\in(0, \epsilon)$. Define $\phi(s) := s^{\frac{3+8\epsilon/3}{2(1-\delta)}}$.
Now, we write
\begin{align}\label{eq:ResBound}
\prod_{k>k_0} \mathcal{I}_s(\mathbf{a}_k) \geq \prod_{k>k_0} \mathcal{I}_s(\mathbf{a}_k)\mathbbm{1}(C^{\mathrm{Ai}}_{\epsilon}\leq \phi(s)) &\geq \exp\Big(-\sum_{k>k_0} \mathcal{J}_s\big((1-\epsilon)\lambda_k+ \phi(s)\big)\Big).
\end{align}
 Appealing to the tail probability of $C^{\mathrm{Ai}}_{\epsilon}$, we have
$\mathbb{P}(C^{\mathrm{Ai}}_{\epsilon}\leq \phi(s))\geq 1- e^{-s^{\frac{3}{2}+\frac{4}{3}\epsilon}}$.
We now claim that for some constant $C>0$,
\begin{equation}\label{eq:ResBoundSep}
\sum_{k>k_0} \mathcal{J}_s((1-\epsilon)\lambda_k+ \phi(s)) \leq \frac{C}{T^{\frac{1}{3}}}\exp(-sT^{\frac{1}{3}}).
\end{equation}
To prove this note that for all $k\geq k_0$,
\begin{align}\label{eq:TwinIneq}
\lambda_k \leq -\Big(\frac{3\pi k}{2}\Big)^{\frac{3}{2}} \quad \text{and, } \quad (1-\epsilon)(\frac{3\pi k}{2}\big)^{\frac{3}{2}} - \phi(s) \geq (1-\epsilon)\Big(\frac{3\pi}{2} (k-k_0)\Big)^{\frac{1}{3}}.
\end{align}
The first inequality of \eqref{eq:TwinIneq} is an outcome of \cite[Proposition~4.6]{CG18} and the second inequality follows from \cite[Lemma~5.6]{CG18}. Applying \eqref{eq:TwinIneq},
we get
\begin{align}\label{eq:JBound}
\mathcal{J}_s\Big((1-\epsilon)\lambda_k+ \phi(s)\Big)\leq  e^{T^{1/3}\big(-s -(1-\epsilon)(3\pi k/2)^{2/3}+\phi(s)\big)} \leq e^{T^{1/3}\big(-s-(1-\epsilon)(k-k_0)^{2/3}\big)}.
\end{align}
Summing over $k>k_0$ in \eqref{eq:JBound}, approximating the sum by the corresponding integral, and evaluating yields \eqref{eq:ResBoundSep}.

\smallskip

Now, we turn to complete the proof of \eqref{eq:UpBound}. Plugging \eqref{eq:ResBoundSep} into the r.h.s. of \eqref{eq:ResBound} yields
\begin{align}\label{eq:ResBoundFinal}
\prod_{k>k_0} \mathcal{I}_s(\mathbf{a}_k)\mathbbm{1}(C^{\mathrm{Ai}}_{\epsilon}\leq \phi(s))\geq \exp\left(-\frac{C}{T^{\frac{1}{3}}} \exp(- sT^{\frac{1}{3}})\right).
\end{align}
Combining \eqref{eq:1stBound} and \eqref{eq:ResBoundFinal} yields
\begin{equation}
\text{l.h.s. of \eqref{eq:TrivLowerBound}}\geq \exp\Big(-\frac{2}{3\pi}s^{\frac{9}{4}+2\epsilon}e^{-\zeta sT^{\frac{1}{3}}}-\frac{C}{T^{\frac{1}{3}}}e^{-sT^{\frac{1}{3}}}\Big) \mathbb{P}\big(C^{\mathrm{Ai}}_{\epsilon}\leq \phi(s), \mathbf{A}\big).  \label{eq:LowB1}
\end{equation}
To finish the proof, we observe that
\begin{equation}\label{eq:LowB2}
\mathbb{P}\big(C^{\mathrm{Ai}}_{\epsilon}\leq \phi(s),\mathbf{A}\big)\geq 1 - \mathbb{P}(C^{\mathrm{Ai}}_{\epsilon}\geq \phi(s)) - \mathbb{P}(\mathbf{A}^c)\geq 1- e^{-s^{\frac{3}{2}+\frac{4}{3}\epsilon}} - e^{-\frac{4}{3}(1-\epsilon)s^{\frac{3}{2}}}
\end{equation}
for all $s\geq s_0$. The second inequality above used $\mathbb{P}(\mathbf{A}^c)=\mathbb{P}(\mathbf{a}_1\geq (1-\zeta)s)\leq \exp(-\frac{4}{3}(1-\epsilon)s^{\frac{3}{2}})$ which holds when $s$ is sufficiently large (see \cite[Theorem~1.3]{RRV11}).
Plugging \eqref{eq:LowB2} into the r.h.s. of \eqref{eq:LowB1} and rearranging yields
$e^{-(1-\epsilon)\zeta s T^{\frac{1}{3}}}\leq 1- \exp\Big(-\frac{2}{3\pi}s^{\frac{9}{4}+2\epsilon}e^{-\zeta sT^{\frac{1}{3}}}-\frac{C}{T^{\frac{1}{3}}}e^{-sT^{\frac{1}{3}}}\Big)\leq e^{-(1+\epsilon)\zeta s T^{\frac{1}{3}}}$
 for sufficiently large $s$. Hence \eqref{eq:UpBound} follows.
\end{proof}

\begin{proof}[Proof of \eqref{eq:LowrBound}]
Here, we need to get an upper bound on $\mathbb{E}\big[\prod_{k=1}^{\infty} \mathcal{I}_s(\mathbf{a}_k)\big]$. We start by splitting $\mathbb{E}\big[\prod_{k=1}^{\infty} \mathcal{I}_s(\mathbf{a}_k)\big]$ into two different parts (again set $\mathbf{A}=\big\{\mathbf{a}_1 \leq (1+\zeta)s\big\}$):
\begin{align}\label{eq:LowTail1stStep}
\mathbb{E}\Big[\prod_{k=1}^{\infty} \mathcal{I}_s(\mathbf{a}_k)\Big]\leq \mathbb{E}\Big[\prod_{k=1}^{\infty} \mathcal{I}_s(\mathbf{a}_k)\mathbbm{1}(\mathbf{A})\Big] + \mathbb{P}(\mathbf{A}^c)\cdot \exp(- \zeta s T^{\frac{1}{3}}).
\end{align}
%
Let us define $\chi^{\mathrm{Ai}}(s):= \#\{\mathbf{a}_i\geq s\}$ and, for $c\in(0,\tfrac{2}{3\pi})$ fixed, define
\begin{align}
\mathbf{B}: = \Big\{ \chi^{\mathrm{Ai}}(-\zeta s)- \mathbb{E}\big[\chi^{\mathrm{Ai}}(-\zeta s)\big]\geq - c(\zeta s)^{\frac{3}{2}}\Big\}
\end{align}
We split the first term on the r.h.s. of \eqref{eq:LowTail1stStep} as follows
 \begin{align}
 \mathbb{E}\Big[\prod_{k=1}^{\infty} \mathcal{I}_s(\mathbf{a}_k)\mathbbm{1}(\mathbf{A})\Big] &\leq \mathbb{E}\Big[\prod_{k=1}^{\infty}\mathcal{I}_s(\mathbf{a}_k)\mathbbm{1}\big(\mathbf{B}\cap \mathbf{A}\big)\Big]+  \mathbb{E}\Big[\mathbbm{1}(\mathbf{B}^c \cap \mathbf{A})\Big]. \label{eq:SplitEq1}
\end{align}

 On the event $\mathbf{B}$, we may bound  $$\prod_{k=1}^{\infty}\mathcal{I}_s(\mathbf{a}_k) \mathbbm{1}(\mathbf{B})\leq \exp\Big(-\Big(\frac{2}{3\pi}-c\Big)(\zeta s)^{\frac{3}{2}}e^{-(1+\zeta)s T^{\frac{1}{3}}}\Big)$$ so that
\begin{align}\label{eq:2ndSplitBd}
\mathbb{E}\Big[\prod_{k=1}^{\infty}\mathcal{I}_s(\mathbf{a}_k)\mathbbm{1}\big(\mathbf{B}\cap\mathbf{A}\big)\Big]\leq \exp\Big(-\Big(\frac{2}{3\pi}-c\Big)(\zeta s)^{\frac{3}{2}}e^{-(1+\zeta)sT^{\frac{1}{3}}}\Big)\cdot\mathbb{P}(\mathbf{A}).
\end{align}
For large $s$, the r.h.s. of \eqref{eq:2ndSplitBd} is bounded above by $\exp\big(-e^{-(1+\zeta)sT^{\frac{1}{3}}}\big)\mathbb{P}(\mathbf{A})$. Thanks to Theorem~1.4 of \cite{CG18}, we know that for any $\delta >0$, there exists $s_{\delta}$ such that $\mathbb{P}(\mathbf{B}^c)\leq e^{- c(\zeta s)^{3-\delta}}$ for all $s\geq s_{\delta}$. Now, we plug these bounds into \eqref{eq:SplitEq1} which provides an upper bound to the first term on the r.h.s. of \eqref{eq:LowTail1stStep}. As a result, we find
\begin{align}
1-\mathbb{E}\big[\prod_{k=1}^{\infty}\mathcal{I}_s(\mathbf{a}_k)\big]&\geq 1- e^{-e^{-(1+\zeta)sT^{\frac{1}{3}}}} - e^{- c(\zeta s)^{3-\delta}} + \mathbb{P}(\mathbf{A}^c)\big( e^{- e^{-(1+\zeta)sT^{\frac{1}{3}}}}- e^{- \zeta s T^{\frac{1}{3}}}\big).\qquad  \label{eq:LowrFnStep}
\end{align}
Finally, we note that $\mathbb{P}(\mathbf{A}^c)\geq \exp\big(-\frac{4}{3}(1+\epsilon)s^{\frac{3}{2}}\big)$ (again thanks to \cite[Theorem~1.3]{RRV11}). Thus, the r.h.s. of \eqref{eq:LowrFnStep} is lower bounded by $\frac{1}{2}e^{-(1+\zeta)s T^{1/3}} + e^{-\frac{4}{3}(1+\epsilon)s^{3/2}}$ for sufficiently large $s$. This completes the proof of \eqref{eq:LowrBound} and hence also of Proposition~\ref{thm:MainTheorem}.
\end{proof}

\section{Upper tail under general initial data}\label{UpperTailGSEC}

 This section contains the proofs of Theorems~\ref{Main4Theorem} and \ref{Main6Theorem}.

\subsection{Proof of Theorem~\ref{Main4Theorem}}\label{Proof4Theorem}
Theorem~\ref{Main4Theorem} will follow directly from the next two propositions which leverage narrow wedge upper tail decay results to give general initial data results. The cost of this generalization is in terms of both the coefficients in the exponent and the ranges on which the inequalities are shown to hold. Recall $h^{f}_T$ and $\Upsilon_T$ from \eqref{eq:ScalCentHeight} and \eqref{eq:DefUpsilon} respectively.

  The following proposition has two parts which correspond to $T$ being greater or, less than equal to $\pi$. The main goal of this proposition is to provide a recipe to deduce upper bounds on $\mathbb{P}(h^{f}_T(0)>s)$ by employing the upper bounds on $\mathbb{P}(\Upsilon_T(0)>s)$. We have noticed in Theorem~\ref{GrandUpTheorem} that the latter bounds vary as $s$ lies in different intervals and furthermore, those intervals vary with $T$. This motivates us to choose a generic set of intervals of $s$ based on a given $T$ and assume upper bounds on $\mathbb{P}(\Upsilon_T(0)>s)$ in those intervals. In what follows, we show how those translate to the upper bounds on $\mathbb{P}(h^{f}_T(0)>s)$.

  \bp\label{SubstituteTheo}
  Fix $\epsilon,\mu\in (0,\frac{1}{2})$, $\nu \in (0,1)$, $C,\theta, \kappa, M>0$ and assume that $f$ belongs to $\mathbf{Hyp}(C,\nu, \theta, \kappa, M)$ (see Definition~\ref{Hypothesis}).
 \be
  \ii Fix $T_0>\pi$. 
   Suppose there exists $s_0= s_0(\epsilon, T_0)$ and for any $T\geq T_0$ there exist $s_1=s_1(\epsilon, T)$ and $s_2=s_2(\epsilon, T)$ with $s_1\leq s_2$ such for any $s\in [s_0,\infty)$,
  \begin{align}\label{eq:AssumBd}
  \mathbb{P}(\Upsilon_T(0)>s)\leq \begin{cases}
  e^{-\frac{4}{3}(1-\epsilon) s^{3/2}} & \text{ if }s\in [s_0, s_1]\cup (s_2, \infty),\\
  e^{-\frac{4}{3}\epsilon s^{3/2}} & \text{ if } s\in (s_1, s_2].
  \end{cases}
 \end{align}
 Let 
 \begin{align}\label{eq:mathbfs}
 \mathbf{s}_0 := \frac{s_0}{1-\frac{2\mu}{3}}, \qquad \mathbf{s}_1 := \frac{\epsilon s_1}{1-\frac{2\mu}{3}}, \qquad \mathbf{s}_2:= \frac{s_2}{1-\frac{2\mu}{3}}.
 \end{align}
  Then, there exists  $s^{\prime}_0= s^{\prime}_0(\epsilon,\mu, C,\nu, \theta, \kappa, M, T_0)$ such that for any $T>T_0$ and any $s\in [\max\{s^{\prime}_0, \mathbf{s}_0\},\infty)$, we have
 \begin{align}\label{eq:ResultBd}
 \mathbb{P}\big(h^{f}_T(0)>s\big)\leq
  \begin{cases}
 e^{- \frac{\sqrt{2}}{3}(1-\epsilon)(1-\mu)s^{3/2}} & \text{ if }s\in [\mathbf{s}_0,\mathbf{s}_1]\cup (\mathbf{s}_2, \infty),\\
 e^{- \frac{\sqrt{2}}{3}\epsilon(1-\mu)s^{3/2}} & \text{ if }s\in (\mathbf{s}_1, \mathbf{s}_2],
 \end{cases}
 \end{align}

 \ii Fix $T_0\in (0,\pi)$. Then, there exists $s^{\prime}_0= s^{\prime}_0(C, \nu, \theta, \kappa, M,T_0)$ satisfying the following: if there exist $s_0=s_0(T_0)>0$ and $c=c(T_0)>0$ such that $\mathbb{P}(\Upsilon_T(0)>s)\leq e^{-cs^{3/2}}$ for all $s\in [s_0,\infty)$ and $T\in [T_0, \pi]$, then,
 \begin{align}\label{eq:GenRoughUpBd}
 \mathbb{P}\big(h^{f}_T(0)>s\big)\leq e^{-\frac{1}{2\sqrt{2}}cs^{3/2}}, \quad \forall s\in  [\max\{s^{\prime}_0,s_0\}, \infty), T\in (T_0,\pi].
\end{align}
\ee
  \ep


 The next proposition provides a lower bound on $\mathbb{P}\big(h^{f}_{T}(0)> s\big)$ in terms of the upper tail probability of the narrow wedge solution. 
 \bp\label{UpTailLowBd}
Fix $\mu\in (0,\frac{1}{2})$,  $n\in \ZZ_{\geq 3}$, $\nu \in (0,1)$, $C,\theta, \kappa, M>0$ and  $T_0>\pi$  and assume that $f\in \mathbf{Hyp}(C,\nu, \theta, \kappa, M)$. Then, there exist $s_0 = s_0(\mu, n, T_0, C, \nu , \theta, \kappa, M)$ and $K=K(\mu)>0$ such that for all $s\geq s_0$ and $T\geq T_0$,
\begin{align}\label{eq:UpTailLowBd}
 \mathbb{P}\big(h^{f}_{T}(0)> s\big) \geq \Big(\mathbb{P}\big(\Upsilon_T(0)>\big(1+\tfrac{2\mu}{3}\big)s\big)\Big)^2 - e^{-Ks^n}.
\end{align}
\ep


We prove Propositions~\ref{SubstituteTheo} and~\ref{UpTailLowBd} in Sections~\ref{UpTailUpBdSEC} and \ref{UpTailLowBdSEC} respectively. In what follows, we complete the proof of Theorem~\ref{Main4Theorem} assuming Propositions~\ref{SubstituteTheo} and~\ref{UpTailLowBd}.

\begin{proof}[Proof of Theorem~\ref{Main4Theorem}]
By Theorem~\ref{GrandUpTheorem}, for any $\epsilon\in (0, \frac{1}{2})$ and $T_0>\pi$, there exists $s_0=s_0(\epsilon, T_0)$ such that for all $T>T_0$ and $s\in [s_0, \infty)$
\begin{align}\label{eq:Hypo}
\mathbb{P}(\Upsilon_T>s)\leq \begin{cases}
e^{-\frac{4}{3}(1-\epsilon) s^{3/2}} & \text{ if }s\in [s_0, \frac{1}{8}\epsilon^2 T]\cup (\frac{9}{16}\epsilon^{-2} T, \infty),\\
  e^{-\frac{4}{3}\epsilon s^{3/2}} & \text{ if } s\in (\frac{1}{8}\epsilon^2 T, \frac{9}{16}\epsilon^{-2} T].
\end{cases}
\end{align}
  For any $\epsilon\in (0,\frac{1}{2})$ and $T>T_0$, \eqref{eq:Hypo} shows that the hypothesis of part (1) of Proposition~\ref{SubstituteTheo} is satisfied with $s_1= \frac{1}{8}\epsilon^2 T$ and $s_2= \frac{9}{16}\epsilon^{-2} T$. Proposition~\ref{SubstituteTheo} yields $s^{\prime}_0=s^{\prime}_0(\epsilon,\mu, T_0, C,\nu, \theta, \kappa, M)$ such that for all $T\geq T_0$ and $s\in [\max\{s^{\prime}_0, s_0/(1-\tfrac{2\mu}{3})\}, \infty)$
  \begin{align}\label{eq:FinalUpTailUpBd}
  \mathbb{P}\big(h^{f}_T(0)>s\big)\leq
  \begin{cases}
 e^{- \frac{\sqrt{2}}{3}(1-\epsilon)(1-\mu)s^{3/2}} & \text{ if }s\in \Big[\frac{s_0}{1-\frac{2\mu}{3}},\frac{\epsilon^3 T}{8(1-\frac{2\mu}{3})} \Big]\cup \Big(\frac{9\epsilon^{-2}T}{16(1-\frac{2\mu}{3})} , \infty\Big),\\
 e^{- \frac{\sqrt{2}}{3}\epsilon(1-\mu)s^{3/2}} & \text{ if }s\in \Big(\frac{\epsilon^3 T}{8(1-\frac{2\mu}{3})}, \frac{9\epsilon^{-2}T}{16(1-\frac{2\mu}{3})} \Big].
 \end{cases}
  \end{align}
  This shows the upper bound on $\mathbb{P}\big(h^{f}_T(0)>s\big)$ when $T_0>\pi$. For any $T_0\in (0, \pi)$, the upper bound on $\mathbb{P}\big(h^{f}_T(0)>s\big)$ follows from \eqref{eq:GenRoughUpBd} for all $T\in [T_0, \pi]$. 
  
  Now, we turn to show the lower bound.
 Let us fix $n=3$. Owing to Proposition~\ref{UpTailLowBd} and the lower bound on the probability $\mathbb{P}(\Upsilon_T(0)\geq s)$ in \eqref{eq:RoughBd1} of Theorem~\ref{GrandUpTheorem}, we observe that the second term $e^{-Ks^3}$ of the r.h.s. of \eqref{eq:UpTailLowBd} is less than the half of the first term when $s$ is large enough. Hence, there exist $s^{\prime}_{0}=s^{\prime}_{0}(\epsilon,\mu, C,\nu, \theta, \kappa, M, T_0)$ such that for all $T\geq  T_0>\pi$ and $s\in [\max\{s^{\prime}_0,s_0/(1+\tfrac{2\mu}{3})\}, \infty)$
\begin{align}\label{eq:FinalUpTailLowBd}
\mathbb{P}\big(h^{f}_T(0)>s\big)\geq
\begin{cases}
\frac{1}{2}e^{-\frac{8}{3}(1+\epsilon)(1+\mu) s^{3/2}} & \text{ if } s\in \Big[\frac{s_0}{1+\frac{2\mu}{3}}, \frac{\epsilon^2 T}{8(1+\frac{2\mu}{3})}\Big],\\
\frac{1}{2}e^{-2^{\frac{9}{2}}\epsilon^{-3}(1+\mu) s^{3/2}} & \text{ if } s\in \Big(\frac{\epsilon^2 T}{8(1+\frac{2\mu}{3})}, \frac{9\epsilon^{-2}T}{16(1+\frac{2\mu}{3})}\Big],\\
\frac{1}{2}e^{-8\sqrt{3}(1+\epsilon)(1+\mu) s^{3/2}} & \text{ if } s\in \Big(\frac{9\epsilon^{-2}T}{16(1+\frac{2\mu}{3})}, \infty\Big).\\
\end{cases}
\end{align}
The sets of three intervals of \eqref{eq:FinalUpTailUpBd} and \eqref{eq:FinalUpTailLowBd} are not same. Note\footnote{The first inequality uses $\epsilon\leq (1-\tfrac{2\mu}{3})(1+\tfrac{2\mu}{3})^{-1}$ for any $\epsilon, \mu\in (0, \frac{1}{2})$ and the second inequality uses $\mu>0$.} that
$\tfrac{\epsilon^3 T}{8(1-\frac{2\mu}{3})}< \tfrac{\epsilon^2T}{8(1+\frac{2\mu}{3})}$ and $\tfrac{9 \epsilon^{-2}T}{16(1-\frac{2\mu}{3})} > \tfrac{9 \epsilon^{-2}T}{16(1+\frac{2\mu}{3})}$.
From this we see that
$
 \Big(\tfrac{\epsilon^2 T}{8(1+\frac{2\mu}{3})}, \tfrac{9\epsilon^{-2}T}{16(1+\frac{2\mu}{3})}\Big] \subset \Big(\tfrac{\epsilon^3 T}{8(1-\frac{2\mu}{3})}, \tfrac{9\epsilon^{-2}T}{16(1-\frac{2\mu}{3})} \Big]$,
$ \Big[\tfrac{s_0}{1-\frac{2\mu}{3}},\tfrac{\epsilon^3 T}{8(1-\frac{2\mu}{3})} \Big]\subset \Big[\frac{s_0}{1+\frac{2\mu}{3}}, \tfrac{\epsilon^2 T}{8(1+\frac{2\mu}{3})}\Big],$ and  $\Big(\tfrac{9\epsilon^{-2}T}{16(1-\frac{2\mu}{3})} , \infty\Big)\subset \Big(\tfrac{9\epsilon^{-2}T}{16(1+\frac{2\mu}{3})}, \infty\Big)$.

By these containments and  \eqref{eq:FinalUpTailUpBd}-\eqref{eq:FinalUpTailLowBd}, for all $s\in \big[\max\big\{s^{\prime}_0,s_0/(1-\tfrac{2\mu}{3}), s_0/(1+\tfrac{2\mu}{3}) \big\}, \infty\big)$ and $T\geq T_0>\pi$, we have $\exp(-c_1s^{\frac{3}{2}})\leq \mathbb{P}\big(h^{f}_T(0)>s\big)\leq \exp(-c_2s^{\frac{3}{2}})$ where
 \begin{align}
 \left.
 \begin{array}{r}
  \frac{\sqrt{2}}{3}(1-\mu)(1-\epsilon)\\
  \frac{\sqrt{2}}{3}(1-\mu)\epsilon\\
  \frac{\sqrt{2}}{3}(1-\mu)(1-\epsilon)
\end{array}
 \right\}\leq c_2<c_1\leq
\begin{cases}
\frac{8}{3}(1+\epsilon)(1+\mu) & \text{ if } s\in \Big[\frac{s_0}{1-\frac{2\mu}{3}}, \frac{\epsilon^3 T}{8(1-\frac{2\mu}{3})}\Big],\\
2^{\frac{9}{2}}\epsilon^{-3}(1+\mu) & \text{ if } s\in \Big(\frac{\epsilon^3 T}{8(1-\frac{2\mu}{3})}, \frac{9\epsilon^{-2}T}{16(1-\frac{2\mu}{3})}\Big],\\
8\sqrt{3}(1+\epsilon)(1+\mu) & \text{ if } s\in \Big(\frac{9\epsilon^{-2}T}{16(1-\frac{2\mu}{3})}, \infty\Big).\\
 \end{cases}
 \end{align}
 The lower bound $\mathbb{P}(h^{f}_{T}(0)>s)\geq e^{-2c_1s^{3/2}}$ for all $T \in [T_0,\pi]$ when $T_0 \in (0,\pi)$ follows by combining the first inequality of \eqref{eq:GenBdRough} with \eqref{eq:UpTailLowBd} (with $n=3$). This completes the proof.
\end{proof}

   \subsubsection{Proof of Proposition~\ref{SubstituteTheo}}\label{UpTailUpBdSEC}
Recall  $h^{f}_T$ and $\Upsilon_T$ from \eqref{eq:ScalCentHeight} and \eqref{eq:DefUpsilon}. By Proposition~\ref{NotMainTheorem}, $\mathbb{P}(h^{f}_T(0)\geq s)= \mathbb{P}(\widetilde{\mathcal{A}}^{f})$ where
 \begin{align}
  \widetilde{\mathcal{A}}^{f}:= \Big\{\int^{\infty}_{-\infty} e^{T^{\frac{1}{3}}\Big(\Upsilon_T(y)+ f(-y)dy\Big)} dy\geq e^{T^{\frac{1}{3}}s}\Big\}.
\end{align}
Let $\zeta_n:=\frac{n}{s^{1+\delta}}$, $n\in \ZZ$ and fix $\tau\in (0,1)$ such that $\nu+\tau<1$. We define the following events:
 \begin{align}
 \widetilde{E}_n&:= \Big\{\Upsilon_T(\zeta_n)\geq -\frac{1-2^{-1}\tau}{2^{2/3}}\zeta^2_n + \big(1-\tfrac{2\mu}{3}\big)s\Big\}\label{eq:tildeE}\\
 \widetilde{F}_n&:= \Big\{\Upsilon_T(y)\geq -\frac{1- \tau}{2^{2/3}} y^2 + \big(1-\frac{\mu}{3}\big)s \quad \text{ for some }y \in [\zeta_n,\zeta_{n+1}]\Big\}.\label{eq:tildeF}
\end{align}
In the same way as in \eqref{eq:BasicStep}, we write
\begin{align}\label{eq:BasicStep3}
\mathbb{P}\big(\widetilde{\mathcal{A}}^{f}\big)\leq \sum_{n\in \ZZ} \mathbb{P}(\widetilde{E}_n) +\mathbb{P}\Big(\widetilde{\mathcal{A}}^{f}\cap \big(\bigcup_{n\in \ZZ} \widetilde{E}_n\big)^c\Big).
\end{align}

 From now on, we will fix some $T>T_0>\pi$ and assume that there exist $s_0=s_0(\epsilon, T_0)$, $s_1=s_1(\epsilon, T)$ and $s_2=s_2(\epsilon, T)$ with $s_1\leq s_2$ such that for all $s\in [s_0,\infty)$ \eqref{eq:AssumBd} is satisfied. In the next result, we demonstrate some upper bound on the first term on the r.h.s. of \eqref{eq:BasicStep3}.

 \bl\label{UpSumProbBd}
  There exist $\bar{s}=\bar{s}(\epsilon, T_0)$ and $\Theta=\Theta(\epsilon, T_0)$ such that for all $s\in [\max\{\bar{s},\mathbf{s}_0\}, \infty)$,
  \begin{align}\label{eq:TotSumBd}
  \sum_{n\in \ZZ}\mathbb{P}\big(\widetilde{E}_n\big)\leq \begin{cases}
  \Theta e^{-\frac{4}{3}(1-\epsilon)(1-\mu) s^{3/2}} & \text{ if } s\in [\mathbf{s}_0, \mathbf{s}_1]\cup (\mathbf{s}_2, \infty),\\
 \Theta e^{-\frac{4}{3}\epsilon(1-\mu) s^{3/2}} & \text{ if } s\in (\mathbf{s}_1,\mathbf{s}_2],
  \end{cases}
  \end{align}
where $\mathbf{s}_0$, $\mathbf{s}_1$ and $\mathbf{s}_2$ are defined in \eqref{eq:mathbfs}.
 \el
\begin{proof}
We first prove \eqref{eq:TotSumBd} when $s\in [\mathbf{s}_0,\mathbf{s}_1]$. If $[\mathbf{s}_0,\mathbf{s}_1]$ is an empty interval, then, nothing to prove. Otherwise, fix any $s\in [\mathbf{s}_0,\mathbf{s}_1]$. Let us denote $$\mathcal{S}_1:= [0,(1-\epsilon)s_1], \quad \mathcal{S}_2:= ((1-\epsilon)s_1, s_2-s_0], \quad \mathcal{S}_3:= (s_2-s_0, \infty).$$

\begin{claim}
\begin{align}\label{eq:IndEventBd}
\mathbb{P}\big(\widetilde{E}_n\big)\leq \begin{cases}
\exp\Big(-\frac{4}{3}(1-\epsilon)\Big(\big(1-\tfrac{2\mu}{3}\big)s+\frac{\tau \zeta^2_n}{2^{5/3}}\Big)^{\frac{3}{2}}\Big) & \text{ when }\frac{\tau \zeta_n^2}{2^{5/3}}\in \mathcal{S}_1\cup \mathcal{S}_3,\\
\exp\Big(-\frac{4}{3}\epsilon\Big(\big(1-\tfrac{2\mu}{3}\big)s+\frac{\tau \zeta^2_n}{2^{5/3}}\Big)^{\frac{3}{2}}\Big) & \text{ when }\frac{\tau \zeta_n^2}{2^{5/3}} \in \mathcal{S}_2.
\end{cases}
\end{align}
\end{claim}

\begin{proof}
Note that $s_0\leq (1-\frac{2\mu}{3})s\leq \epsilon s_1$. This implies $\big(1-\frac{2\mu}{3}\big)s + 2^{-5/3}\tau \zeta^2_n $ is bounded above by $ \epsilon s_1+(1-\epsilon)s_1= s_1$ whenever $2^{-5/3}\tau \zeta^2_n\leq (1-\epsilon)s_1$ whereas it is bounded below by $s_0 + s_2-s_0 =s_2$ if $2^{-5/3}\tau \zeta^2_n>s_2-s_0$. Owing to this and \eqref{eq:AssumBd}, we have
\begin{align}\label{eq:Case1}
\mathbb{P}(\widetilde{E}_n)\leq \exp\Big(-\frac{4}{3}(1-\epsilon)\Big(\big(1-\tfrac{2\mu}{3}\big)s+\frac{\tau \zeta^2_n}{2^{5/3}}\Big)^{\frac{3}{2}}\Big) \quad \text{ when }\frac{\tau \zeta_n^2}{2^{5/3}}\in \mathcal{S}_1\cup \mathcal{S}_3.
\end{align}
Furthermore, $(1-\frac{2\mu}{3})s+2^{-5/3}\tau \zeta^2_n$ is greater than $s_0$ when $s\geq \mathbf{s}_0$. Thanks to $\epsilon<\frac{1}{2}$, one can now see the following from \eqref{eq:AssumBd}:
\begin{align}\label{eq:Case2}
\mathbb{P}(\widetilde{E}_n)\leq \exp\Big(-\frac{4}{3}\epsilon\Big(\big(1-\tfrac{2\mu}{3}\big)s+\frac{\tau \zeta^2_n}{2^{5/3}}\Big)^{\frac{3}{2}}\Big) \quad \text{ when } \frac{\tau \zeta_n^2}{2^{5/3}}\in \mathcal{S}_2.
\end{align}
Combining \eqref{eq:Case1} and \eqref{eq:Case2}, we get \eqref{eq:IndEventBd}.
\end{proof}


Let $n_0=n_0(s,\delta, \tau)<n^{\prime}_0 = n^{\prime}_0(s,\delta, \tau)\in \NN$ be such that $2^{-5/3}\tau \zeta^2_n\in \mathcal{S}_2$ for all integer $n$ in $[n_0, n^{\prime}_0]\cup [-n^{\prime}_0, -n_0]$. Using the reverse Minkowski's inequality,
\begin{align}\label{eq:BreakSquare}
\frac{\tau \zeta^2_n}{2^{5/3}}\geq \frac{\tau \zeta^2_{n_0}}{2^{5/3}} + \frac{\tau \zeta^2_{|n|-n_0}}{2^{5/3}}, \quad \forall n\in \{[n_0, n^{\prime}_0]\cup [- n^{\prime}_0, -n_0]\}\cap \ZZ.
\end{align}
 Owing to $s_1\geq \epsilon^{-1}(1-\tfrac{2\mu}{3}) s$, we get
\begin{align}\label{eq:LeftBound}
\frac{\tau \zeta^2_{n_0}}{2^{5/3}}\geq (1-\epsilon)s_1\geq \epsilon^{-1}\big(1-\tfrac{2\mu}{3}\big)(1-\epsilon)s.
\end{align}
Combining \eqref{eq:BreakSquare} with \eqref{eq:LeftBound} and invoking the reverse Minkowski's inequality  yields
\begin{equation}
\Big((1-\tfrac{2\mu}{3})s+\frac{\tau \zeta^2_{n}}{2^{5/3}}\Big)^{\frac{3}{2}} \geq \Big(\epsilon^{-1}\big(1-\tfrac{2\mu}{3}\big)(1-\epsilon)s\Big)^{\frac{3}{2}} + \frac{\tau^{3/2}\zeta^3_{|n|-n_0}}{2^{5/2}}, \quad \text{when}\quad\frac{\tau \zeta^2_{n}}{2^{5/3}}\in \mathcal{S}_2.
\end{equation}
Plugging this into \eqref{eq:IndEventBd}, summing in a similar way as in the proof of Lemma~\ref{MeshBound} and noticing
$$\epsilon\Big(\epsilon^{-1}(1-\frac{2\mu}{3})(1-\epsilon)\Big)^{\frac{3}{2}}> \epsilon^{-\frac{1}{2}}(1-\epsilon)^{\frac{1}{2}}(1-\mu)(1-\epsilon)>(1-\mu)(1-\epsilon),$$
 we arrive at
\begin{align}
\sum_{n: 2^{-5/3}\tau \zeta^2_n\in \mathcal{S}_2}\mathbb{P}\big(\widetilde{E}_n\big)\leq C_1 \exp\Big(-\frac{4}{3}(1-\epsilon)(1-\mu)s^{\frac{3}{2}}\Big)\label{eq:MidSum}
\end{align}
for some $C_1=C_1(\epsilon, T_0)$ when $s$ is large enough.
From the reverse Minkowski's inequality,
 \begin{align}\label{eq:RevMin}
 \Big(\big(1-\tfrac{2\mu}{3}\big)s+\frac{\tau \zeta^2_n}{2^{5/3}}\Big)^{\frac{3}{2}}\geq \big(1-\tfrac{2\mu}{3}\big)^{\frac{3}{2}}s^{\frac{3}{2}} + \frac{\tau^{3/2} \zeta^3_{|n|}}{2^{5/2}}.
\end{align}
 Applying \eqref{eq:RevMin} to the r.h.s. of \eqref{eq:IndEventBd} for all $n$ such that $2^{-5/3}\tau \zeta^2_n\in \mathcal{S}_1\cup \mathcal{S}_3$ and summing in a similar way as in the proof of Lemma~\ref{MeshBound} yields
\begin{align}
\sum_{n:2^{-5/3}\tau \zeta^2_n\in \mathcal{S}_1\cup \mathcal{S}_3}\mathbb{P}\big(\widetilde{E}_n\big)\leq C_2\exp\Big(-\frac{4}{3}(1-\epsilon)\big(1-\tfrac{2\mu}{3}\big)^{3/2}s^{\frac{3}{2}}\Big)  \label{eq:EndSum}
\end{align}
for some $C_2=C_2(\epsilon, T_0)$.
Adding \eqref{eq:MidSum} and \eqref{eq:EndSum} and noticing that $\big(1-\tfrac{2\mu}{3}\big)^{\frac{3}{2}}\geq (1-\mu)$, we obtain \eqref{eq:TotSumBd} if $s\in [\mathbf{s}_0,\mathbf{s}_1]\cap [\bar{s}, \infty)$ where $\bar{s}$ depends on $\epsilon$ and $T_0$.

Now, we turn to the case when $s\in \big\{(\mathbf{s}_1, \mathbf{s}_2]\cup (\mathbf{s}_2, \infty)\big\}\cap[\mathbf{s}_0,\infty)$. Owing to \eqref{eq:AssumBd}, for all $n\in \ZZ$ and $s\in [\mathbf{s}_0,\infty)$,
\begin{align}
\mathbb{P}(\widetilde{E}_n)\leq \begin{cases}
\exp\Big(-\frac{4}{3}\epsilon\Big(\big(1-\tfrac{2\mu}{3}\big)s+\frac{\tau \zeta^2_n}{2^{5/3}}\Big)^{\frac{3}{2}}\Big) & \text{ if } s\in (\mathbf{s}_1, \mathbf{s}_2],\\\exp\Big(-\frac{4}{3}(1-\epsilon)\Big(\big(1-\tfrac{2\mu}{3}\big)s+\frac{\tau \zeta^2_n}{2^{5/3}}\Big)^{\frac{3}{2}}\Big) & \text{ if }s\in  (\mathbf{s}_2, \infty).
\end{cases}\label{eq:IndEventBd2}
\end{align}
  Applying \eqref{eq:RevMin} and summing the r.h.s. of \eqref{eq:IndEventBd2} in the same way as  \eqref{eq:EndSum}, we find \eqref{eq:TotSumBd}.
\end{proof}

 Now, we show an analogue of Lemma~\ref{EffOfInitCond}.

 \bl\label{SumProbBdLem}
There exists $s^{\prime} = s^{\prime}(\epsilon,  T_0, C, \nu, \theta, \kappa, M)$ such that for all $s\geq s^{\prime}$,
\begin{align}\label{eq:Containment3}
\Big(\bigcup_{n\in \ZZ}\widetilde{E}_n\Big)^c \cap \Big(\bigcup_{n\in \ZZ} \widetilde{F}_n\Big)^c \subseteq (\widetilde{\mathcal{A}}^f)^c.
\end{align}
 \el
\begin{proof}
Assume the event of the l.h.s. of \eqref{eq:Containment3} occurs. By \eqref{eq:InMomBd} of Definition~\ref{Hypothesis} and $\tau+\nu<1$,
\begin{align}
\int_{-\infty}^{\infty} e^{T^{1/3}\big(\Upsilon_T(y)+f(-y)\big)} dy\leq \int_{-\infty}^{\infty} e^{T^{1/3}\big(C-\frac{1-\tau}{2^{2/3}}y^2+ (1-\frac{\mu}{3})s + \frac{\nu}{2^{2/3}} y^2\big)} dy\leq  \frac{K}{T^{1/6}} e^{(1-\frac{\mu}{3})sT^{1/3} }.
\end{align}
for some $K=K(C,T,\tau,\nu)>0$. There exists $s^{\prime}= s^{\prime}(\mu, T_0,  C,\nu, \theta, \kappa, M)$ such that the right hand side of the above inequality is bounded above by $\exp(sT^{\frac{1}{3}})$, thus confirming \eqref{eq:Containment3}.
\end{proof}

 Applying \eqref{eq:Containment3} and Bonferroni's union bound, (see \eqref{eq:BonfrBd} for a similar inequality)
\begin{align}\label{eq:MyNewEq}
\mathbb{P}\Big(\widetilde{\mathcal{A}}^{f} \cap \Big(\bigcup_{n\in \ZZ} \widetilde{E}_n\Big)^c\Big)\leq \sum_{n\in \ZZ} \mathbb{P}\big( \widetilde{E}^c_{n-1}\cap \widetilde{E}^c_{n+1}\cap \widetilde{F}_n\big).
\end{align}

\bl\label{BigMaxApp}
 There exists $s^{\prime\prime}= s^{\prime\prime}(\epsilon,\mu, T_0)$ and $\Theta=\Theta(\epsilon, T_0)$ such that for all $s\in [\max\{s^{\prime\prime}, \mathbf{s}_0\}, \infty)$,
\begin{align}\label{eq:NoBigMaxRes}
\sum_{n\in \ZZ} \mathbb{P}\big(\widetilde{E}^c_{n-1}\cap \widetilde{E}^c_{n+1}\cap \widetilde{F}_n\big) \leq \begin{cases}
\Theta e^{-\frac{\sqrt{2}}{3}(1-\epsilon)(1-\mu) s^{3/2}} & \text{ if } s\in [\mathbf{s}_0, \mathbf{s}_1]\cup (\mathbf{s}_2, \infty),\\
 \Theta e^{-\frac{\sqrt{2}}{3}\epsilon(1-\mu) s^{3/2}} & \text{ if } s\in (\mathbf{s}_1,\mathbf{s}_2].
 \end{cases}
\end{align}
See \eqref{eq:mathbfs} for the definitions of $\mathbf{s}_0$, $\mathbf{s}_1$ and $\mathbf{s}_2$.
\el

 \begin{proof}
  We need to bound $\mathbb{P}\big(\widetilde{E}^c_{n-1}\cap \widetilde{E}^c_{n+1}\cap \widetilde{F}_n\big)$ for all $n\in \ZZ$. Define
 \begin{align}\label{eq:TwoEvents}
 \widetilde{\mathcal{E}}_n := \Big\{ \Upsilon_T(\zeta_n) \geq -\frac{1+2^{-1}\tau}{2^{2/3}}\zeta^2_n- s^{\frac{2}{3}}\Big\}, \qquad \textrm{for } n\in \ZZ.
\end{align}
    We begin with the following inequality
\begin{align*}
\mathbb{P}\big(\widetilde{E}^c_{n-1} \cap \widetilde{E}^c_{n+1} \cap \widetilde{F}_n\big)\leq \mathbb{P} \big((\widetilde{E}^c_{n-1}\cap \widetilde{\mathcal{E}}_{n-1}) \cap (\widetilde{E}^c_{n+1}\cap \widetilde{\mathcal{E}}_{n+1}) \cap \widetilde{F}_n\big) + \mathbb{P}( \widetilde{\mathcal{E}}^c_{n-1})+ \mathbb{P}(\widetilde{\mathcal{E}}^c_{n+1}).
\end{align*}
We will bound each term on the r.h.s. above.
    Proposition~\ref{NotMainTheorem} provides $s^{\prime\prime}:=s^{\prime\prime}(\epsilon, T_0)$, $K=K(\epsilon, T_0)>0$ and the following upper bound\footnote{Taking $\epsilon=\delta$ in Proposition~\ref{NotMainTheorem} the r.h.s. of \eqref{eq:PrevRes1} $\leq \exp(-T^{\frac{1}{3}}\frac{4(1-\epsilon)s^{5/2}}{15\pi}) + \exp(-Ks^{3-\epsilon})$.} for $s\geq s^{\prime\prime}$ and $T\geq T_0$
   \[\mathbb{P}(\widetilde{\mathcal{E}}^c_n)\leq \exp\Big(- T^{\frac{1}{3}}\frac{4}{15\pi}(1-\epsilon)\big(s^{\frac{2}{3}} +\frac{\tau \zeta^2_n}{2^{5/3}}\big)^{\frac{5}{2}}\Big)+ \exp\Big(-K\big(s^{\frac{2}{3}} +\frac{\tau \zeta^2_n}{2^{5/3}}\big)^{3-\epsilon}\Big).\]
   Summing over all $n\in \ZZ$ (in the same way as in Lemma~\ref{MeshBound}) yields
    \begin{align}\label{eq:SumRes}
    \sum_{n\in \ZZ}  \big(\mathbb{P}( \widetilde{\mathcal{E}}^c_{n-1})+ \mathbb{P}(\widetilde{\mathcal{E}}^c_{n+1})\big) \leq e^{-T^{1/3} \frac{4}{15\pi}(1-\epsilon)s^{5/3}} + e^{-Ks^{2-2\epsilon/3}}.
    \end{align}

\begin{claim}
There exists $s^{\prime\prime} = s^{\prime\prime}(\epsilon, \mu, T_0)$, such that for all $s\geq s^{\prime\prime}$, $T\geq T_0$ and $n\in \ZZ$,
    \begin{align}\label{eq:CplexBd}
   \mathbb{P} \big( (\widetilde{E}^c_{n-1}\cap \widetilde{\mathcal{E}}_{n-1}) \cap (\widetilde{E}^c_{n+1}\cap \widetilde{\mathcal{E}}_{n+1}) \cap \widetilde{F}_n\big)\leq 2\mathbb{P}\Big(\Upsilon_T(0)\geq 2^{-\frac{11}{3}}\zeta^2_n+\frac{1}{2}\big(1-\tfrac{2\mu}{3}\big)s\Big).
    \end{align}
\end{claim}

\begin{figure}[t]
\includegraphics[width=.5\linewidth]{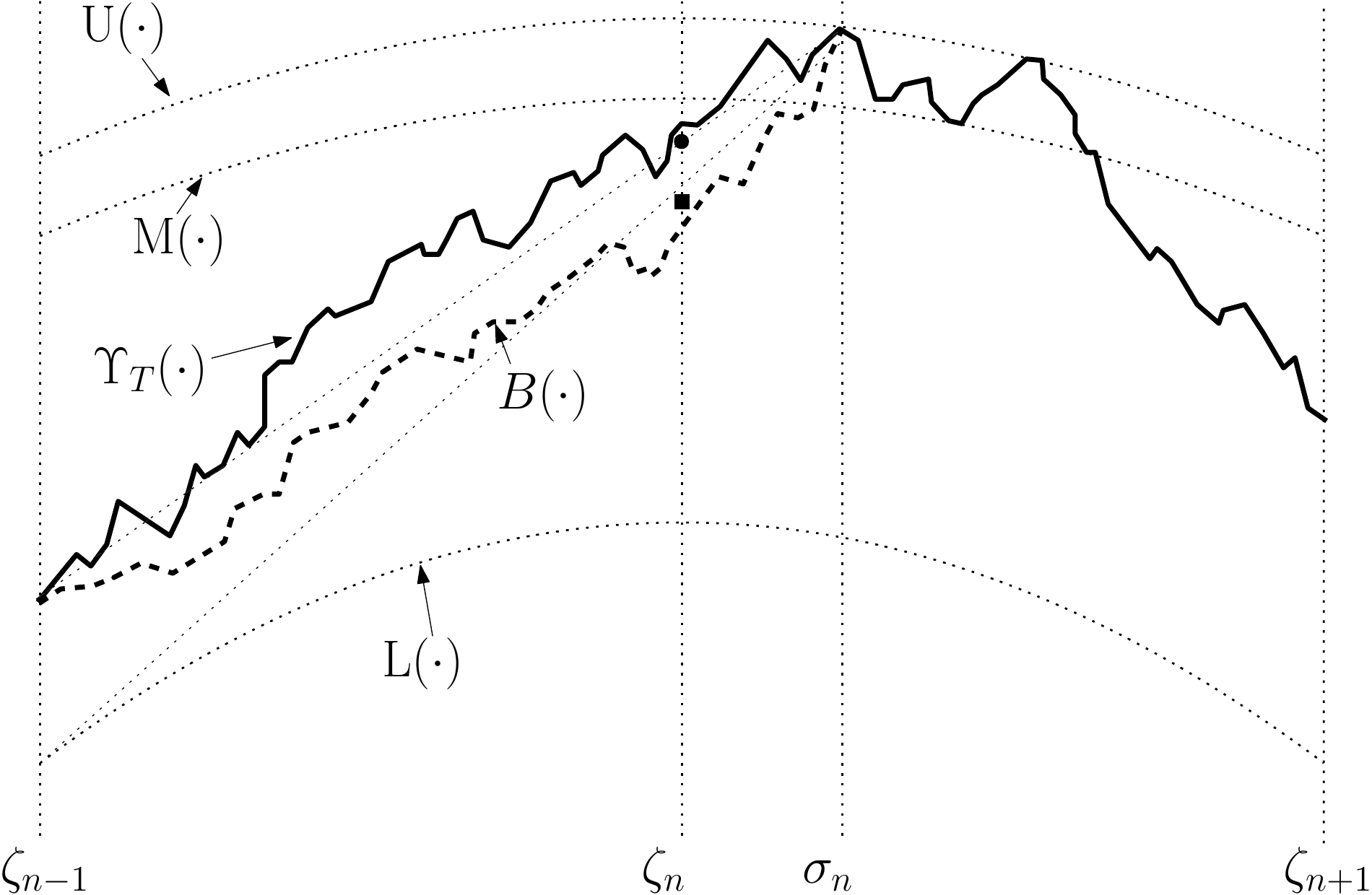}
\caption[]{Illustration from the proof of \eqref{eq:CplexBd}. The three parabolas are $U(\cdot)$, $M(\cdot)$ and $L(\cdot)$. The solid black curve is $\Upsilon^{(1)}_T(\cdot)$ when $\widetilde{E}^c_{n-1}\cap \widetilde{\mathcal{E}}_{n-1} \cap \widetilde{E}^c_{n+1}\cap \widetilde{\mathcal{E}}_{n+1} \cap \widetilde{F}_n$ occurs. Note that $\Upsilon^{(1)}_T(\cdot)$  stays in between $M(\cdot)$ and $L(\cdot)$ at $\zeta_{n-1}$ and $\zeta_{n+1}$. The rightmost point in $(\zeta_n,\zeta_{n+1})$ where $\Upsilon^{(1)}_T(\cdot)$ hits $U(\cdot)$ is labeled  $\sigma_n$. The event that the black curve stays above the square at $\zeta_n$ is $\widetilde{\mathfrak{B}}_n$ and $\mathbb{P}_{\mathbf{H}_{2T}}(\widetilde{\mathfrak{B}}_n)$ (see \eqref{eq:DefP1Up} for $\mathbb{P}_{\mathbf{H}_{2T}}$) is the probability of $\widetilde{\mathfrak{B}}_n$ conditioned on the sigma algebra $\mathcal{F}_{\mathrm{ext}}\big(\{1\}\times (\zeta_{n-1}, \sigma_n)\big)$. On the other hand, $\widetilde{\mathbb{P}}_{\mathbf{H}_{2T}}(\widetilde{\mathfrak{B}}_n)$ (see~\eqref{eq:DefP2Up} for $\widetilde{\mathbb{P}}_{\mathbf{H}_{2T}}$) is the probability of $\widetilde{\mathfrak{B}}_n$ under the free Brownian bridge (scaled by $2^{\frac{1}{3}}$) measure on the interval $(\zeta_{n-1}, \sigma_n)$ with same starting and end point as $\Upsilon^{(1)}_T(\cdot)$. The dashed black curve is such a free Brownian bridge coupled to $\Upsilon^{(1)}_T(\cdot)$ so that $B(y)\leq \Upsilon^{(1)}_T(y)$ for all $y \in (\zeta_{n-1}, \sigma_n)$. Owing to this coupling, $\mathbb{P}_{\mathbf{H}_{2T}}(\widetilde{\mathfrak{B}}_n)\geq \widetilde{\mathbb{P}}_{\mathbf{H}_{2T}}(\widetilde{\mathfrak{B}}_n)$. The probability of $B(\sigma_n)$ staying above the bullet point is $\frac{1}{2}$ which implies that $\widetilde{\mathbb{P}}_{\mathbf{H}_{2T}}(\widetilde{\mathfrak{B}}_n)\geq \frac{1}{2}$. Consequently, we can bound the probability of $(\widetilde{E}^c_{n-1}\cap \widetilde{\mathcal{E}}_{n-1}) \cap (\widetilde{E}^c_{n+1}\cap \widetilde{\mathcal{E}}_{n+1}) \cap \widetilde{F}_n$ by $2\mathbb{P}(\widetilde{\mathfrak{B}}_n)$ (see \eqref{eq:AllIsAbove}). The expected value of $\mathbb{P}(\widetilde{\mathfrak{B}}_n)$ can be bounded above by the upper tail probability of $\Upsilon^{(1)}_T(\zeta_n)+\frac{\zeta^2_n}{2^{2/3}}$ (see~\eqref{eq:EachBd}). The upper bound in \eqref{eq:CplexBd} follows then  by invoking Proposition~\ref{StationarityProp}.}
\label{fig:Figure2}
\end{figure}

\begin{proof}
We parallel the proof of \cite[Proposition~4.4]{CH14} (see also
    \cite[Lemma~4.1]{CorHam16}). Figure~\ref{fig:Figure2} illustrates the main objects in this proof and the argument (whose details we now provide).

By Proposition~\ref{NWtoLineEnsemble} the  curve $2^{-\frac{1}{3}}\Upsilon^{(1)}_T(\cdot)$ from the KPZ line ensemble $\{2^{-\frac{1}{3}}\Upsilon^{(n)}_T(x)\}_{n\in \NN,x\in \RR}$ has the same distribution as $2^{-\frac{1}{3}}\Upsilon_T(\cdot)$. For the rest of this proof, we replace $\Upsilon_T$ by $\Upsilon^{(1)}_T$ in the definitions of $\{\widetilde{E}_n\}_{n}$, $\{\widetilde{F}_n\}_n$ and $\{\widetilde{\mathcal{E}}_n\}_{n}$.
   We define the following three curves:
   \begin{align}
   U(y):= -\tfrac{(1-\tau)}{2^{2/3}}y^2+ \Big(1-\tfrac{\mu}{3}\Big)s, \quad L(y):= -\tfrac{(1+2^{-1}\tau)}{2^{2/3}}y^2 -s^{\frac{2}{3}}, \quad M(y):= -\tfrac{(1-\tau)}{2^{2/3}}y^2& + (1-\tfrac{\mu}{3})s.
\end{align}
If $\widetilde{E}^c_{n-1}\cap \widetilde{\mathcal{E}}_{n-1}$ and $\widetilde{E}^c_{n-1}\cap \widetilde{\mathcal{E}}_{n-1}$ occurs, then, $\Upsilon^{(1)}_T(\cdot)$ stays in between the curves $M(\cdot)$ and $L(\cdot)$ at the points $\zeta_{n-1}$ and $\zeta_{n+1}$ respectively. If $\widetilde{F}_n$ occurs, then, $\Upsilon^{(1)}_T(\cdot)$ touches the curve $U(\cdot)$ at some point in the interval $[\zeta_{n}, \zeta_{n+1}]$. Therefore, on the event $(\widetilde{E}^c_{n-1}\cap \widetilde{\mathcal{E}}_{n-1})\cap (\widetilde{E}^c_{n+1} \cap \widetilde{\mathcal{E}}_{n+1})\cap F_n$, $\Upsilon^{(1)}_T(\cdot)$ hits $U(\cdot)$ somewhere in the interval $(\zeta_n, \zeta_{n+1})$ whereas it stays in between $M(\cdot)$ and $L(\cdot)$ at the points $\zeta_{n-1}$ and $\zeta_{n+1}$. Let us define $\sigma_n := \sup \Big\{y\in (\zeta_n, \zeta_{n+1}): \Upsilon^{(1)}_{T}(y)\geq U(y) \Big\}.$

Recall that $\zeta_{n-1}< \zeta_{n}< \zeta_{n+1}$. Consider the following crossing event
   \begin{align}\label{eq:DefMathfrakB}
    \widetilde{\mathfrak{B}}_n := \Big\{\Upsilon^{(1)}_T(\zeta_n)\geq \frac{\sigma_n-\zeta_n}{\sigma_n -\zeta_{n-1}}L(\zeta_{n-1})+ \frac{\zeta_n-\zeta_{n-1}}{\sigma_n- \zeta_{n-1}}U(\sigma_n)\Big\}.
   \end{align}
   We will use the following abbreviation for the probability measures
   \begin{align}
\mathbb{P}_{\mathbf{H}_{2T}}&:=\mathbb{P}^{1,1,(\zeta_{n-1}, \sigma_n), 2^{-\frac{1}{3}}\Upsilon^{(1)}_T(\zeta_{n-1}), 2^{-\frac{1}{3}}\Upsilon^{(1)}_T(\sigma_{n}), +\infty, 2^{-\frac{1}{3}}\Upsilon^{(2)}_T}_{\mathbf{H}_{2T}},\label{eq:DefP1Up}\\
\widetilde{\mathbb{P}}_{\mathbf{H}_{2T}}&:=\mathbb{P}^{1,1, (\zeta_{n-1}, \sigma_n), 2^{-\frac{1}{3}}\Upsilon^{(1)}_T(\zeta_{n-1}), 2^{-\frac{1}{3}}\Upsilon^{(1)}_T(\sigma_{n}), +\infty, -\infty}_{\mathbf{H}_{2T}}.\label{eq:DefP2Up}
\end{align}

   Since, $(\zeta_{n-1}, \sigma_n)$ is a $\{1\}$-stopping domain (see Definition~\ref{LineEnsemble}) for the KPZ line ensemble, the strong $\mathbf{H}_{2T}$-Brownian Gibbs property (see Lemma~2.5 of \cite{CorHam16}) applies to show that
 \begin{align}\label{eq:TildeB}
 \mathbb{E}&\Big[\mathbbm{1}\big((\widetilde{E}^c_{n-1}\cap \widetilde{\mathcal{E}}_{n-1})\cap (\widetilde{E}^c_{n+1}\cap \widetilde{\mathcal{E}}_{n+1})\cap \widetilde{F}_n\big)\cdot \mathbbm{1}(\widetilde{\mathfrak{B}}_n)|\mathcal{F}_{\mathrm{ext}}\big(\{1\}\times (\zeta_{n-1}, \sigma_{n})\big)\Big] \\ = &\mathbbm{1}\big((\widetilde{E}^c_{n-1}\cap \widetilde{\mathcal{E}}_{n-1})\cap (\widetilde{E}^c_{n+1}\cap \widetilde{\mathcal{E}}_{n+1})\cap \widetilde{F}_n\big)\cdot \mathbb{P}_{\mathbf{H}_{2T}}(\widetilde{\mathfrak{B}}_n).
\end{align}

   By Proposition~\ref{Coupling1}, there exists a monotone coupling\footnote{If $B$ is $\mathbb{P}_{\mathbf{H}_{2T}}$ distributed and $\tilde{B}$ is $\widetilde{\mathbb{P}}_{\mathbf{H}_{2T}}$ distributed, then, under the coupling, $B(y)\geq \tilde{B}(y), \space\forall y\in ( \zeta_{n-1}, \sigma_n)$} between the probability measures $\mathbb{P}_{\mathbf{H}_{2T}}$ and $\widetilde{\mathbb{P}}_{\mathbf{H}_{2T}}$.
 Using this and the fact that the probability of$\widetilde{\mathfrak{B}}_n$ increases under pointwise increase of its sample paths, we have $\mathbb{P}_{\mathbf{H}_{2T}}(\widetilde{\mathfrak{B}}_n)\geq \widetilde{\mathbb{P}}_{\mathbf{H}_{2T}}(\widetilde{\mathfrak{B}}_n)$. Since $\widetilde{\mathbb{P}}_{\mathbf{H}_{2T}}$ is the law of a Brownian bridge on the interval $(\zeta_{n-1}, \sigma_n)$ with end points $2^{-\frac{1}{3}}\Upsilon^{(1)}_T(\zeta_{n-1})$ and $2^{-\frac{1}{3}}\Upsilon^{(1)}_T(\sigma_n)$, the probability that it stays above the line joining the two end points at a given intermediate point is $\frac{1}{2}$. Therefore $\widetilde{\mathbb{P}}_{\mathbf{H}_{2T}}(\widetilde{\mathfrak{B}}_n)\geq\frac{1}{2}$.  Plugging this into \eqref{eq:TildeB} and taking expectation yields
    \begin{align}
\mathbb{P}\big( (\widetilde{E}^c_{n-1}\cap \widetilde{\mathcal{E}}_{n-1})\cap (\widetilde{E}^c_{n+1}\cap \widetilde{\mathcal{E}}_{n-1})\cap \widetilde{F}_n\big)\leq 2\mathbb{E}\Big[\mathbbm{1}((\widetilde{E}^c_{n-1}\cap \widetilde{\mathcal{E}}_{n-1})\cap (\widetilde{E}^c_{n+1}\cap \widetilde{\mathcal{E}}_{n+1})\cap \widetilde{F}_n)\cdot \mathbbm{1}(\widetilde{\mathfrak{B}}_n)\Big].&&\label{eq:AllIsAbove}
\end{align}
  Now, we bound the r.h.s. of \eqref{eq:AllIsAbove}.
 Note the following holds\footnote{To see the first inequality of \eqref{eq:Series3}, note that $(\zeta_n-\zeta_{n-1})/(\sigma_n-\zeta_{n-1})\geq \frac{1}{2}$ and $\sigma^2_n\geq \zeta^2_n -2|\zeta_n|s^{-(1+\delta)}$;  the second inequality follows from $8^{-1}\zeta^2_n- 2|\zeta_n|s^{-(1+\delta)}\geq 0$ for all $|n|\geq 16$.} for all $n\in \ZZ$:
  \begin{align}
  \frac{(\sigma_n -\zeta_{n})\zeta^2_{n-1}+(\zeta_n -\zeta_{n-1})\sigma^2_{n}}{\sigma_n - \zeta_{n-1}}- \zeta^2_n &= (\sigma_n -\zeta_{n})(\zeta_n -\zeta_{n-1})\leq \frac{1}{s^{2+2\delta}}, &&&\label{eq:Series1}\\
  \frac{-\frac{1}{2}(\sigma_n -\zeta_{n})\zeta^2_{n-1}+(\zeta_n -\zeta_{n-1})\sigma^2_{n}}{\sigma_n - \zeta_{n-1}}  +\frac{1}{2}\zeta^2_n &= -\frac{1}{2}(\sigma_n -\zeta_{n})(\zeta_n -\zeta_{n-1}) + \frac{3}{2}\frac{\zeta_n -\zeta_{n-1}}{\sigma_n -\zeta_{n-1}}\sigma^2_{n}, &&&\label{eq:Series2}\\
  \frac{3}{2}\frac{\zeta_n -\zeta_{n-1}}{\sigma_n -\zeta_{n-1}}\sigma^2_{n}-\frac{1}{2}\zeta^2_n &\geq \frac{1}{4}\zeta^2_n - 2\frac{|\zeta_n|}{s^{1+\delta}} \geq \frac{1}{8}\zeta^2_n - \frac{32}{s^{2+2\delta}}.&&&\label{eq:Series3}
\end{align}
   Combining \eqref{eq:Series1}, \eqref{eq:Series2} and \eqref{eq:Series3} yields
  \begin{align*}
  \frac{\sigma_n-\zeta_n}{\sigma_n -\zeta_{n-1}}L(\zeta_{n-1})+ \frac{\zeta_n-\zeta_{n-1}}{\sigma_n- \zeta_{n-1}}U(\sigma_n)\geq -\frac{(1-8^{-1}\tau)}{2^{2/3}}\zeta^2_n -\frac{(4+34\tau)}{2^{2/3}s^{2+2\delta}}+ \frac{1}{2}\Big(\big(1-\frac{\mu}{3}\big)s - s^{\frac{2}{3}}\Big).
  \end{align*}
%
  This implies that when $\widetilde{\mathfrak{B}}_n$ occurs, $\Upsilon^{(1)}_T(\zeta_{n}) $ will be greater than the r.h.s above. The r.h.s is bounded below by $-2^{-\frac{2}{3}}(1-8^{-1}\tau)+\frac{1}{2}\big(1-\tfrac{2\mu}{3}\big)s$ when $s$ is large enough. Hence, we have
   \begin{align}
   \text{r.h.s. of \eqref{eq:AllIsAbove}} \leq 2\,\mathbb{P}(\widetilde{\mathfrak{B}}_n)\leq 2 \,\mathbb{P}\left(\Upsilon^{(1)}_T(\zeta_n)\geq -\frac{(1-8^{-1}\tau)}{2^{2/3}}\zeta^2_n + \frac{1}{2}\big(1-\tfrac{2\mu}{3}\big)s\right).\label{eq:EachBd}
   \end{align}
   Now, the claim follows from \eqref{eq:AllIsAbove} and \eqref{eq:EachBd} by recalling that $\Upsilon^{(1)}_T(\zeta^2_n)+\frac{\zeta^2_n}{2^{2/3}}\stackrel{d}{=}\Upsilon_T(0)$.
\end{proof}

Using \eqref{eq:CplexBd} and a similar analysis as in Lemma~\ref{UpSumProbBd}, there exist $s^{\prime\prime}=s^{\prime\prime}(\epsilon, \mu, T_0)$ and $C^{\prime}=C^{\prime}(\epsilon, T_0)$ such that for all $s\in [\max\{s^{\prime\prime}, \mathbf{s}_0\},\infty)$,
\begin{align*}
\sum_{n\in \ZZ}\mathbb{P} \big( (\widetilde{E}^c_{n-1}\cap \widetilde{\mathcal{E}}_{n-1}) \cap (\widetilde{E}^c_{n+1}\cap \widetilde{\mathcal{E}}_{n+1}) \cap \widetilde{F}_n\big)
&\leq \begin{cases}
  C^{\prime} e^{-\frac{\sqrt{2}}{3}(1-\epsilon)(1-\mu) s^{3/2}} & \text{ if } s\in [\mathbf{s}_0, \mathbf{s}_1]\cup (\mathbf{s}_2, \infty),\\
 C^{\prime} e^{-\frac{\sqrt{2}}{3}\epsilon(1-\mu) s^{3/2}} & \text{ if } s\in (\mathbf{s}_1,\mathbf{s}_2].
  \end{cases}
\end{align*}
Combining this with \eqref{eq:SumRes}, we arrive at \eqref{eq:NoBigMaxRes}.
 \end{proof}

   \smallskip

\noindent\textsc{Final step of the proof of Proposition~\ref{SubstituteTheo}:}  Define $s^{\prime}_0:= \max\{\bar{s}, s^{\prime}, s^{\prime\prime}\}$ where $\bar{s},s^{\prime},s^{\prime\prime}$ are taken from Lemmas~\ref{UpSumProbBd}, \ref{SumProbBdLem} and \ref{BigMaxApp} respectively.
\be
\ii Owing to \eqref{eq:MyNewEq} and \eqref{eq:NoBigMaxRes}, when $T_0>\pi$, there exists $\Theta=\Theta(\epsilon, T_0)$ such that for all $s\in [\max\{s^{\prime}_0,\mathbf{s}_0 \},\infty)$
\begin{align}\label{eq:OneCompBd}
\mathbb{P}\left(\widetilde{\mathcal{A}}^{f}\cap \left(\bigcup_{n\in \ZZ} \widetilde{E}_n\right)^c\right)\leq \begin{cases}
  \Theta e^{-\frac{\sqrt{2}}{3}(1-\epsilon)(1-\mu) s^{3/2}} & \text{ when } s\in [\mathbf{s}_0, \mathbf{s}_1]\cup (\mathbf{s}_2, \infty),\\
 \Theta e^{-\frac{\sqrt{2}}{3}\epsilon(1-\mu) s^{3/2}} & \text{ when } s\in (\mathbf{s}_1,\mathbf{s}_2].
  \end{cases}
\end{align}
Plugging \eqref{eq:OneCompBd} and \eqref{eq:TotSumBd} of Lemma~\ref{UpSumProbBd} into the r.h.s. of \eqref{eq:BasicStep3} yields \eqref{eq:ResultBd}.

\ii When $T_0\in (0,\pi)$, the proof of \eqref{eq:GenRoughUpBd} follows in the same way as in the proof of \eqref{eq:ResultBd} by assuming $\mathbb{P}(\Upsilon_T(0)>s)\leq e^{-cs^{3/2}}$ for all $s\geq s_0$ and $T\in [T_0,\pi]$.
\ee

 \subsubsection{Proof of Proposition~\ref{UpTailLowBd}}\label{UpTailLowBdSEC}

   Let $\mathcal{I}$ be a subinterval of $[-M,M]$ with $|\mathcal{I}|=\theta$ such that $f(y)\geq -\kappa$ for all $y\in \mathcal{I}$. Assume $s$ is large enough such that $s^{-n+2}\leq \theta$. Let $\chi_1\leq \chi_2\in \mathcal{I}$ be such that $\chi_2-\chi_1= s^{-n+2}$. Define
  \begin{align}\label{eq:NewEvent}
  \mathcal{W}_{i} &:= \Big\{ \Upsilon_T(- \chi_{i})\geq - \frac{\chi^2_i}{2^{2/3}} + \big(1+\tfrac{2\mu}{3}\big)s\Big\} \quad \text{for }i=1,2, \\ \mathcal{W}_{\mathrm{int}} &:= \Big\{\Upsilon_T(y)\geq - \frac{y^2}{2^{2/3}} +\big(1+\tfrac{\mu}{3}\big)s \text{ for all }y\in (-\chi_2, -\chi_1)\Big\}.
  \end{align}
We claim that there exists $s^{\prime}= s^{\prime}(\mu, \theta ,\kappa, T_0)$ such that for all $s\geq s^{\prime}$ and $T\geq T_0$
 \begin{align}\label{eq:LowerEvent}
 \mathbb{P}(\mathcal{W}_{1} \cap \mathcal{W}_{2}\cap \mathcal{W}_{\mathrm{int}}) \leq \mathbb{P}(h^{f}_T(0)\geq s).
 \end{align}
To  show this, assume that the event $\mathcal{W}_{1}\cap \mathcal{W}_{2}\cap \mathcal{W}_{\mathrm{int}} $ occurs. Then\footnote{We use below $-\mathcal{I}$ as a shorthand notation for $\{x: -x\in \mathcal{I}\}$},
  \begin{align*}
  \int^{\infty}_{-\infty} e^{T^{1/3}(\Upsilon_T(y)+ f(-y))} dy \geq \int_{-\mathcal{I}}e^{T^{1/3}(\Upsilon_T(y)+ f(-y))}dy \geq 2\theta e^{T^{1/3}((1+\mu/3)s-\kappa)} \geq e^{T^{1/3}s}
  \end{align*}
  where the last inequality holds when $s$ exceeds some $s^{\prime}(\mu, \theta ,\kappa, T_0)$. This shows that
  \[\mathbb{P}\big(\mathcal{W}_{1} \cap \mathcal{W}_{2}\cap \mathcal{W}_{\mathrm{int}}\big)\leq \mathbb{P}\Big(\int^{\infty}_{-\infty} e^{T^{1/3}(\Upsilon_T(y)+ f(y))} dy\geq e^{T^{1/3}s}\Big)= \mathbb{P}(h^{f}_T(0)\geq s).\]

 To finish the proof of \eqref{eq:UpTailLowBd} we combine \eqref{eq:LowerEvent} with \eqref{eq:AcLowBd} below and take $s_0= \max\{s^{\prime}, s^{\prime\prime}\}$.

 \begin{claim} There exist $s^{\prime\prime} = s^{\prime\prime}(\mu, n, T_0)$, $K=K(\mu)>0$ such that for all $s\geq s^{\prime\prime}$ and $T\geq T_0$,
  \begin{align}\label{eq:AcLowBd}
  \mathbb{P}\left(\mathcal{W}_{1}\cap \mathcal{W}_{2} \cap \mathcal{W}_{\mathrm{int}}\right)\geq \Big(\mathbb{P}\big(\Upsilon_T(0)>\big(1+\tfrac{2\mu}{3}\big)s\big)\Big)^2 - e^{-Ks^n}.
  \end{align}
\end{claim}
\begin{proof}
  We start by writing $\mathbb{P}\big(\mathcal{W}_{1} \cap \mathcal{W}_{2}\cap \mathcal{W}_{\mathrm{int}}\big) = \mathbb{P}(\mathcal{W}_{1}\cap \mathcal{W}_{2}) - \mathbb{P}(\mathcal{W}_{1}\cap \mathcal{W}_{2} \cap \mathcal{W}^c_{\mathrm{int}})$.
   Using the FKG inequality from \eqref{eq:RevFKG}, \begin{align}\label{eq:TwoPtLowBd}
   \mathbb{P}(\mathcal{W}_{1}\cap \mathcal{W}_{2})\geq \mathbb{P}(\mathcal{W}_{1})\mathbb{P}(\mathcal{W}_{2})\geq \Big(\mathbb{P}\big(\Upsilon_T(0)>\big(1+\tfrac{2\mu}{3}\big)s\big)\Big)^2
   \end{align}
   where the last inequality follows from Proposition~\ref{StationarityProp}. Note that \eqref{eq:TwoPtLowBd} provides a lower bound for the first term on the r.h.s. of \eqref{eq:AcLowBd}.
   To complete the proof, we need to demonstrate an upper bound on $\mathbb{P}(\mathcal{W}_{1} \cap \mathcal{W}_{2}\cap \mathcal{W}^c_{\mathrm{int}})$ of the form $e^{-Ks^n}$. To achieve this we go to the KPZ line ensemble and use its Brownian Gibbs property. We may replace $\Upsilon_T$ by $\Upsilon^{(1)}_T$ in all definitions without changing the value of $\mathbb{P}(\mathcal{W}_{1} \cap \mathcal{W}_{2}\cap \mathcal{W}^c_{\mathrm{int}})$ (see Proposition~\ref{NWtoLineEnsemble}). Let us define
\begin{align}
\mathbb{P}_{\mathbf{H}_{2T}}&:=\mathbb{P}^{1,1, (-\chi_2, -\chi_1), 2^{-\frac{1}{3}}\Upsilon^{(1)}_T(-\chi_2), 2^{-\frac{1}{3}}\Upsilon^{(1)}_T(-\chi_1), +\infty, 2^{-\frac{1}{3}}\Upsilon^{(2)}_T}_{\mathbf{H}_{2T}},\\
\widetilde{\mathbb{P}}_{\mathbf{H}_{2T}}&:=\mathbb{P}^{1,1, (-\chi_2, -\chi_1), 2^{-\frac{1}{3}}\Upsilon^{(1)}_T(-\chi_2), 2^{-\frac{1}{3}}\Upsilon^{(1)}_T(-\chi_1), +\infty, -\infty}_{\mathbf{H}_{2T}}.
\end{align} Using the $\mathbf{H}_{2T}$-Brownian Gibbs property of the KPZ line ensemble $\{2^{-\frac{1}{3}}\Upsilon^{(n)}_T(x)\}_{n\in \NN,x\in \RR}$,
 \begin{align}
 \mathbb{P}\big(\mathcal{W}_{1}\cap \mathcal{W}_{2}\cap \mathcal{W}^{c}_{\mathrm{int}}\big) = \mathbb{E}\big[\mathbbm{1}(\mathcal{W}_{1}\cap \mathcal{W}_{2})\cdot \mathbb{P}_{\mathbf{H}_{2T}} (\mathcal{W}^c_{\mathrm{int}})\big].\label{eq:CondInt}
\end{align}

Via Proposition~\ref{Coupling1}, there exists a monotone coupling between $\mathbb{P}_{\mathbf{H}_{2T}}$ and $\widetilde{\mathbb{P}}_{\mathbf{H}_{2T}}$ so that
   \begin{align}\label{eq:CouplingIneq}
   \mathbb{P}_{\mathbf{H}_{2T}} (\mathcal{W}^c_{\mathrm{int}})\leq \widetilde{\mathbb{P}}_{\mathbf{H}_{2T}} (\mathcal{W}^c_{\mathrm{int}}).
   \end{align}
   Recall that $\widetilde{\mathbb{P}}_{\mathbf{H}_{2T}}$ is the measure of a Brownian bridge on  $(-\chi_2, -\chi_1)$ with starting and end points at $2^{-\frac{1}{3}}\Upsilon^{(1)}_T(-\chi_2)$ and $2^{-\frac{1}{3}}\Upsilon^{(1)}_T(-\chi_1)$.
 Applying \eqref{eq:CouplingIneq} into the r.h.s. of \eqref{eq:CondInt} implies
 \begin{align}
 \mathbbm{1}(\mathcal{W}_{1}\cap \mathcal{W}_{2}) \cdot \widetilde{\mathbb{P}}_{\mathbf{H}_{2T}} (\mathcal{W}^c_{\mathrm{int}})\leq \mathbb{P}^{1,1, (-\chi_2, -\chi_1), -\frac{\chi^2_2}{2}+2^{-\frac{1}{3}}\big(1+\tfrac{2\mu}{3}\big)s,  -\frac{\chi^2_1}{2}+2^{-\frac{1}{3}}\big(1+\tfrac{2\mu}{3}\big)s}_{\mathrm{free}} \big(\mathcal{W}^c_{\mathrm{int}}\big).
 \end{align}
 Therefore (using Lemma~\ref{BBFlucLem} for the second inequality) there exists $K=K(\mu)$ such that
 \begin{align*}
 \text{l.h.s. of \eqref{eq:CondInt}}\leq \mathbb{P}^{1,1, (-\chi_2, -\chi_1), -\frac{\chi^2_2}{2}+2^{-\frac{1}{3}}\big(1+\tfrac{2\mu}{3}\big)s,  -\frac{\chi^2_1}{2}+2^{-\frac{1}{3}}\big(1+\tfrac{2\mu}{3}\big)s}_{\mathrm{free}}(\mathcal{W}^c_{\mathrm{int}})\leq e^{-Ks^{n}}.
 \end{align*}

 \end{proof}


 \subsection{Proof of Theorem~\ref{Main6Theorem}}\label{Proof5Theorem}

Theorem~\ref{Main6Theorem} follows by combining all three parts of Theorem~\ref{GrandUpTheorem} with the following results which are in the same spirit of Proposition~\ref{SubstituteTheo} and~\ref{UpTailLowBd} respectively.

Recall $\Upsilon_T$ and $h^{\mathrm{Br}}_T$ from \eqref{eq:DefUpsilon} and \eqref{eq:h_BrDefine} respectively.

\bp\label{BrUpTailSubstituteTheo}
Fix $\epsilon, \mu\in (0,\frac{1}{2})$.
\be
\ii Fix $T_0>\pi$. Suppose there exists $s_0 =s_0(\epsilon, T_0)$ and for any $T\geq T_0$, there exist $s_2=s_2(\epsilon, T)$ and $s_3=s_3(\epsilon, T)$ with $s_1\leq s_2\leq s_3$ such that  for any $s\in [s_0, \infty)$,
\begin{align}\label{eq:AssumeNW}
\mathbb{P}\big(\Upsilon_T(0)>s\big)\leq \begin{cases} e^{-\frac{4}{3}(1-\epsilon)s^{\frac{3}{2}}} & \text{ if } s\in [s_0, s_1]\cup (s_2,\infty),\\
 e^{-\frac{4}{3}\epsilon s^{\frac{3}{2}}} & \text{ if } s\in (s_1, s_2].
\end{cases}
\end{align}
Then, there exists $s^{\prime}_0 = s^{\prime}_0(\epsilon, \mu, T_0)$ such that for any $T>T_0$ and $s\in [\max\{s^{\prime}_0, \mathbf{s}_0\},\infty)$, we have (recall $\mathbf{s}_0, \mathbf{s}_1$ and $\mathbf{s}_2$ from \eqref{eq:mathbfs})
\begin{align}\label{eq:BrUpTailUpLow}
\mathbb{P}\big(h^{\mathrm{Br}}_T(0)>s\big)\leq
\begin{cases}
e^{-\frac{\sqrt{2}}{3}(1-\epsilon)(1-\mu)s^{3/2}}+ e^{-\frac{1}{9\sqrt{3}} (\mu s)^{3/2}}& \text{ if } s\in [\mathbf{s}_0, \mathbf{s}_1]\cup (\mathbf{s}_2,\infty),\\
e^{-\frac{\sqrt{2}}{3}\epsilon(1-\mu)s^{3/2}} +e^{-\frac{1}{9\sqrt{3}} (\mu s)^{3/2}}& \text{ if } s\in (\mathbf{s}_1, \mathbf{s}_2].
\end{cases}
\end{align}
\ii For any $T_0\in (0,\pi)$, there exists $s^{\prime}_0= s^{\prime}_0(T_0)>0$ satisfying the following: if there exists $s_0 = s_0(T_0)>0$ such that $\mathbb{P}(\Upsilon_T(0)>s)\leq e^{-cs^{3/2}}$ for all $s\geq s_0$ and $T\in [T_0, \pi]$, then,
\begin{align}\label{eq:BrUpRoughBd}
\mathbb{P}\big(h^{f}_{T}(0)>s\big)\leq e^{-cs^{3/2}}, \quad \forall s\in [\max\{s^{\prime}_0,s_0 \}, \infty), T \in [T_0, \pi].
\end{align}
\ee
\ep

\bp\label{BrUpTailLowBdProp}
Fix $\mu\in (0,\frac{1}{2})$, $n\in \ZZ_{\geq 3}$ and $T_0>\pi$. Then, there exist $s_0=s_0 (\mu,n, T_0), K=K(\mu, n)>0$ such that for all $s\geq s_0$ and $T\geq T_0$,
\begin{align}\label{eq:BrLowTail}
\mathbb{P}(h^{\mathrm{Br}}_T(0)> s)\geq \Big(\mathbb{P}\Big(\Upsilon_T(0)>\big(1+\frac{2\mu}{3}\big)s\Big)\Big)^2 - e^{-Ks^n}.
\end{align}
\ep

We prove these propositions using similar arguments  as in Section~\ref{UpTailUpBdSEC} and~\ref{UpTailLowBdSEC}. Propositions ~\ref{BrUpTailSubstituteTheo} and \ref{BrUpTailLowBdProp} are proved in Sections~\ref{BrUpTailUpBd} and Section~\ref{BrUpTailLowBd}, respectively.

  \smallskip

\begin{proof}[Proof of Theorem~\ref{Main6Theorem}]
This theorem is proved in the same way as Theorem~\ref{Main4Theorem} by combining Proposition~\ref{BrUpTailSubstituteTheo} and Proposition~\ref{BrUpTailLowBdProp}. We do not duplicate the details.
\end{proof}

\subsubsection{Proof of Proposition~\ref{BrUpTailSubstituteTheo}}\label{BrUpTailUpBd}

To prove this proposition, we use similar arguments as in Section~\ref{UpTailLowBdSEC}. Let $\tau\in (0,\frac{1}{2})$ be fixed (later we choose its value). Recall the events $\widetilde{E}_n$ and $\widetilde{F}_n$ from Section~\ref{UpTailUpBdSEC} and define
\begin{align}\label{eq:UpBrEvent}
\widetilde{\mathcal{A}}^{\mathrm{Br}}:= \left\{\int^{\infty}_{-\infty} e^{T^{1/3}\big(\Upsilon_T(y)+ B(-y)\big)} dy> e^{sT^{1/3}}\right\}
\end{align}
where $B$ is a two sided Brownian motion with diffusion coefficient $2^{\frac{1}{3}}$ and $B(0)=0$. Appealing to Proposition~\ref{Distribution}, we see that $\mathbb{P}(h^{\mathrm{Br}}_T(0)>s)= \mathbb{P(\widetilde{\mathcal{A}}^{\mathrm{Br}}})$. Now, we write
\begin{align}\label{eq:BrSplit}
\mathbb{P}\big(\widetilde{\mathcal{A}}^{\mathrm{Br}}\big)\leq \sum_{n\in \ZZ}\mathbb{P}\big( \widetilde{E}_n\big) + \mathbb{P}\Big(\widetilde{\mathcal{A}}^{\mathrm{Br}}\cap \big(\bigcup_{n\in \ZZ} \widetilde{E}_n\big)^{c}\cap \big(\bigcup_{n\in \ZZ} \widetilde{F}_n\big)\Big) + \mathbb{P}\Big(\widetilde{\mathcal{A}}^{\mathrm{Br}}\cap \big(\bigcup_{n\in \ZZ} \widetilde{E}_n\big)^{c}\cap \big(\bigcup_{n\in \ZZ} \widetilde{F}_n\big)^c\Big).
\end{align}
 Using Lemma~\ref{UpSumProbBd} (see \eqref{eq:TotSumBd}) and Lemma~\ref{BigMaxApp} (see \eqref{eq:NoBigMaxRes}) we can bound the first two terms on the right side hand side of \eqref{eq:BrSplit}.  However, unlike in Theorem~\ref{SubstituteTheo}, the last term in \eqref{eq:BrSplit} is not zero. We now provide an upper bound to this term.
   \smallskip

\begin{claim} There exists $s^{\prime}=s^{\prime}(\tau, \mu)$ such that for all $s\geq s^{\prime}$,
\begin{align}\label{eq:BrwLEv}
\mathbb{P}\Big(\widetilde{\mathcal{A}}^{\mathrm{Br}}\cap \big(\bigcup_{n\in \ZZ} \widetilde{E}_n\big)^{c}\cap \big(\bigcup_{n\in \ZZ} \widetilde{F}_n\big)^c\Big) \leq  \exp\left(-\tfrac{\sqrt{(1-2\tau)}}{3\sqrt{6}}\Big(\tfrac{2\mu s}{3}+\log\big((2\pi)^{-1} \tau (2T)^{\frac{1}{3}}\big)\Big)^{\frac{3}{2}} \right). &&
\end{align}\end{claim}

\begin{proof}
Note that
 \begin{align}\label{eq:EveInc}
 \Big\{\widetilde{\mathcal{A}}^{\mathrm{Br}}\cap \big(\bigcup_{n\in \ZZ} \widetilde{E}_n\big)^{c}\cap \big(\bigcup_{n\in \ZZ} \widetilde{F}_n\big)^c\Big\}& \subseteq \Big\{\int^{\infty}_{-\infty}e^{T^{1/3}\big(-\frac{(1-\tau)y^2}{2^{2/3}}+ B(-y)\big)} dy \geq e^{3^{-1}\mu s T^{1/3}}\Big\}.
 \end{align}
 We claim that
 \begin{align}\label{eq:Cont4}
 \text{r.h.s. of \eqref{eq:EveInc}}\subseteq  \Big\{\max_{y\in \RR}\Big\{ - \frac{(1-2\tau)y^2+ B(-y)}{2^{2/3}}\Big\}\geq \tfrac{1}{3}\mu s +\tfrac{1}{2}\log((2\pi)^{-1}\tau (2T)^{\frac{1}{3}}) \Big\}.
 \end{align}
 To see this by contradiction, assume the complement of the r.h.s. of \eqref{eq:Cont4}. This implies that
 \begin{align}
 \int^{\infty}_{-\infty}e^{T^{1/3}\Big(-\frac{(1-\tau)y^2}{2^{2/3}}+ B(-y)\Big)} dy< \sqrt{(2\pi)^{-1} \tau (2T)^{1/3}}e^{3^{-1}\mu s T^{1/3}}\int^{\infty}_{-\infty}e^{-\tau y^2T^{1/3}/2^{\frac{2}{3}}} dy= e^{3^{-1}\mu s T^{1/3}}.
 \end{align}
 Therefore, \eqref{eq:Cont4} holds. Applying Proposition~\ref{BMminusParabola} (with $\xi=\frac{1}{2}$), we see that
  \begin{align*}
  \mathbb{P}\big(\textrm{r.h.s. of }\eqref{eq:Cont4}\big)
  \leq \frac{1}{\sqrt{3}} \exp\bigg(-\frac{\sqrt{(1-2\tau)}}{3\sqrt{6}}\Big(\frac{2\mu s}{3}+\log((2\pi)^{-1} \tau (2T)^{\frac{1}{3}})\Big)^{\frac{3}{2}} \bigg).
  \end{align*}
  when $s$ is large enough.
  Combining this with \eqref{eq:EveInc}, we arrive at \eqref{eq:BrwLEv} showing the claim.
\end{proof}

Now, we turn to complete the proof of Proposition~\ref{BrUpTailSubstituteTheo}. Choosing $\tau = \frac{1}{8}$,  we notice
\[ \text{r.h.s. of \eqref{eq:BrwLEv}} \leq \exp\Big(-\frac{1}{9\sqrt{3}}\Big(\mu s-\frac{3\log(16\pi)}{2}\Big)^{3/2}\Big), \quad  \forall\, T>\pi.\]
For the rest of this proof, we will fix some $T\geq T_0$ and assume that there exist $s_0=s_0(\epsilon, T_0)$, $s_1=s_1(\epsilon, T)$ and $s_2=s_2(\epsilon, T)$ with $s_1\leq s_2$ such that \eqref{eq:AssumeNW} is satisfied for all $s\in [s_0,\infty)$. Owing to \eqref{eq:TotSumBd} of Lemma~\ref{UpSumProbBd} and \eqref{eq:NoBigMaxRes} of Lemma~\ref{BigMaxApp}, there exist $\Theta= \Theta(\epsilon, T_0)$ and $\tilde{s}=\tilde{s}(\epsilon, \mu, T_0)$ such that for all $s\in [\max\{\tilde{s},\mathbf{s}_0)\},\infty)$,
\begin{align*}
\mathbb{P}\Big(\bigcup_{n\in \ZZ} \widetilde{E}_n\Big) +
\mathbb{P}\Big(\widetilde{\mathcal{A}}^{\mathrm{Br}}\cap \big(\bigcup_{n\in \ZZ} \widetilde{E}_n\big)^c \cap \big(\bigcup_{n\in \ZZ} \widetilde{F}_n\big)\Big)
\leq \begin{cases}
 \Theta e^{- \frac{\sqrt{2}}{3}(1-\epsilon)(1-\mu)s^{3/2}} & \text{ if }s\in [\mathbf{s}_0,\mathbf{s}_1 ]\cup (\mathbf{s}_2, \infty),\\
 \Theta e^{- \frac{\sqrt{2}}{3}\epsilon(1-\mu)s^{3/2}} & \text{ if }s\in (\mathbf{s}_1, \mathbf{s}_2].
 \end{cases}
\end{align*}
Combining this with \eqref{eq:BrwLEv} and plugging into \eqref{eq:BrSplit}, we get \eqref{eq:BrUpTailUpLow} for all $T>T_0\geq \pi$. In the case when $T_0\in (0,\pi)$, we obtain \eqref{eq:BrUpRoughBd} in a similar way as in the proof of \eqref{eq:GenRoughUpBd} of Proposition~\ref{SubstituteTheo} by combining the inequality of the above display with $\mathbb{P}(\Upsilon_T(0))\leq e^{-cs^{3/2}}$ for all $s\geq s_0$ and $T\in [T_0,\pi].$

 \subsubsection{Proof of Proposition~\ref{BrUpTailLowBdProp}}\label{BrUpTailLowBd}

 We use similar argument as in Proposition~\ref{UpTailLowBd}. The main difference from the proof of Proposition~\ref{UpTailLowBd} is that we do not expect \eqref{eq:LowerEvent} to hold because the initial data is now a two sided Brownian motion, hence, \eqref{eq:LowInitBd} of Definition~\ref{Hypothesis} is not satisfied. However, it holds with high probability which follows from the following simple consequence of the reflection principle for $B$ (a two-sided Brownian motion with diffusion coefficient $2^{\frac{1}{3}}$ and $B(0)=0$)
\begin{align}\label{eq:BTail}
\mathbb{P}\big(\mathcal{M}_s\big)\leq e^{-\frac{\mu^2 }{36}s^{n}},\qquad \textrm{where} \quad \mathcal{M}_s =\Big\{\min_{y\in [-s^{-n+2}, s^{-n+2}]} B(t) \leq -\tfrac{\mu}{6} s\Big\}.
\end{align}
To complete the proof, let us define:
\begin{align}\label{eq:TildeW}
\widetilde{\mathcal{W}}_{\pm} &:= \left\{\Upsilon_T(\pm s^{-n+2})\geq  - \frac{1}{2^{2/3}s^{2(n-2)}}+\big(1+\tfrac{2\mu}{3}\big)\right\},\\\widetilde{\mathcal{W}}_{\mathrm{int}}&:= \left\{\Upsilon_T(y)\geq -\frac{y^2}{2^{2/3}}+ \big(1+\tfrac{\mu}{3}\big)s, \quad \forall y\in [-s^{-n+2}, s^{-n+2}]\right\}.
\end{align}
We claim that there exists $s^{\prime} = s^{\prime}(\mu,n, T_0)$ such that for all $s\geq s^{\prime}$ and $T\geq T_0$,
\begin{align}\label{eq:WInsertion}
\mathbb{P}\big(h^{\mathrm{Br}}_T(0)>s\big)\geq \mathbb{P}\big(\widetilde{\mathcal{W}}_{+}\cap \widetilde{\mathcal{W}}_{-}\cap \widetilde{\mathcal{W}}_{\mathrm{int}}\big) - e^{-\frac{\mu^2 }{36} s^{n}}.
\end{align}
To see this, assume $ \widetilde{\mathcal{W}}_{+}\cap \widetilde{\mathcal{W}}_{-}\cap \widetilde{\mathcal{W}}_{\mathrm{int}}\cap \mathcal{M}_s$ occurs.
Then, for $s$ large enough,
\begin{align}\label{eq:IntDom}
\int^{\infty}_{-\infty}e^{T^{1/3}\big(\Upsilon_T(y)+B(-y)\big)}dy\geq \int^{s^{-n+2}}_{-s^{-n+2}} e^{T^{1/3}\big(-\frac{1}{2^{2/3}s^{2(n-2)}}+(1+\frac{\mu}{6})s\big)} dy> e^{sT^{1/3}}.
\end{align}
By Proposition~\ref{Distribution}, the event $\{$l.h.s. of \eqref{eq:IntDom} $\geq$ r.h.s. of \eqref{eq:IntDom}$\}$ equals $\{h^{\mathrm{Br}}_T(0)>s\}$. Therefore (using \eqref{eq:BTail} for the second inequality) we arrive at the claimed \eqref{eq:WInsertion} via
\begin{align*}
 \mathbb{P}\Big(\{h^{\mathrm{Br}}_T(0)>s\}\Big) \geq \mathbb{P}\big(\widetilde{\mathcal{W}}_{+}\cap \widetilde{\mathcal{W}}_{-}\cap \widetilde{\mathcal{W}}_{\mathrm{int}}\cap \mathcal{M}_s\big)\geq \mathbb{P}\Big(\widetilde{\mathcal{W}}_{+}\cap \widetilde{\mathcal{W}}_{-}\cap \widetilde{\mathcal{W}}_{\mathrm{int}}\Big)- e^{-\frac{\mu^2 }{36}s^n}.
\end{align*}

To finish the proof of Proposition~\ref{BrUpTailLowBdProp} we use a similar argument as used to prove \eqref{eq:AcLowBd}. For any $n\in \ZZ_{\geq 3}$, there exists $s^{\prime\prime}=s^{\prime\prime}(\mu, n, T_0)$ such that for all $s\geq s^{\prime\prime}$ and $T\geq T_0$,
\begin{align*}
\mathbb{P}\big(\widetilde{\mathcal{W}}_{+}\cap \widetilde{\mathcal{W}}_{-}\cap \widetilde{W}_{\mathrm{int}}\big)\geq \Big(\mathbb{P}\big(\Upsilon_T(0)>\big(1+\tfrac{2\mu}{3}\big)s\big)\Big)^2 - e^{- Ks^n}.
\end{align*}
Combining this with \eqref{eq:WInsertion} and taking $s_0= \max\{s^{\prime},s^{\prime\prime}\} $, we arrive at \eqref{eq:BrLowTail} for all $s\geq s_0$.

 \bibliographystyle{alpha}
\bibliography{Reference}

\end{document}